\documentclass[a4paper]{amsart}
\pdfoutput=1
\usepackage[utf8]{inputenc}
\usepackage[T1]{fontenc}
\usepackage{lmodern}
\usepackage{amsxtra,amssymb}
\usepackage{mathtools,enumitem,dsfont}
\usepackage{stackengine}
\usepackage{microtype}

\usepackage{tikz,tikz-cd}
\usetikzlibrary{arrows}
\tikzset{cd/.style=matrix of math nodes} 
\tikzset{cdar/.style=->,auto}

\tikzset{dar/.style={double,double equal sign distance,-implies}}

\usepackage[colorlinks=true,linkcolor=blue, citecolor=blue,
urlcolor=blue, raiselinks=true,
pdftitle={Groupoid models for relative Cuntz--Pimsner algebras of groupoid correspondences},
pdfauthor={Ralf Meyer},
pdfsubject={Mathematics}
]{hyperref}

\setlist[enumerate,1]{label=\textup{(\arabic*)}}
\setlist[enumerate,2]{label=\textup{(\alph*)}}
\usepackage{xcolor}

\usepackage[lite]{amsrefs}

\renewcommand*{\PrintDOI}[1]{\href{http://dx.doi.org/\detokenize{#1}}{doi:
\detokenize{#1}}}

\BibSpec{book}{%
  +{}  {\PrintPrimary}                {transition}
  +{,} { \textit}                     {title}
  +{.} { }                            {part}
  +{:} { \textit}                     {subtitle}
  +{,} { \PrintEdition}               {edition}
  +{}  { \PrintEditorsB}              {editor}
  +{,} { \PrintTranslatorsC}          {translator}
  +{,} { \PrintContributions}         {contribution}
  +{,} { }                            {series}
  +{,} { \voltext}                    {volume}
  +{,} { }                            {publisher}
  +{,} { }                            {organization}
  +{,} { }                            {address}
  +{,} { \PrintDateB}                 {date}
  +{,} { }                            {status}
  +{}  { \parenthesize}               {language}
  +{}  { \PrintTranslation}           {translation}
  +{;} { \PrintReprint}               {reprint}
  +{.} { }                            {note}
  +{.} {}                             {transition}
  +{} { \PrintDOI}                   {doi}
  +{} { available at \url}            {eprint}
  +{}  {\SentenceSpace \PrintReviews} {review}
}

\BibSpec{thesis}{%
  +{}  {\PrintAuthors}                {author}
  +{,} { \textit}                     {title}
  +{:} { \textit}                     {subtitle}
  +{,} { \PrintThesisType}            {type}
  +{,} { }                            {organization}
  +{,} { }                            {address}
  +{,} { \PrintDateB}                 {date}
  +{,} { \eprint}                     {eprint}
  +{,} { }                            {status}
  +{}  { \parenthesize}               {language}
  +{}  { \PrintTranslation}           {translation}
  +{;} { \PrintReprint}               {reprint}
  +{.} { }                            {note}
  +{.} {}                             {transition}
  +{} { \PrintDOI}                   {doi}
  +{} { available at \url}            {eprint}
  +{}  {\SentenceSpace \PrintReviews} {review}
}

\usepackage{amsthm}
\newcommand{\longref}[2]{\hyperref[#2]{#1~\textup{\ref*{#2}}}}

\numberwithin{equation}{section}
\theoremstyle{plain}
\newtheorem{theorem}[equation]{Theorem}
\newtheorem{lemma}[equation]{Lemma}
\newtheorem{proposition}[equation]{Proposition}

\theoremstyle{definition}
\newtheorem{definition}[equation]{Definition}

\theoremstyle{remark}
\newtheorem{remark}[equation]{Remark}
\newtheorem{example}[equation]{Example}

\newcommand*{\alb}{\hspace{0pt}} 
\newcommand*{\nb}{\nobreakdash}

\newcommand*{\Mult}{\mathcal M}
\newcommand*{\Slice}[1][U]{\mathcal{#1}}
\newcommand*{\Rep}{T}

\newcommand{\Hilm}{\mathcal{E}}
\newcommand{\HilmF}{\mathcal{F}}

\newcommand*{\ev}{\mathrm{ev}}

\newcommand*{\Id}{\mathrm{id}}

\newcommand*{\Qu}{\mathsf{p}}
\newcommand*{\s}{s} 
\newcommand*{\rg}{r}

\newcommand{\Gr}[1][G]{\mathcal #1}
\newcommand*{\Bisp}[1][X]{\mathcal #1}
\newcommand{\Reg}{\mathcal R}

\newcommand{\Mod}{\mathcal M}

\newcommand{\into}{\rightarrowtail}
\newcommand{\prto}{\twoheadrightarrow}

\newcommand*{\ContS}{\mathfrak S}
\newcommand*{\Cont}{\mathrm C}
\newcommand*{\Contb}{\mathrm{C_b}}
\newcommand*{\Contc}{\mathrm{C_c}}

\newcommand*{\Cst}{\texorpdfstring{\textup C^*}{C*}}
\newcommand*{\Star}{$^*$\nobreakdash-}

\newcommand*{\defeq}{\mathrel{\vcentcolon=}}
\newcommand*{\congto}{\xrightarrow\sim}

\newcommand*{\pr}{\mathrm{pr}}

\newcommand{\C}{\mathbb{C}}
\newcommand{\N}{\mathbb{N}}

\newcommand{\Comp}{\mathbb{K}}
\newcommand{\Bound}{\mathbb{B}}

\newcommand{\Bis}{\mathcal{B}}
\newcommand{\IS}{\mathcal{I}}

\newcommand{\Grcomp}{\circ}

\newcommand{\blank}{{\llcorner\!\lrcorner}}
\newcommand*{\conj}[1]{\overline{#1}}

\DeclareMathOperator{\Dom}{Dom}

\DeclarePairedDelimiter{\abs}{\lvert}{\rvert}
\DeclarePairedDelimiter{\ket}{\lvert}{\rangle}
\DeclarePairedDelimiter{\bra}{\langle}{\rvert}
\DeclarePairedDelimiterX{\braket}[2]{\langle}{\rangle}{#1\,\delimsize\vert\,\mathopen{}#2}
\DeclarePairedDelimiterX{\BRAKET}[2]{\langle\!\langle}{\rangle\!\rangle}{#1\,\delimsize\vert\,\mathopen{}#2}
\DeclarePairedDelimiterX{\setgiven}[2]{\{}{\}}{#1\,{:}\,\mathopen{}#2}

\begin{document}
\title[Groupoid models for relative Cuntz--Pimsner algebras]{Groupoid models for relative Cuntz--Pimsner algebras of groupoid correspondences}
\author{Ralf Meyer}
\email{rmeyer2@uni-goettingen.de}
\address{Mathematisches Institut\\
  Universität Göttingen\\
  Bunsenstraße 3--5\\
  37073 Göttingen\\
  Germany}

\keywords{étale groupoid; groupoid correspondence; bicategory;
  Cuntz--Pimsner algebra; Toeplitz algebra; self-similar group;
  self-similar graph}

\begin{abstract}
  A groupoid correspondence on an \'etale, locally compact groupoid
  induces a \(\Cst\)\nb-correspondence on its groupoid
  \(\Cst\)\nb-algebra.  We show that the Cuntz--Pimsner algebra for
  this \(\Cst\)\nb-correspondence relative to an ideal associated to
  an open invariant subset of the groupoid is again a groupoid
  \(\Cst\)\nb-algebra for a certain groupoid.  We describe this groupoid
  explicitly and characterise it by a universal property that specifies
  its actions on topological spaces.  Our construction unifies the
  construction of groupoids underlying the \(\Cst\)\nb-algebras of
  topological graphs and self-similar groups.
\end{abstract}
\subjclass[2020]{46L55; 22A22}
\thanks{This work is part of the project Graph Algebras partially
  supported by EU grant HORIZON-MSCA-SE-2021 Project 101086394.}

\maketitle

\section{Introduction}
\label{sec:intro}

We show that the Cuntz--Pimsner algebra of the
\(\Cst\)\nb-correspondence defined by a groupoid correspondence is a
groupoid \(\Cst\)\nb-algebra for a canonical groupoid, which is
characterised by a universal property that specifies its actions on
topological spaces.
Our groupoid construction unifies several existing ones, starting with
the groupoids underlying Cuntz algebras
(see~\cite{Renault:Groupoid_Cstar}) and graph \(\Cst\)\nb-algebras
(see~\cite{Kumjian-Pask-Raeburn:Cuntz-Krieger_graphs}), and continuing
with topological graphs (see~\cite{Katsura:class_I}), self-similar
groups (see~\cite{Nekrashevych:Cstar_selfsimilar}), self-similar
graphs (see~\cite{Exel-Pardo:Self-similar}), or self-similar actions
of groupoids
(see~\cite{Laca-Raeburn-Ramagge-Whittaker:Equilibrium_self-similar_groupoid}).
In fact, the Cuntz--Pimsner algebras of groupoid correspondences may
be viewed as being the \(\Cst\)\nb-algebras of self-similar actions of
general \'etale groupoids as opposed to the discrete groupoids studied
in~\cite{Laca-Raeburn-Ramagge-Whittaker:Equilibrium_self-similar_groupoid}.

Let \(\Gr\) and~\(\Gr[H]\) be \'etale, locally compact groupoids,
possibly non-Hausdorff.  A groupoid correspondence
\(\Bisp\colon \Gr[H]\leftarrow \Gr\) is a topological space~\(\Bisp\)
with commuting actions of~\(\Gr[H]\) on the left and~\(\Gr\) on the
right, such that the right action is free and proper and its anchor
map \(\s\colon \Bisp\to \Gr^0\) is a local homeomorphism; this
forces~\(\Bisp\) to be locally Hausdorff and locally quasi-compact.
Such a groupoid correspondence induces a
\(\Cst(\Gr[H])\)-\(\Cst(\Gr)\)-correspondence, that is, a Hilbert
\(\Cst(\Gr)\)\nb-module with a nondegenerate action of
\(\Cst(\Gr[H])\) by adjointable operators
(see~\cite{Antunes-Ko-Meyer:Groupoid_correspondences}).

The anchor map \(\rg\colon \Bisp\to \Gr[H]^0\) of the left action
on~\(\Bisp\) induces a continuous map
\(\rg_*\colon \Bisp/\Gr \to \Gr[H]^0\) because it is
\(\Gr\)\nb-invariant.  Let
\(\Reg\subseteq \Gr[H]^0\) be an open, \(\Gr[H]\)\nb-invariant subset.
Then \(\Gr_\Reg = \s^{-1}(\Reg) = \rg^{-1}(\Reg)\) is an open
subgroupoid of~\(\Gr\) and \(\Cst(\Gr_\Reg)\) is an ideal in
\(\Cst(\Gr)\).  We call~\(\Bisp\) \emph{relatively proper} on~\(\Reg\)
if~\(\rg_*\) restricts to a proper map \(\rg_*^{-1}(\Reg) \to \Reg\).
This implies that \(\Cst(\Gr_\Reg)\) acts by compact operators on
\(\Cst(\Bisp)\).  If \(\Gr=\Gr[H]\), then this allows us to define a
Cuntz--Pimsner algebra \(\mathcal{O}_{\Cst(\Bisp),\Cst(\Gr_\Reg)}\)
for \(\Cst(\Bisp)\) relative to the ideal \(\Cst(\Gr_\Reg)\) in
\(\Cst(\Gr)\).  If~\(\Reg\) is empty, then this is the Toeplitz
\(\Cst\)\nb-algebra of~\(\Cst(\Bisp)\).  We are going to write
\(\mathcal{O}_{\Cst(\Bisp),\Cst(\Gr_\Reg)}\) as the
\(\Cst\)\nb-algebra of an \'etale groupoid, which we call the
\emph{groupoid model} of \((\Gr,\Bisp,\Reg)\).

A groupoid model as above was first constructed in the case
\(\Reg=\Gr^0\) by Albandik~\cite{Albandik:Thesis}.  It is described in a
better way in~\cite{Meyer:Diagrams_models}, and our description will
use the same idea.  These two articles also consider diagrams of
groupoid correspondences, which correspond roughly to higher-rank
graphs.  The generalisation to relative Cuntz--Pimsner algebras has
been worked out by Antunes~\cite{Antunes:Thesis}.  The special
case of a groupoid correspondence of a discrete groupoid is also
treated in~\cite{Miller-Steinberg:Homology_K_self-similar}.  Here we
mostly follow~\cite{Antunes:Thesis}, correcting one important
inaccuracy (see Remark~\ref{rem:inaccuracy}).
Currently, there is no treatment that allows relative Cuntz--Pimsner
algebras of diagrams of groupoid correspondences that fail to be
proper.  Indeed, this theory is more difficult because extra concepts
such as compact alignment become relevant on the \(\Cst\)\nb-algebraic
side.

To define the groupoid model, we first specify what an action of
\((\Gr,\Bisp,\Reg)\) on a topological space is.  The groupoid
model~\(\Mod\) is defined as a groupoid whose actions are in a
natural bijection with actions of \((\Gr,\Bisp,\Reg)\).  We show that
there is a groupoid with this universal property and that it is unique
up to isomorphism.  A key step here is to rewrite an action of
\((\Gr,\Bisp,\Reg)\) as an action of a certain inverse
semigroup~\(\IS\) with certain extra properties.  Thus our groupoid
model arises as an inverse semigroup transformation groupoid
\(\Omega\rtimes \IS\).  We describe \Star{}homomorphisms from both
\(\mathcal{O}_{\Cst(\Bisp),\Cst(\Gr_\Reg)}\) and
\(\Cst(\Omega\rtimes \IS)\) to \(\Comp(\Hilm)\) for a Hilbert
module~\(\Hilm\) over a \(\Cst\)\nb-algebra in terms of the original
diagram.  These universal properties are reasonably close to each
other, and then we do some technical work to identify the data that
appears in both universal properties.  This implies the desired
isomorphism
\(\mathcal{O}_{\Cst(\Bisp),\Cst(\Gr_\Reg)}\cong \Cst(\Omega\rtimes
\IS)\).

The groupoid model may be used to characterise when the relative
Cuntz--Pimsner algebra is simple, by checking whether it is Hausdorff,
effective, amenable, and minimal.  Criteria for this are developed
in~\cite{Antunes:Thesis}.  In ongoing work with de
Castro~\cite{Castro-Meyer:Graph_actors}, we use the
universal property that characterises the groupoid model to study
morphisms between graph \(\Cst\)\nb-algebras through certain morphisms
between the underlying groupoids.

\section{Relatively proper groupoid correspondences}
\label{sec:correspondences}

This article depends
on~\cite{Antunes-Ko-Meyer:Groupoid_correspondences}, which defines
groupoid correspondences and the associated
\(\Cst\)\nb-correspondences.  In addition, a composition of groupoid
correspondences is defined and shown to be compatible with the
composition of \(\Cst\)\nb-correspondences in the sense that the
passage from groupoid to \(\Cst\)\nb-correspondences is a homomorphism
of bicategories.  We shall not repeat the theory developed
in~\cite{Antunes-Ko-Meyer:Groupoid_correspondences}.

Throughout this article, we fix an (\'etale, locally compact)
groupoid~\(\Gr\), an (\'etale, locally compact) groupoid
correspondence \(\Bisp\colon \Gr \leftarrow \Gr\), and an open
\(\Gr\)\nb-invariant subset \(\Reg\subseteq \Gr^0\) such that the map
\(\rg_*\colon \Bisp/\Gr \to \Gr^0\) restricts to a proper map
\(\rg_*^{-1}(\Reg) \to \Reg\); then we briefly call~\(\Bisp\)
\emph{proper on~\(\Reg\)}.  Our notation differs
from~\cite{Antunes:Thesis} because we do \emph{not} assume
\(\Reg \subseteq \rg(\Bisp)\) as in~\cite{Antunes:Thesis}.

The case \(\Reg=\emptyset\) is allowed, and then~\(\Bisp\) may be any
groupoid correspondence.  The assumptions imply that \(\Gr^0\) and the
orbit space~\(\Bisp/\Gr\) of the right \(\Gr\)\nb-action on~\(\Bisp\)
are locally compact Hausdorff spaces.  In contrast, the arrow
space~\(\Gr\) and the space~\(\Bisp\) itself need not be Hausdorff.

We will discuss self-similar groupoid actions as an example in
Section~\ref{sec:examples}.  Throughout the paper, we use the case of
directed graphs and their \(\Cst\)\nb-algebras to illustrate our
theory.

\begin{example}
  \label{exa:graph_as_correspondence}
  Let \(\Gr=V\) be just a discrete set with only identity arrows.  So
  \(\Cst(V) = \Cont_0(V)\).  A correspondence
  \(\Bisp\colon \Gr\leftarrow \Gr\) is just a discrete set~\(E\) with
  two maps \(\rg,\s\colon E\rightrightarrows V\) or, in other words, a
  directed graph (see
  \cite{Antunes-Ko-Meyer:Groupoid_correspondences}*{Example~4.1}).
  Recall that a vertex \(v\in V\) is \emph{regular} if
  \(\rg^{-1}(v) \subseteq E\) is nonempty and finite.  Let
  \(\Reg \subseteq V\) be the subset of all regular vertices.  Then
  the Cuntz--Pimsner algebra of \(\Cst(E)\) relative to the ideal
  \(\Cont_0(\Reg) \subseteq \Cont_0(V)\) is the usual graph
  \(\Cst\)\nb-algebra of the graph
  \(\rg,\s\colon E\rightrightarrows V\) (see, for instance,
  \cite{Raeburn:Graph_algebras}*{Section~8}).
\end{example}

The next lemma describes all possible choices for~\(\Reg\):

\begin{lemma}
  \label {lem:Katsura_ideal_for_spaces}
  Let \(f\colon X\to Y\) be a continuous map between Hausdorff locally
  compact spaces.  There is an open subset \(Y_{\max}\subseteq Y\)
  such that for an open subset \(W\subseteq X\) the restriction
  \(f^{-1}(W) \to W\) of~\(f\) is proper if and only if
  \(W\subseteq Y_{\max}\).
\end{lemma}

\begin{proof}
  The map~\(f\) induces a nondegenerate \Star{}homomorphism
  \[
    f^*\colon \Cont_0(Y) \to \Mult(\Cont_0(X)) = \Contb(X).
  \]
  The restriction of~\(f\) to \(f^{-1}(W) \to W\) is proper for an
  open subset \(W\subseteq Y\) if and only if~\(f^*\) maps
  \(\Cont_0(W)\) to \(\Cont_0(f^{-1}(W))\).  Since \(f^*(h)\) for
  \(h\in\Cont_0(W)\) vanishes outside \(f^{-1}(W)\), this happens if
  and only if \(f^*(\Cont_0(W)) \subseteq \Cont_0(X)\).  The preimage
  \((f^*)^{-1}(\Cont_0(X))\) is an ideal in \(\Cont_0(Y)\).  It must
  be of the form \(\Cont_0(Y_{\max})\) for an open subset
  \(Y_{\max} \subseteq Y\).  The argument above shows that the
  restriction of~\(f\) to \(f^{-1}(W) \to W\) is proper if and only if
  \(W \subseteq Y_{\max}\).
\end{proof}

The following theorem would have fitted well
into~\cite{Antunes-Ko-Meyer:Groupoid_correspondences}.  We include it
here because it was used in a previous version of the proof of the
main theorem.
Muhly--Renault--Williams~\cite{Muhly-Renault-Williams:Equivalence}
show that a Morita equivalence between two Hausdorff, locally compact
groupoids induces a Morita equivalence of their \(\Cst\)\nb-algebras.
This is extended by Tu~\cite{Tu:Non-Hausdorff} to non-Hausdorff
groupoids, but using a different, more complicated construction of the
bimodule.

\begin{theorem}
  \label{the:Morita_equivalence}
  Let \(\Bisp\colon \Gr[H]\leftarrow \Gr\) be a Morita equivalence
  between two \'etale groupoids.  Then \(\Cst(\Bisp)\) is a
  Morita--Rieffel equivalence bimodule between \(\Cst(\Gr[H])\) and
  \(\Cst(\Gr)\).
\end{theorem}

\begin{proof}
  Let~\(\Bisp\) be a Morita equivalence between two \'etale groupoids
  \(\Gr[H]\) and~\(\Gr\) and let~\(\Bisp^*\) be~\(\Bisp\) with the
  left and right actions exchanged.  This is also a Morita
  equivalence.  Both \(\Bisp\) and~\(\Bisp^*\) are groupoid
  correspondences: both the left and the right actions on~\(\Bisp\)
  must be free and proper, and the anchor maps induce homeomorphisms
  \(\Bisp/\Gr \cong \Gr[H]^0\) and
  \(\Gr[H] \backslash \Bisp \cong \Gr^0\).  Orbit space projections
  are local homeomorphisms by
  \cite{Antunes-Ko-Meyer:Groupoid_correspondences}*{Lemma~2.10}.  It
  follows that the source and range anchor maps on~\(\Bisp\) are local
  homeomorphisms.  In addition, \(\Bisp \circ \Bisp^*\) and
  \(\Bisp^* \circ \Bisp\) are isomorphic to the identity
  correspondences on \(\Gr\) and~\(\Gr[H]\), respectively.
  So~\(\Bisp\) is an equivalence in the bicategory of groupoid
  correspondences.  The map \(\Bisp\mapsto \Cst(\Bisp)\) is part of a
  homomorphism of bicategories by
  \cite{Antunes-Ko-Meyer:Groupoid_correspondences}*{Theorem~7.13}.
  Therefore, \(\Cst(\Bisp)\) is an equivalence in the bicategory of
  \(\Cst\)\nb-correspondences.  These equivalences are precisely the
  Morita--Rieffel equivalences (see
  \cite{Buss-Meyer-Zhu:Higher_twisted}*{Proposition~2.11}).
  So~\(\Cst(\Bisp)\) is a Morita--Rieffel equivalence bimodule between
  \(\Cst(\Gr[H])\) and~\(\Cst(\Gr)\).
\end{proof}

\begin{remark}
  The left inner product on \(\Cst(\Bisp)\) is defined by
  \[
    \BRAKET{\xi}{\eta}(h)
    = \sum_{\setgiven{x\in\Bisp}{\rg(x) =  \s(h)}}
    \xi(h\cdot x) \conj{\eta(x)}
  \]
  for all \(\xi,\eta\in \ContS(\Bisp)\) and \(h\in\Gr[H]\).  This is the
  formula we get by transferring the right inner product on
  \(\Cst(\Bisp^*)\) defined in
  \cite{Antunes-Ko-Meyer:Groupoid_correspondences}*{Equation~(7.2)} to
  a left inner product on \(\Cst(\Bisp)\).
\end{remark}

\section{Actions of groupoid correspondences}
\label{sec:actions_groupoid_corr}

Let~\(Y\) be a topological space.  We are going to define actions of
\((\Gr,\Bisp,\Reg)\) and then use these to define the groupoid model of
\((\Gr,\Bisp,\Reg)\).

\begin{definition}[compare \cite{Antunes:Thesis}*{Definition~4.1}]
  \label{def:correspondence_action}
  An \emph{action} of~\((\Gr,\Bisp,\Reg)\) on~\(Y\) consists of a groupoid
  action of~\(\Gr\) on~\(Y\) and a continuous, open map~\(\mu\) from the
  fibre product \(\Bisp\times_{\s,\Gr^0,\rg} Y\) to~\(Y\), written
  multiplicatively as \((x,y)\mapsto x\cdot y\), such that
  \begin{enumerate}[label=\textup{(\ref*{def:correspondence_action}.\arabic*)},
    leftmargin=*,labelindent=0em]
  \item \label{en:correspondence_action_1}%
    \(\rg(x\cdot y) = \rg(x)\) and
    \(g\cdot (x\cdot y) = (g \cdot x)\cdot y\) for \(g\in \Gr\),
    \(x\in \Bisp\) and \(y\in Y\) with \(\s(g) = \rg(x)\) and
    \(\s(x) = \rg(y)\);
  \item \label{en:correspondence_action_2}%
    let \(x,x'\in \Bisp\) and \(y,y'\in Y\) satisfy \(\s(x) = \rg(y)\) and
    \(\s(x') = \rg(y')\); then \(x\cdot y = x'\cdot y'\) if and only
    if there is \(g\in \Gr\) with \(x' = x\cdot g\) and
    \(y = g\cdot y'\);
  \item \label{en:correspondence_action_3}%
    \(\rg^{-1}(\Reg) \subseteq Y\) is contained in the image
    \(\Bisp\cdot Y\) of~\(\mu\);
  \item \label{en:correspondence_action_4}%
    if \(K\subseteq \Bisp/\Gr\) is compact, then the set of all
    \(x\cdot y\) with \(x\in \Bisp\), \(y\in Y\), \([x]\in K\) and
    \(\s(x) = \rg(y)\) is closed in~\(Y\).
  \end{enumerate}
  A map \(\varphi\colon Y_1 \to Y_2\) between two
  \((\Gr,\Bisp,\Reg)\)\nb-actions is
  \emph{\((\Gr,\Bisp)\)\nb-equivariant} if
  \begin{itemize}
  \item \(\rg_{Y_2}\circ\varphi = \rg_{Y_1}\);
  \item \(\varphi(g\cdot y) = g\cdot \varphi(y)\) for all
    \(g\in \Gr\), \(y\in Y_1\) with \(\s(g) = \rg(y)\);
  \item \(\varphi(x\cdot y) = x\cdot \varphi(y)\) for all
    \(x\in \Bisp\), \(y\in Y_1\) with \(\s(x) = \rg(y)\); and
  \item \(\varphi^{-1}(\Bisp\cdot Y_2) = \Bisp\cdot Y_1\).
  \end{itemize}
\end{definition}

Let \(\Bisp\Grcomp Y\) be the orbit space of
\(\Bisp\times_{\s,\Gr^0,\rg} Y\) for the diagonal \(\Gr\)\nb-action
\((x,y)\cdot g \defeq (x\cdot g, g^{-1}\cdot y)\).  This is how the
composition of groupoid correspondences is defined, and it has similar
properties here although nothing acts on~\(Y\) on the right:

\begin{lemma}
  \label{lem:XY_properties}
  Let \(\Bisp\colon \Gr[H]\leftarrow \Gr\) be a groupoid
  correspondence and let~\(Y\) be a \(\Gr\)\nb-space.
  \begin{enumerate}[label=\textup{(\ref*{lem:XY_properties}.\arabic*)},
    leftmargin=*,labelindent=0em]
  \item The orbit space projection
    \(\Bisp\times_{\s,\Gr^0,\rg} Y \to \Bisp\Grcomp Y\) is a
    surjective local homeomorphism.
  \item The formula \(h\cdot [x,y] = [h\cdot x,y]\) defines a
    continuous \(\Gr[H]\)\nb-action on \(\Bisp\Grcomp Y\).
  \item There is a well defined, \(\Gr[H]\)\nb-equivariant, open and
    continuous map
    \[
      \pi_1 \colon \Bisp\Grcomp Y \to \Bisp/\Gr,\qquad
      [x,y] \mapsto [x].
    \]
  \end{enumerate}
\end{lemma}

\begin{proof}
  The diagonal \(\Gr\)\nb-action on \(\Bisp\times_{\s,\Gr^0,\rg} Y\)
  is basic because the action on~\(\Bisp\) is.  Then the first
  statement follows from
  \cite{Antunes-Ko-Meyer:Groupoid_correspondences}*{Lemma~2.10}.  The
  second statement is easy.  The map~\(\pi_1\) is induced on orbit
  spaces by the coordinate projection
  \(\mathrm{pr}_1\colon \Bisp\times_{\s,\Gr^0,\rg} Y\to \Bisp\), which
  is equivariant for \(\Gr\) and~\(\Gr[H]\).  As such, it is well
  defined, \(\Gr[H]\)\nb-equivariant, and continuous.  It is also open
  because all orbit space projections are continuous, open and
  surjective and \(\mathrm{pr}_1\) is open.
\end{proof}

\begin{remark}
  \label{rem:mu_homeo}
  Conditions \ref{en:correspondence_action_1}
  and~\ref{en:correspondence_action_2} say exactly that~\(\mu\)
  induces a well defined, injective, \(\Gr\)\nb-equivariant map
  \(\mu'\colon \Bisp\Grcomp Y \to Y\).  By
  Lemma~\ref{lem:XY_properties}, \(\mu\) is continuous and open if and
  only if~\(\mu'\) is, if and only if~\(\mu'\) is a homeomorphism onto
  its image.  The condition~\ref{en:correspondence_action_3} says that the
  image of~\(\mu'\) contains \(\rg^{-1}(\Reg) \subseteq Y\).
  The condition~\ref{en:correspondence_action_4} says that
  \(\mu'(\pi_1^{-1}(K)) \subseteq Y\) is closed in~\(Y\).  Since
  \(\Bisp/\Gr\) is Hausdorff, any compact subset
  \(K\subseteq \Bisp/\Gr\) is closed, and then \(\mu'(\pi_1^{-1}(K))\)
  is relatively closed in \(\Bisp\cdot Y = \mu'(\Bisp\circ Y)\)
  because~\(\pi_1\) is continuous and~\(\mu'\) is a homeomorphism onto
  its image~\(\Bisp\cdot Y\).
\end{remark}

\begin{remark}
  If \(\Reg=\Gr^0\), then \ref{en:correspondence_action_3} says
  that~\(\mu'\) is a homeomorphism \(\Bisp\circ Y \congto Y\).
  So~\ref{en:correspondence_action_4} is automatic by the previous
  remark.  Thus the definition of a groupoid correspondence action in
  \cite{Meyer:Diagrams_models}*{Definition~4.1} is the special case
  \(\Reg=\Gr^0\) of our definition.
\end{remark}

The condition~\ref{en:correspondence_action_4} in our definition
corrects a subtle error in~\cite{Antunes:Thesis} and deserves further
discussion.  We first discuss the problem in~\cite{Antunes:Thesis}.
Then we show by an example that this condition is not automatic, and
finally we prove that it follows if the anchor map \(Y\to \Gr^0\) is
proper or, more generally, if there is an equivariant map to another
action with proper anchor map.

\begin{remark}
  \label{rem:inaccuracy}
  In~\cite{Antunes:Thesis}, it is assumed instead
  of~\ref{en:correspondence_action_4} that the anchor map
  \(\rg\colon Y\to \Gr^0\) should be proper.  We will show in
  Lemma~\ref{lem:action_extra_condition_proper_case} that this is a
  stronger condition.  The universal action built
  in~\cite{Antunes:Thesis} has this extra property.  However, a
  general action of the groupoid model \(\Omega\rtimes \IS\) that is
  constructed in~\cite{Antunes:Thesis} has this extra property if and
  only if its anchor map to~\(\Omega\) is proper.  So
  \cite{Antunes:Thesis}*{Theorem~5.17} is not correct as stated.
\end{remark}

The following example shows that~\ref{en:correspondence_action_4} is
necessary to make the theory work when we allow actions whose anchor
map fails to be proper.

\begin{example}
  \label{exa:non-action}
  Let \(\Gr\) be the one-point groupoid, let
  \(\Bisp\colon \Gr\leftarrow \Gr\) be the two-point correspondence,
  and let \(\Reg=\emptyset\) be empty.  The resulting Toeplitz algebra is
  the Cuntz--Toeplitz algebra~\(\mathcal{TO}_2\).  This is known to be
  a groupoid \(\Cst\)\nb-algebra of an \'etale groupoid with a totally
  disconnected object space.  Now let~\(Y\) be the open interval
  \((0,1)\) and let \(\mu\colon Y\sqcup Y \cong \Bisp \times Y \to Y\)
  be the map that identifies the two copies of \((0,1)\) with
  \((0,1/2)\) and \((1/2,1)\) by piecewise linear maps and then embeds
  \((0,1/2) \sqcup (1/2,1)\) into \((0,1)\).  This is a homeomorphism
  from \(\Bisp\times Y\) onto an open subset of~\(Y\).  It has all the
  properties of an action in
  Definition~\ref{def:correspondence_action}
  except~\ref{en:correspondence_action_4}.  It should not be an action
  of \((\Gr,\Bisp,\emptyset)\) because there is no equivariant
  continuous map to the object space of the groupoid model
  of~\(\mathcal{TO}_2\).
\end{example}

\begin{lemma}
  \label{lem:action_extra_condition_proper_case}
  Let~\(Y\) be a locally compact, Hausdorff \(\Gr\)\nb-space with a
  \(\Gr\)\nb-equivariant map \(\Bisp\Grcomp Y \to Y\) that is a
  homeomorphism onto an open subset of~\(Y\).  If the anchor map
  \(Y\to \Gr^0\) is proper or if there is an equivariant map to
  another action with proper anchor map, then the data above also
  satisfies~\ref{en:correspondence_action_4}, making it an
  action of \((\Gr,\Bisp,\emptyset)\).
\end{lemma}

\begin{proof}
  By \cite{Antunes-Ko-Meyer:Groupoid_correspondences}*{Lemma~5.1},
  pull-backs of proper maps are again proper.  So the first coordinate
  projection
  \(\mathrm{pr}_1\colon \Bisp \times_{\s,\Gr^0,\rg} Y \to \Bisp\) is
  proper.  This implies that the induced map on orbit spaces
  \(\Bisp \Grcomp Y \to \Bisp/\Gr\) is proper (see
  \cite{Antunes-Ko-Meyer:Groupoid_correspondences}*{Lemma~5.5}).  As a
  consequence, the preimage of a compact subset
  \(K\subseteq \Bisp/\Gr\) in~\(\Bisp\Grcomp Y\) is compact.  Then it
  must be closed in the Hausdorff space~\(Y\).

  Let \(\varphi\colon Y_1 \to Y_2\) be an equivariant map of
  \((\Gr,\Bisp,\emptyset)\)-actions and assume that the anchor map
  of~\(Y_2\) is proper.  Then the preimage of~\(K\) in
  \(\Bisp\Grcomp Y_2\) is compact and hence closed in~\(Y_2\).
  Therefore its \(\varphi\)\nb-preimage in~\(Y_1\) is also closed
  because~\(\varphi\) is continuous.
\end{proof}

In the following, the indices in nets belong to some directed sets,
which we do not care to name.

\begin{lemma}
  \label{lem:extra_condition_convergence}
  Let~\(Y\) be a locally compact, Hausdorff \(\Gr\)\nb-space with a
  \(\Gr\)\nb-equivariant map \(\mu\colon \Bisp\Grcomp Y \to Y\) that is
  a homeomorphism onto an open subset of~\(Y\).  The following
  statements are equivalent to~\ref{en:correspondence_action_4}:
  \begin{enumerate}[label=\textup{(\ref*{lem:extra_condition_convergence}.\arabic*)},
    leftmargin=*,labelindent=0em]
  \item \label{en:extra_condition_convergence_1}%
    if \((x_n,y_n)\) is a net in~\(\Bisp\times_{\s,\Gr^0,\rg} Y\) such
    that the nets \((x_n)\) in~\(\Bisp\) and \(x_n \cdot y_n\) in~\(Y\)
    converge, then~\(y_n\) also converges in~\(Y\);
  \item \label{en:extra_condition_convergence_2}%
    for all \(h\in \Cont_0(\Bisp/\Gr)\), the function
    \(\pi_1^*h\colon Y\to \C\) defined by \(\pi_1^*(y) \defeq 0\) if
    \(y\notin \Bisp\cdot Y\) and
    \((\pi_1^* h)(x\cdot y) \defeq h([x])\) otherwise is continuous.
  \end{enumerate}
\end{lemma}

\begin{proof}
  We first assume \ref{en:correspondence_action_4} and show
  that~\ref{en:extra_condition_convergence_1} follows.  Let \((x_n)\)
  and \((y_n)\) be as in~\ref{en:extra_condition_convergence_1}.
  Since~\(\Bisp/\Gr\) is locally compact, we may assume without loss
  of generality that all~\((x_n)\) belong to a compact subset
  \(K\subseteq \Bisp/\Gr\).  Then \ref{en:correspondence_action_4}
  implies that \(z_\infty \defeq \lim x_n \cdot y_n\) also belongs to
  \(\Bisp\cdot Y\).  So \(z_\infty = x\cdot y\) for some
  \((x,y) \in \Bisp\times_{\s,\Gr^0,\rg} Y\).  Then
  \(\lim {}[x_n,y_n] = [x,y]\) in \(\Bisp\Grcomp Y\) because the
  multiplication map~\(\mu'\) is a homeomorphism onto its image by
  Remark~\ref{rem:mu_homeo}.  Since the map
  \(\pi_1\colon \Bisp\Grcomp Y \to \Bisp/\Gr\) is continuous,
  \([x] = \pi_1(z_\infty) = \lim \pi_1(x_n \cdot y_n) = \lim
  {}[x_n]\).  This means that \(\lim x_n = x\cdot g\) for some
  \(g\in \Gr\).  Replacing \((x,y)\) by \((x\cdot g,g^{-1}\cdot y)\),
  we may arrange that \(\lim x_n= x\).  The orbit space projection
  \(\Bisp\times_{\s,\rg} Y\to \Bisp\Grcomp Y\) is a local
  homeomorphism by Lemma~\ref{lem:XY_properties}.  Therefore, since
  \((x_n,y_n)\) converges in the orbit space
  \(\Bisp\Grcomp Y = (\Bisp\times_{\s,\rg} Y)/\Gr\), there are
  elements \(h_n\in \Gr\) such that \((x_n\cdot h_n,h_n^{-1}\cdot y_n)\)
  is defined and converges in \(\Bisp\times_{\s,\rg} Y\).
  Since~\(\Gr\) is \'etale, its right action on~\(\Bisp\) is free and
  proper, and~\((x_n)\) itself already converges, \((h_n)\) must be
  eventually constant.  Then~\((y_n)\) converges, which is what we had
  to show.

  Now assume that \ref{en:correspondence_action_4} fails.  We prove
  that~\ref{en:extra_condition_convergence_1} fails.  By assumption,
  there is a compact subset \(K\subseteq \Bisp/\Gr\) whose preimage
  under \(\pi_1\colon \Bisp\cdot Y \to \Bisp/\Gr\) is not closed
  in~\(Y\).  Since the preimage of~\(K\) is relatively closed
  in~\(\Bisp\cdot Y\), this set must have an accumulation point
  outside~\(\Bisp\cdot Y\).  So there is a net of the form
  \(x_n \cdot y_n\) with \([x_n]\in K\), \(y_n\in Y\),
  \(\s(x_n) = \rg(y_n)\), which converges towards some point
  \(z\notin \Bisp\cdot Y\).  Since~\(K\) is compact, we may pass to a
  subnet such that \(\lim {}[x_n]\) exists.  Since the orbit space
  projection \(\Bisp\to \Bisp/\Gr\) is a local homeomorphism, we may
  find a net \(g_n\in \Gr\) such that \(\lim x_n\cdot g_n\) exists
  in~\(\Bisp\).  Replacing \((x_n,y_n)\) by
  \((x_n\cdot g_n, g_n^{-1}\cdot y_n)\), we may arrange that already
  \(x = \lim x_n\) exists in~\(\Bisp\).  If the net~\(y_n\) would
  converge towards some \(y\in Y\), then
  \(\lim x_n \cdot y_n = x\cdot y\) would belong to \(\Bisp\cdot Y\),
  which is a contradiction.  So we have found a net \((x_n,y_n)\) in
  \(\Bisp\times_{\s,\Gr^0,\rg} Y\) such that the nets~\((x_n)\) and
  \((x_n \cdot y_n)\) converge, but~\((y_n)\) does not converge.

  Next we prove that \ref{en:correspondence_action_4}
  implies~\ref{en:extra_condition_convergence_2}.  Let
  \(h\in \Cont_0(\Bisp/\Gr)\).  Since
  \(\pi_1\colon \Bisp\cdot Y \to \Bisp/\Gr\) is continuous, the
  function \(\pi_1^* h\) is continuous on \(\Bisp\cdot Y \).
  By construction, it vanishes outside.  So it remains to prove that
  the set of points with \(\abs{(\pi_1^* h)(y)} < \varepsilon\) is
  open for all \(\varepsilon>0\).  The complement of this is the
  preimage of the set of all \([x]\in \Bisp/\Gr\) with
  \(\abs{h([x])}\ge \varepsilon\).  Since the latter set is compact,
  the preimage is closed in~\(Y\) by \ref{en:correspondence_action_4}.
  So its complement is open in~\(Y\) as needed.  Conversely, assume
  that there is a compact subset \(K\subseteq \Bisp/\Gr\) whose
  preimage in~\(Y\) is not closed.  There is a function
  \(h\in \Cont_0(\Bisp/\Gr)\) with \(h|_K\ge 1\).  Then the function
  \(\pi_1^*h\) is not continuous at any accumulation point of the
  preimage of~\(K\) outside \(\Bisp\cdot Y\), and such
  accumulation points exist by assumption.
\end{proof}

\begin{definition}
  \label{def:groupoid_model}
  A \emph{groupoid model} for \((\Gr,\Bisp,\Reg)\) is an \'etale
  groupoid~\(\Mod\) such that for all topological spaces~\(Y\), there
  is a natural bijection between \((\Gr,\Bisp,\Reg)\)-actions and
  \(\Mod\)\nb-actions on~\(Y\).  Here naturality means that a
  continuous map between two spaces is \((\Gr,\Bisp)\)-equivariant
  if and only if it is \(\Mod\)\nb-equivariant for the associated
  \(\Mod\)\nb-actions.
\end{definition}

\begin{example}
  \label{exa:not_equivariant}
  Let~\(Y\) carry an action of \((\Gr,\Bisp,\Reg)\) with
  \(\Bisp\cdot Y\subsetneq Y\).  The subset
  \(\Bisp\cdot Y\subsetneq Y\) is invariant for the
  \(\Gr\)\nb-action, and the multiplication restricts to a map
  \(\Bisp\circ (\Bisp\cdot Y)\to(\Bisp\cdot Y)\).  This defines an
  action of \((\Gr,\Bisp,\Reg)\) on \(\Bisp\cdot Y\); the proof of
  Lemma~\ref{lem:action_iterated_extra_condition} shows that the
  restricted action inherits \ref{en:correspondence_action_4}.  It
  is remarkable that the inclusion map
  \(\iota\colon \Bisp\cdot Y\hookrightarrow Y\) is \emph{not}
  \((\Gr,\Bisp)\)-equivariant because
  \(\iota^{-1}(\Bisp\cdot(\Bisp\cdot Y)) \neq \Bisp\cdot Y\).
  Therefore, the subset \(\Bisp\cdot Y\subseteq Y\) cannot be
  invariant for the induced action of the groupoid model.  In fact,
  this will be visible from the construction of that action below.
  Any slice of~\(\Bisp\) will map \(Y\setminus \Bisp\cdot Y\) to
  \(\Bisp \cdot Y \setminus \Bisp\cdot \Bisp \cdot Y\), so that its
  adjoint maps some points of~\(\Bisp \cdot Y\) into
  \(Y\setminus \Bisp\cdot Y\).
\end{example}

Next, we are going to prove that the groupoid model is unique up
to isomorphism.  Later, we will construct
such a groupoid and then show that its groupoid \(\Cst\)\nb-algebra is
canonically isomorphic to the Cuntz--Pimsner algebra of
\(\Cst(\Bisp)\) relative to the ideal \(\Cst(\Gr_\Reg)\).

\begin{definition}
  \label{def:universal_action}
  A \((\Gr,\Bisp,\Reg)\)-action~\(\Omega\) is \emph{universal} if for
  any \((\Gr,\Bisp,\Reg)\)-action~\(Y\), there is a unique
  \((\Gr,\Bisp)\)\nb-equivariant map \(Y\to \Omega\), that is, it
  is a terminal object in the category of
  \((\Gr,\Bisp,\Reg)\)-actions.
\end{definition}

\begin{proposition}
  \label{pro:universal_action}
  Let~\(\Mod\) be a groupoid model for \((\Gr,\Bisp,\Reg)\)-actions.
  Its object space~\(\Mod^0\) carries a canonical
  \((\Gr,\Bisp,\Reg)\)-action, which is universal.
\end{proposition}

\begin{proof}
  Give~\(\Mod^0\) its canonical \(\Mod\)\nb-action, where an arrow
  \(g\in\Mod\) maps \(\s(g)\) to~\(\rg(g)\).  We claim
  that~\(\Mod^0\) is terminal in the category of \(\Mod\)\nb-actions.
  Let~\(Y\) be an \(\Mod\)\nb-space.  The anchor map
  \(s \colon Y\to \Mod^0\) is \(\Mod\)\nb-equivariant.  It is the only
  \(\Mod\)\nb-equivariant map \(Y\to \Mod^0\) because any
  \(\Mod\)\nb-equivariant map \(f \colon Y \to \Mod^0\) must intertwine
  the anchor maps to~\(\Mod^0\).  Thus~\(\Mod^0\) is terminal in the
  category of \(\Mod\)\nb-actions.  By the definition of a groupoid
  model, we may turn an \(\Mod\)\nb-action on a space~\(Y\) into a
  \((\Gr,\Bisp,\Reg)\)-action on~\(Y\), such that this defines an
  isomorphism of categories.  Therefore, it preserves terminal
  objects.
\end{proof}

\begin{proposition}
  \label{pro:groupoid_model_unique}
  Let \(\Mod\) and~\(\Mod'\) be two groupoid models for
  \((\Gr,\Bisp,\Reg)\)-actions.  There is a unique groupoid
  isomorphism \(\Mod\cong\Mod'\) that is compatible with the
  equivalence between actions of \(\Mod\), \(\Mod'\)
  and~\((\Gr,\Bisp,\Reg)\).
\end{proposition}

\begin{proof}
  This may be deduced quickly from
  \cite{Meyer-Zhu:Groupoids}*{Propositions 4.23 and~4.24} together
  with a description of the invariant maps for \'etale groupoid
  actions in terms of equivariant maps.  We give a more direct
  argument here.  The canonical actions of \(\Mod\) and~\(\Mod'\) on
  their object spaces correspond to actions of \((\Gr,\Bisp,\Reg)\),
  which are both universal.  Since any two terminal objects are
  isomorphic with a unique isomorphism, there is a unique
  \((\Gr,\Bisp)\)-equivariant homeomorphism
  \(\Mod^0 \cong (\Mod')^0\).  To simplify notation, we identify
  \(\Mod^0\) and \((\Mod')^0\) through this isomorphism.

  An \(\Mod\)\nb-action on a space~\(Y\) contains an anchor map
  \(Y\to \Mod^0\).  Forgetting the rest of the action gives a
  forgetful functor from the category of \(\Mod\)\nb-actions to the
  category of spaces over~\(\Mod^0\).  This functor has a left
  adjoint, namely, the map that sends a space~\(Z\) with a continuous
  map \(\varrho\colon Z\to\Mod^0\) to
  \(\Mod\times_{\s,\Mod^0,\varrho} Z\) with the obvious left
  \(\Mod\)\nb-action,
  \(\gamma_1\cdot (\gamma_2,z) \defeq (\gamma_1\cdot \gamma_2,z)\).
  Being left adjoint to the forgetful functor means that
  \(\Mod\)\nb-equivariant continuous maps
  \(\psi\colon \Mod\times_{\s,\Mod^0,\varrho} Z \to Y\) for an
  \(\Mod\)\nb-space~\(Y\) are in natural bijection with maps
  \(\varphi\colon Z \to Y\) that satisfy
  \(\rg\circ \varphi = \varrho\).  Under this natural bijection,
  \(\psi\) corresponds to the map \(\psi^\flat\colon Z\to Y\),
  \(z\mapsto \psi(1_{\varrho(z)},y)\), whereas
  \(\varphi\colon Z\to Y\) corresponds to the map
  \(\varphi^\#\colon \Mod\times_{\s,\Mod^0,\varrho} Z \to Y\) defined
  by \(\varphi^\#(\gamma,z) \defeq \gamma\cdot \varphi(z)\).  We
  give~\(\Mod\) the left multiplication action of~\(\Mod\), whose
  anchor map is~\(\rg\).  The adjunct of the unit map
  \(u\colon \Mod^0 \to \Mod\), \(x\mapsto 1_x\), is the canonical
  \(\Mod\)\nb-equivariant isomorphism
  \(u^\#\colon \Mod\times_{\s,\Mod^0,\Id} \Mod^0 \congto \Mod\).  The
  identity map \(\Mod \to \Mod\) and the map
  \(u\rg\colon \Mod\to\Mod\), \(\gamma\mapsto 1_{\rg(\gamma)}\), are
  maps over~\(\Mod^0\).  Their adjuncts are the multiplication map
  \(\Id^\#\colon \Mod\times_{\s,\Mod^0,\rg} \Mod \to \Mod\),
  \((\gamma,\eta) \mapsto \gamma\cdot \eta\), and the first coordinate
  projection
  \((u\rg)^\#\colon \Mod\times_{\s,\Mod^0,\rg} \Mod \to \Mod\),
  \((\gamma,\eta) \mapsto \gamma\cdot 1_{\rg(\eta)} = \gamma\),
  respectively.

  Since \(\Mod\) and~\(\Mod'\) are both groupoid models for
  \((\Gr,\Bisp,\Reg)\), the categories of \(\Mod\)\nb-actions and
  \(\Mod'\)\nb-actions are isomorphic to the category of
  \((\Gr,\Bisp,\Reg)\)-actions.  It follows that there is a unique
  isomorphism between the left adjoint functors
  \(\Mod\times_{\s,\Mod^0,\varrho} \blank\) and
  \(\Mod'\times_{\s,\Mod^0,\varrho} \blank\) that preserves the
  adjunction \(\varphi\mapsto \varphi^\#\).  Plugging in~\(\Mod^0\),
  this gives a \((\Gr,\Bisp)\)-equivariant homeomorphism
  \(h\colon \Mod \congto \Mod'\).  Since~\(h\) is
  \((\Gr,\Bisp)\)-equivariant, it intertwines the range maps on
  \(\Mod\) and~\(\Mod'\).  If \(x\in\Mod^0\), then the inclusion map
  makes~\(\{x\}\) a space over~\(\Mod^0\), and the inclusion
  \(\{x\} \hookrightarrow \Mod^0\) is a map of spaces over~\(\Mod^0\).
  By naturality, the homeomorphism~\(h\) intertwines the inclusion of
  \(\Mod\times_{\s,\Mod^0,x} \{x\} =
  \setgiven{\gamma\in\Mod}{\s(\gamma)=x}\) into~\(\Mod\) and the
  corresponding map for~\(\Mod'\).  Therefore, \(h\) intertwines the
  source maps to~\(\Mod^0\).

  Next, the natural isomorphism between
  \(\Mod\times_{\s,\Mod^0,\varrho} \blank\) and
  \(\Mod'\times_{\s,\Mod^0,\varrho} \blank\) specialises to a
  \((\Gr,\Bisp)\)-equivariant homeomorphism
  \(\Mod\times_{\s,\Mod^0,\rg} \Mod \congto
  \Mod'\times_{\s,\Mod^0,\rg} \Mod\).  Composing with the
  homeomorphism \(\Mod\cong \Mod'\), we get a
  \((\Gr,\Bisp)\)-equivariant homeomorphism
  \(h_2\colon\Mod\times_{\s,\Mod^0,\rg} \Mod \congto
  \Mod'\times_{\s,\Mod^0,\rg} \Mod'\).  This map must be compatible
  with the adjunction \(\varphi\mapsto \varphi^\#\) for the maps
  \(\Id\) and~\(u\rg\) described above.  So~\(h_2\) intertwines the
  multiplication maps and the first coordinate projections.  Since the
  inclusion map \(\{\eta\} \hookrightarrow \Mod\) is a map of spaces
  over~\(\Mod^0\), naturality implies that \(h_2(\gamma,\eta)\) has
  the form \((\gamma',h(\eta))\).  Since the first coordinate
  projection is intertwined, we get
  \(h_2(\gamma,\eta)=(h(\gamma),h(\eta))\).  Thus~\(h\) is a groupoid
  isomorphism.
\end{proof}

\begin{example}
  \label{exa:graph_action}
  Let \(\Gr=V\) be a discrete set.  So~\(\Bisp\) is a directed graph
  \(\rg,\s\colon E\rightrightarrows V\) (see
  Example~\ref{exa:graph_as_correspondence}).  Let \(\Reg\subseteq V\)
  be a set of regular vertices.  An action of~\(\Bisp\) on a
  space~\(Y\) is given by a continuous map \(\rg\colon Y \to V\) and a
  continuous, open map \(\mu\colon E\times_{\s,V,\rg} Y\to Y\),
  \((e,y) \mapsto e\cdot y\).  The conditions in
  Definition~\ref{def:correspondence_action} simplify to
  \(\rg(e\cdot y) = \rg(e)\); \(\mu\) being injective;
  \(\rg^{-1}(\Reg) \subseteq E\cdot Y\); and \(F\cdot Y \subseteq Y\)
  being closed in~\(Y\) for all finite \(F\subseteq E\).  It suffices,
  of course, to require this when~\(F\) is a singleton.  The map
  \(\rg\colon Y\to V\) is equivalent to a disjoint union decomposition
  \(Y = \bigsqcup_{v\in V} Y_v\) with \(Y_v \defeq \rg^{-1}(v)\).
  Then \(E\times_{\s,V,\rg} Y \cong \bigsqcup_{e\in E} Y_{\s(e)}\).
  So the injective, continuous, open map~\(\mu\) is equivalent to a
  family of homeomorphisms~\(\mu_e\) from~\(Y_{\s(e)}\) to clopen
  subsets \(e\cdot Y_{\s(e)}\subseteq Y_{\rg(e)}\) for all \(e\in E\),
  such that the subsets \(e\cdot Y_{\s(e)}\) for \(e\in E\) are
  mutually disjoint.  In addition, \ref{en:correspondence_action_3} is
  equivalent to \(Y_v = \bigsqcup_{\rg(e) = v} e\cdot Y_{\s(e)}\) for
  all \(v\in \Reg\).
\end{example}

\section{Encoding actions through inverse semigroups}
\label{sec:act_through_slices}

In this section, we encode an action of \((\Gr,\Bisp,\Reg)\) on~\(Y\)
in terms of certain partial homeomorphisms of~\(Y\).  Partial
homeomorphisms generate pseudogroups or inverse semigroups, and we
then rewrite an action of \((\Gr,\Bisp,\Reg)\) as an action of a
certain inverse semigroup with certain extra properties.  Like a group
action, an inverse semigroup action on a space has a transformation
groupoid.  We will later construct the groupoid model as the
transformation groupoid for the action of this inverse semigroup
induced by the universal action of \((\Gr,\Bisp,\Reg)\).  This is why
it is crucial to rewrite the action of \((\Gr,\Bisp,\Reg)\) as an
inverse semigroup action.

Let \(\Bisp\colon \Gr[H]\leftarrow \Gr\) be a groupoid correspondence.
A \emph{slice} of~\(\Bisp\) is an open subset \(\Slice\subseteq \Bisp\)
such that both \(\s\colon \Bisp \to \Gr^0\) and the orbit space
projection \(\Qu\colon \Bisp\prto \Bisp/\Gr\) are injective on~\(\Slice\)
(see
\cite{Antunes-Ko-Meyer:Groupoid_correspondences}*{Definition~7.2}).
It is shown in~\cite{Antunes-Ko-Meyer:Groupoid_correspondences} that
any point in~\(\Bisp\) has an open neighbourhood that is a slice.  Let
\(\Bis(\Bisp)\) be the set of all slices of~\(\Bisp\) and let
\[
  \Bis \defeq \Bis(\Gr) \sqcup \Bis(\Bisp).
\]

If \(\Bisp=\Gr\) is the arrow space of a groupoid, then our slices are
the usual ones, which are also often called bisections.  If \(\Slice\)
and~\(\Slice[V]\) are slices in composable groupoid correspondences
\(\Bisp\) and~\(\Bisp[Y]\), then
\(\Slice \Slice[V] = \setgiven{x\cdot y}{x\in\Slice,\ y\in
  \Slice[V]}\) is a slice in \(\Bisp \Grcomp \Bisp[Y]\) (see
\cite{Antunes-Ko-Meyer:Groupoid_correspondences}*{Lemma~7.11}).  In
particular, since \(\Gr[H]\Grcomp \Bisp = \Bisp = \Bisp \Grcomp \Gr\),
the product of a slice in~\(\Bisp\) with slices in \(\Gr[H]\)
and~\(\Gr\) is again a slice.  If \(\Slice,\Slice[V]\) are two slices
in~\(\Bisp\), let \(\braket{\Slice}{\Slice[V]}\) be the set of all
\(g\in\Gr\) for which there are \(x\in \Slice\), \(y\in \Slice[V]\)
with \(x\cdot g= y\).  This is a slice in~\(\Gr\) (see
\cite{Antunes-Ko-Meyer:Groupoid_correspondences}*{Lemma~7.4}).  The
bracket notation is justified by
\cite{Antunes-Ko-Meyer:Groupoid_correspondences}*{Proposition~3.5}.

\begin{lemma}
  \label{lem:slice_acts}
  Let \(\Bisp\colon \Gr[H]\leftarrow \Gr\) be a groupoid
  correspondence and let~\(Y\) be a \(\Gr\)\nb-space.  Let
  \(\Slice\subseteq\Bisp\) be a slice.  Write \([\gamma,y]\) instead of
  \(\Qu(\gamma,y)\) for the image of
  \((\gamma,y) \in \Bisp\times_{\s,\Gr^0,\rg} Y\)
  in~\(\Bisp\Grcomp_{\Gr} Y\).  There is a homeomorphism~\(\Slice_*\)
  between the open subsets \(\rg^{-1}(\s(\Slice))\subseteq Y\) and
  \(\Qu(\Slice\times_{\s,\Gr^0,\rg} Y)\subseteq \Bisp\Grcomp_{\Gr} Y\)
  which maps \(y\in Y\) to \([\gamma,y]\) for the unique
  \(\gamma\in\Slice\) with \(\s(\gamma) = \rg(y)\).
\end{lemma}

\begin{proof}
  Since~\(\s|_{\Slice}\) is a homeomorphism
  \(\Slice \to \s(\Slice)\subseteq \Gr^0\), the map~\(\Slice_*\) is well defined
  and continuous.  Its image is
  \(\Qu(\Slice\times_{\s,\Gr^0,\rg} Y) \subseteq \Bisp\Grcomp_{\Gr}Y\).
  Suppose \(\Slice_*(y_1)=\Slice_*(y_2)\) for \(y_1,y_2\in\rg^{-1}(\s(\Slice))\).
  That is, there are \(\gamma_i\in\Slice\) with \(\s(\gamma_i)=\rg(y_i)\)
  for \(i=1,2\) and \([\gamma_1,y_1]=[\gamma_2,y_2]\) in
  \(\Bisp\Grcomp_{\Gr} Y\).  Then there is \(g\in\Gr\) with
  \(\gamma_2 = \gamma_1 g\) and \(y_1 = g y_2\).  Then
  \(\Qu(\gamma_1)=\Qu(\gamma_2)\).  This implies
  \(\gamma_1 = \gamma_2\) because~\(\Qu\) is injective on~\(\Slice\).
  Since the right \(\Gr\)\nb-action on~\(\Bisp\) is free, \(g\) is
  a unit and hence \(y_1=y_2\).  Thus~\(\Slice_*\) is injective.

  The quotient map
  \(\Qu\colon \Bisp\times_{\s,\Gr^0,\rg} Y \prto \Bisp\Grcomp_{\Gr}
  Y\) is a local homeomorphism by Lemma~\ref{lem:XY_properties}.  So
  is the coordinate projection
  \(\pr_2\colon \Bisp\times_{\s,\Gr^0,\rg} Y \prto Y\) by
  \cite{Antunes-Ko-Meyer:Groupoid_correspondences}*{Lemma~2.9} because
  \(\s\colon \Bisp\to\Gr^0\) is a local homeomorphism.  The
  map~\(\Slice_*\) is the composite of the local section
  \(y\mapsto (\gamma,y)\) of \(\pr_2\) and the map~\(\Qu\).  Hence it
  is a local homeomorphism.  Being injective, it is a homeomorphism
  onto an open subset of \(\Bisp\Grcomp_{\Gr} Y\).
\end{proof}

An action of an étale groupoid induces an action of its slices by
partial homeomorphisms (see~\cite{Exel:Inverse_combinatorial}).  The
following construction generalises this to an action
of~\((\Gr,\Bisp,\Reg)\) on~\(Y\).  A slice \(\Slice\subseteq\Bisp\)
induces a partial homeomorphism~\(\Slice_*\) from~\(Y\) to
\(\Bisp \Grcomp Y\) by Lemma~\ref{lem:slice_acts}.  Composing
with \(\mu'\colon \Bisp \Grcomp Y \to Y\) gives a partial
homeomorphism
\[
\vartheta(\Slice)\defeq \mu'\circ\Slice_* \colon
Y \supseteq \rg^{-1}(\s(\Slice)) \to Y;
\]
it maps~\(y\) to \(\gamma\cdot y\) for the unique \(\gamma\in \Slice\)
with \(\s(\gamma)=\rg(y)\).  We also view~\(\vartheta(\Slice)\) as a
partial homeomorphism of~\(Y\).  Similarly, the \(\Gr\)\nb-action
induces a homeomorphism \(\Gr\Grcomp Y \to Y\), and so a
slice~\(\Slice\) in~\(\Gr\) induces a partial homeomorphism of~\(Y\)
as well.

\begin{lemma}
  \label{lem:theta_multiplicative}
  Define \(\vartheta(\Slice)\) for slices of \(\Gr\) and~\(\Bisp\) as
  above, and let~\(\vartheta(\Slice)^*\) be the partial inverse
  of~\(\vartheta(\Slice)\).  Then
  \(\vartheta(\Slice \Slice[V]) =
  \vartheta(\Slice)\vartheta(\Slice[V])\) for all
  \(\Slice,\Slice[V]\in \Bis\) with \(\Slice\in \Bis(\Gr)\) or
  \(\Slice[V]\in \Bis(\Gr)\), and
  \(\vartheta(\Slice_1)^*\vartheta(\Slice_2) =
  \vartheta(\braket{\Slice_1}{\Slice_2})\) if \(\Slice_1,\Slice_2\)
  are slices in~\(\Bisp\).
\end{lemma}

\begin{proof}
  We prove the relation
  \(\vartheta(\Slice_1)^*\vartheta(\Slice_2) =
  \vartheta(\braket{\Slice_1}{\Slice_2})\).  The relation
  \(\vartheta(\Slice \Slice[V]) =
  \vartheta(\Slice)\vartheta(\Slice[V])\) for
  \(\Slice,\Slice[V]\in \Bis\) with \(\Slice\in \Bis(\Gr)\) or
  \(\Slice[V]\in \Bis(\Gr)\) follows similarly from
  \ref{en:correspondence_action_1}
  and~\ref{en:correspondence_action_2}.  Recall that
  \(\vartheta(\Slice_i)\) for \(i=1,2\) is the partial homeomorphism
  that maps \(y_i\in Y\) with \(\rg(y_i)\in\s(\Slice_i)\) to
  \(\gamma_i\cdot y_i\) for the unique \(\gamma_i\in\Slice_i\) with
  \(\s(\gamma_i) = \rg(y_i)\).  Thus the composite partial
  homeomorphism \(\vartheta(\Slice_1)^*\vartheta(\Slice_2)\)
  maps~\(y_2\) to~\(y_1\) if and only if
  \(\gamma_1 \cdot y_1 = \gamma_2 \cdot y_2\).  By
  \ref{en:correspondence_action_2}, this happens if and only if there
  is \(g\in\Gr\) with \(\gamma_1 g = \gamma_2\) and \(g y_2 = y_1\).
  If such a~\(g\) exists, then it is unique and belongs to
  \(\vartheta(\braket{\Slice_1}{\Slice_2})\), and we get
  \(y_1 = \vartheta(\braket{\Slice_1}{\Slice_2}) y_2\).  Thus
  \(\vartheta(\Slice_1)^*\vartheta(\Slice_2) =
  \vartheta(\braket{\Slice_1}{\Slice_2})\).
\end{proof}

\begin{lemma}
  \label{lem:range_source_theta_X}
  Let \(\rg_Y\colon Y\to\Gr^0\) be the anchor map and let
  \(\pi_1\colon \Bisp\cdot Y \to \Bisp/\Gr\) be as in
  Lemma~\textup{\ref{lem:XY_properties}}.  If \(\Slice\in\Bis(\Gr)\),
  then \(\vartheta(\Slice)\) is a homeomorphism from
  \(\rg_Y^{-1}(\s(\Slice)) \subseteq Y\) to
  \(\rg_Y^{-1}(\rg(\Slice)) \subseteq Y\).  If
  \(\Slice\in\Bis(\Bisp)\), then \(\vartheta(\Slice)\) is a
  homeomorphism from \(\rg_Y^{-1}(\s(\Slice)) \subseteq Y\) to
  \(\pi_1^{-1}(\Qu(\Slice)) \subseteq Y\).
\end{lemma}

\begin{proof}
  We prove the claim for \(\Slice\in\Bis(\Bisp)\).  The claims for
  the source and range of \(\vartheta(\Slice)\) for
  \(\Slice\in\Bis(\Gr)\) are proven similarly.  By construction, the
  image of \(\vartheta(\Slice)\) is the set of all \(x\cdot y\) for
  \(x\in\Bisp\), \(y\in Y\) with \(\s(x) = \rg(y)\) and
  \(x \in \Slice\).  Since
  \(x \cdot y = (x\cdot g)\cdot (g^{-1}\cdot y)\), we may replace
  the condition \(x \in \Slice\) by \(\Qu(x) \in \Qu(\Slice)\).  So
  the range of~\(\vartheta(\Slice)\) is equal to
  \(\pi_1^{-1}(\Qu(\Slice))\subseteq \Bisp\cdot Y\).
\end{proof}

For a space~\(Y\), let \(I(Y)\) denote the inverse semigroup of all
partial homeomorphisms on~\(Y\).

\begin{lemma}
  \label{lem:action_from_theta}
  Let~\(Y\) be a space and let \(\vartheta\colon \Bis\to I(Y)\) be a
  map.  It comes from a \((\Gr,\Bisp,\Reg)\)\nb-action on~\(Y\) if and
  only if
  \begin{enumerate}[label=\textup{(\ref*{lem:action_from_theta}.\arabic*)},
    leftmargin=*,labelindent=0em]
  \item \label{en:action_from_theta1}%
    \(\vartheta(\Slice \Slice[V]) =
    \vartheta(\Slice)\vartheta(\Slice[V])\) if
    \(\Slice,\Slice[V]\in \Bis\) and \(\Slice\in \Bis(\Gr)\) or
    \(\Slice[V]\in \Bis(\Gr)\);
  \item \label{en:action_from_theta2}%
    \(\vartheta(\Slice_1)^*\vartheta(\Slice_2) =
    \vartheta(\braket{\Slice_1}{\Slice_2})\) for all
    \(\Slice_1,\Slice_2\in\Bis(\Bisp)\);
  \item \label{en:action_from_theta4}%
    \(\vartheta(\Gr^0)\) is the identity map on all of~\(Y\) and
    \(\vartheta(\emptyset)\) is the empty partial map;
  \item \label{en:action_from_theta5}%
    the restriction of~\(\vartheta\) to open subsets of~\(\Gr^0\)
    commutes with arbitrary unions;
  \item \label{en:action_from_theta3}%
    the images of~\(\vartheta(\Slice)\) for \(\Slice\in\Bis(\Bisp)\) cover the
    domain of~\(\vartheta(\Reg)\);
  \item \label{en:action_from_theta6}%
    if \(\Slice\in\Bis(\Bisp)\) is precompact in~\(\Bisp\), then the
    closure of the codomain of \(\vartheta(\Slice)\) is contained in the
    union of the codomains of~\(\vartheta(W)\) for \(W\in\Bis(\Bisp)\).
  \end{enumerate}
  The corresponding \((\Gr,\Bisp,\Reg)\)-action on~\(Y\) is unique if it
  exists.
\end{lemma}

\begin{proof}
  Assume first that the maps~\(\vartheta\) come from an action of
  \((\Gr,\Bisp,\Reg)\).  Then \ref{en:action_from_theta1}
  and~\ref{en:action_from_theta2} follow from
  Lemma~\ref{lem:theta_multiplicative}.  If \(W\subseteq \Gr^0\) is
  open, then \(\vartheta(W)\) acts on~\(Y\) by the identity map on
  \(\rg^{-1}(W)\).  This implies the statements in
  \ref{en:action_from_theta4} and~\ref{en:action_from_theta5}.
  Lemma~\ref{lem:range_source_theta_X} implies that the ranges of
  \(\vartheta(\Slice)\) for \(\Slice\in\Bis(\Bisp)\)
  cover~\(\rg^{-1}(\Reg)\) if and only if ~\(\rg^{-1}(\Reg)\) is
  contained in \(\Bisp\cdot Y\).  So~\ref{en:action_from_theta3}
  holds.

  We claim that~\ref{en:action_from_theta6} is
  equivalent to \ref{en:correspondence_action_4}.
  In one direction, the closure of the image of~\(\Slice\)
  in~\(\Bisp/\Gr\) is compact, and so
  \(\vartheta(\Slice)\cdot Y = \Slice \cdot Y\) is contained in a relatively
  closed subset in \(\Bisp\cdot Y \subseteq Y\) by
  \ref{en:correspondence_action_4}.
  In the other direction, if \(K\subseteq \Bisp/\Gr\), there are
  finitely many slices~\(\Slice_i\) in~\(\Bisp\) whose images
  in~\(\Bisp/\Gr\) cover~\(K\).  By~\ref{en:action_from_theta6}, the
  codomains of~\(\Slice_i\) have closure contained in
  \(\Bisp\cdot Y\), and then this remains so for their union.  This
  implies that the preimage of~\(K\) in \(\Bisp\cdot Y \subseteq Y\)
  is relatively closed as required in
  \ref{en:correspondence_action_4}.

  So far, we have seen that all the conditions in the lemma are
  necessary for~\(\vartheta\) to come from an action of
  \((\Gr,\Bisp,\Reg)\).  Next, we prove that they are also sufficient.
  Any open subset of~\(\Gr^0\) is an idempotent slice in~\(\Gr\).
  The condition \ref{en:action_from_theta1} for such slices implies that
  their image is again idempotent, so that \(\vartheta(W)\) for
  \(W\subseteq\Gr^0\) is the identity map on some open subset
  of~\(Y\).  So~\(\vartheta\) restricts to a map from the lattice of
  open subsets of~\(\Gr^0\) to the lattice of open subsets of~\(Y\).
  The conditions \ref{en:action_from_theta1},
  \ref{en:action_from_theta4} and~\ref{en:action_from_theta5} for
  subsets of~\(\Gr^0\) say that the latter map commutes with finite
  intersections and arbitrary unions.  Since the two spaces involved
  are locally compact, they are ``sober''.  Therefore, by
  \cite{Meyer-Nest:Bootstrap}*{Lemma~2.25}, any map between their
  lattices of open subsets that commutes with arbitrary unions and
  finite intersections is of the form \(W\mapsto\rg^{-1}(W)\) for a
  continuous map \(\rg\colon Y\to\Gr^0\).  That is,
  \(\vartheta(W)\) is the identity map on \(\rg^{-1}(W)\) for any open
  subset \(W\subseteq \Gr^0\).

  Let \(\Slice\in\Bis(\Bisp)\), \(x\in\Slice\), \(y\in Y\) satisfy
  \(\s(x) = \rg(y)\).  The construction of~\(\rg\) says that~\(y\) is
  in the domain of \(\Slice^*\Slice = \s(\Slice)\).  So
  \(\vartheta(\Slice)\) is defined at~\(y\).  We want to put
  \(x\cdot y \defeq \vartheta(\Slice)(y)\).  This is the only
  possibility for a \((\Gr,\Bisp,\Reg)\)-action that induces the
  map~\(\vartheta(\Slice)\).  We must show that this formula for all
  \((x,y)\) gives a well defined map
  \(\Bisp\times_{\s,\Gr^0,\rg} Y \to Y\).  If
  \(\Slice[V]\in\Bis(\Bisp)\) is another slice with \(x\in\Slice[V]\),
  then \(x\in \Slice\cap \Slice[V]\).  Let
  \(W = \s(\Slice\cap \Slice[V])\subseteq \Gr^0\), viewed as a slice
  of~\(\Gr\).  Then
  \(\Slice \cdot W = \Slice\cap \Slice[V] = \Slice[V] \cdot W\) and
  \[
  \vartheta(\Slice)y
  = \vartheta(\Slice)\vartheta(W) y
  = \vartheta(\Slice[V])\vartheta(W) y
  = \vartheta(\Slice[V])y.
  \]
  Thus \(x\cdot y\) is well defined.  The resulting map
  \(\Bisp \times_{\s,\rg} Y \to Y\) is continuous and open because the
  maps~\(\vartheta(\Slice)\) are partial homeomorphisms.  Its image
  contains the domain of \(\vartheta(\Reg)\)
  by~\ref{en:action_from_theta3}; this is \(\rg^{-1}(\Reg)\) by the
  construction of~\(\rg\).

  A similar construction for slices in~\(\Gr\) gives a continuous
  multiplication map \(\Gr\times_{\s,\Gr^0,\rg} Y\ \to Y\).
  The condition~\ref{en:action_from_theta4} implies
  \(1_{\rg(y)} \cdot y = y\) for all \(y\in Y\),
  and~\ref{en:action_from_theta1} for two slices in~\(\Gr\) implies
  that this is a \(\Gr\)\nb-action.

  The multiplicativity of~\(\vartheta\) for slices in~\(\Gr^0\)
  and~\(\Bisp\) implies \(\rg(x\cdot y) = \rg(x)\) for all \(x\in\Slice\),
  \(y\in Y\) with \(\s(x) = \rg(y)\) because both sides are in the
  domains of the same \(\vartheta(W)\) for \(W\subseteq \Gr^0\) open.
  When applied to slices in \(\Gr\) and~\(\Bisp\), it implies
  \(g\cdot (x\cdot y) = (g\cdot x)\cdot y\) for all \(g\in\Gr\) with
  \(\s(g)=\rg(x)\).  When applied to slices in~\(\Bisp\) and \(\Gr\),
  it implies \((x \cdot g) \cdot (g^{-1}\cdot y) = x\cdot y\) for all
  \(g\in\Gr\) with \(\rg(g)=\s(x)\).  Assume now that
  \(x_1 \cdot y_1 = x_2 \cdot y_2\) for \(x_1,x_2\in\Bisp\),
  \(y_1,y_2\in Y\).  Choose \(\Slice_1,\Slice_2\in\Bis(\Bisp)\) with
  \(x_j\in \Slice_j\) for \(j=1,2\).  Then
  \[
    \vartheta(\Slice_1) y_1 = x_1\cdot y_1 = x_2\cdot y_2 = \vartheta(\Slice_2)y_2
  \]
  and so \(y_1 = \vartheta(\Slice_1)^* \vartheta(\Slice_2) y_2\),
  which is equal to \(\vartheta(\braket{\Slice_1}{\Slice_2}) y_2\)
  by~\ref{en:action_from_theta3}.  Equivalently, \(y_1 = g\cdot y_2\)
  for the unique \(g\in \braket{\Slice_1}{\Slice_2}\) with
  \(\s(g) = \rg(y_2)\).  Now \(\rg(g) = \rg(y_1) = \s(x_1)\) and
  \(g\in \braket{\Slice_1}{\Slice_2}\) imply \(x_1 \cdot g = x_2\).
  So there is \(g\in \Gr\) with \(x_2 = x_1 \cdot g\) and
  \(y_1 = g\cdot y_2\), as required in
  \ref{en:correspondence_action_2}.  We have already shown
  that~\ref{en:action_from_theta6} is equivalent to
  \ref{en:correspondence_action_4}.  So we get the data and all the
  conditions for a  \((\Gr,\Bisp,\Reg)\)-action.
\end{proof}

The next lemma describes \((\Gr,\Bisp)\)-equivariant maps through the
partial homeomorphisms \(\vartheta(\Slice)\) for \(\Slice\in\Bis\).
This is important to characterise the groupoid model.

\begin{lemma}
  \label{lem:theta_gives_equivariant_invariant}
  Let \(Y\) and~\(Y'\) be \((\Gr,\Bisp,\Reg)\)-actions.
  A continuous map \(\varphi\colon Y\to Y'\) is
  \((\Gr,\Bisp)\)-equivariant if and only if
  \(\vartheta'(\Slice)\circ\varphi = \varphi\circ\vartheta(\Slice)\)
  and
  \(\vartheta'(\Slice)^*\circ\varphi =
  \varphi\circ\vartheta(\Slice)^*\) as partial maps for all
  \(\Slice\in \Bis\).
\end{lemma}

\begin{proof}
  Our construction of the maps \(\vartheta(\Slice)\) from an action
  implies \(\vartheta'(\Slice)\circ\varphi = \varphi\circ\vartheta(\Slice)\)
  and \(\vartheta'(\Slice)^*\circ\varphi = \varphi\circ\vartheta(\Slice)^*\)
  as partial maps if~\(\varphi\) is equivariant.  A subtle point
  here is the equality of the domains.  By
  Lemma~\ref{lem:range_source_theta_X}, the domains of
  \(\vartheta'(\Slice)\circ\varphi\) and \(\varphi\circ\vartheta(\Slice)\) are
  \(\varphi^{-1}\bigl(\rg^{-1}_{Y'}(\s(\Slice))\bigr)\) and
  \(\rg_Y^{-1}(\s(\Slice))\), respectively, and the domains of
  \(\vartheta'(\Slice)^*\circ\varphi\) and
  \(\varphi\circ\vartheta(\Slice)^*\)
  for \(\Slice\in\Bis(\Bisp)\)
  are \(\varphi^{-1}\bigl(\pi_1^{-1}(\Qu(\Slice))\bigr)\) and
  \(\pi_1^{-1}(\Qu(\Slice))\), respectively, for the canonical map
  \(\pi_1\colon \Bisp\circ Y \to \Bisp/\Gr\).
  The equality of the first two domains is equivalent to
  \(\rg\circ\varphi = \rg\).
  The equality of the latter domains for all \(\Slice\in\Bis(\Bisp)\)
  is proven using \(\varphi^{-1}(\Bisp\cdot Y') = \Bisp\cdot Y\).
  Conversely, if \(\vartheta'(\Slice)\circ\varphi =
  \varphi\circ\vartheta(\Slice)\) and
  \(\vartheta'(\Slice)^*\circ\varphi =
  \varphi\circ\vartheta(\Slice)^*\) holds for all
  bisections~\(\Slice\) in \(\Gr\) and~\(\Bisp\), then the
  map~\(\varphi\) must be equivariant; to prove this, use the
  construction of~\(\rg\) in the proof of
  Lemma~\ref{lem:action_from_theta} and that any element~\(x\) of
  \(\Gr\) or~\(\Bisp\) belongs to a slice~\(\Slice\), and
  \(\vartheta(\Slice)y = x\cdot y\) if \(y\in Y\) satisfies \(\s(x) =
  \rg(y)\).
\end{proof}

\section{The universal action}
\label{sec:universal_action}

In this section, we build the universal action.  We first treat the
case \(\Reg=\emptyset\).  Then we reduce the general case to this one.  A
crucial ingredient for the construction is the iterated composition
of~\(\Bisp\) with itself.  Let \(\Bisp_0 \defeq \Gr\) be the identity
groupoid correspondence and let
\(\Bisp_n \defeq \Bisp_{n-1} \Grcomp \Bisp\); in particular,
\(\Bisp_1 =\Bisp\).  These are all groupoid correspondences
(see~\cite{Antunes-Ko-Meyer:Groupoid_correspondences}).  So their
right \(\Gr\)\nb-action is free and proper, so that the orbit spaces
\(\Bisp_n/\Gr\) are locally compact, Hausdorff.  The left
\(\Gr\)\nb-action on~\(\Bisp_n\) induces a \(\Gr\)\nb-action
on~\(\Bisp_n/\Gr\).  There are canonical \(\Gr\)\nb-equivariant maps
\[
  \pi_n^m\colon \Bisp_n/\Gr \to \Bisp_m/\Gr,\qquad
  [x_1,\dotsc,x_n] \mapsto   [x_1,\dotsc,x_m],
\]
for \(m \le n\), which satisfy \(\pi_m^k \circ \pi_n^m= \pi_n^k\) for
\(k \le m \le n\).  Unlike in~\cite{Meyer:Diagrams_models}, the
maps~\(\pi_n^m\) are not proper, so that it is not useful to form
their projective limit: this space need not be locally compact.

Let~\(Y\) carry an action of \((\Gr,\Bisp,\Reg)\).  The multiplication
map~\(\mu\) induces iterated multiplication maps
\[
  \mu_n\colon \Bisp_n \Grcomp Y \to Y,\qquad
  \mu_n\bigl([x_1,\dotsc,x_n,y]\bigr) \defeq x_1 \cdot \mu_{n-1}([x_2,\dotsc,x_n,y])
\]
for \(n\ge1\).  This is just~\(\mu\) for \(n=1\).  We let
\(\mu_0\colon \Bisp_0 \Grcomp Y = \Gr\Grcomp Y \to Y\) be the
\(\Gr\)\nb-action on~\(Y\).  The maps~\(\mu_n\) for \(n\in\N\) are
homeomorphisms onto open subsets in~\(Y\).

\begin{lemma}
  \label{lem:action_iterated_extra_condition}
  For \(n\in\N\) and \(h\in\Cont_0(\Bisp_n/\Gr)\), define a function
  \(\pi_n^* h\colon Y\to \C\) by \((\pi_n^* h)(y)\defeq 0\) for
  \(y\notin \Bisp_n\cdot Y\), and
  \((\pi_n^* h)(\omega \cdot y)\defeq h([\omega])\) for all
  \(\omega\in \Bisp_n\), \(y\in Y\) with \(\s(\omega) = \rg(y)\).
  These functions are continuous.
\end{lemma}

\begin{proof}
  For \(n=0\), the statement is trivial.  For \(n=1\), the statement
  is equivalent to \ref{en:correspondence_action_4} by
  Lemma~\ref{lem:extra_condition_convergence}.  So it remains to prove
  that the maps~\(\mu_n\) inherit this extra property from~\(\mu\).
  We show this by induction on~\(n\), using the equivalent
  characterisation in
  \ref{en:extra_condition_convergence_1}.  So take a net
  \((x_n,\omega_n,y_n)\) in
  \(\Bisp\times_{\s,\rg} \Bisp_{n-1} \times_{\s,\rg} Y\) such that the
  net \([x_n,\omega_n]\) converges in~\(\Bisp_n\) and
  \(x_n\cdot \omega_n \cdot y_n\) converges in~\(Y\).  Then there are
  arrows \(g_n'\in\Gr\) such that \((x_n\cdot g,g^{-1}\cdot \omega_n)\)
  converges in \(\Bisp\times_{\s,\rg} \Bisp_{n-1}\).  We may change
  our net to arrange without loss of generality that already
  \((x_n, \omega_n)\) converges in
  \(\Bisp\times_{\s,\rg} \Bisp_{n-1}\).  In particular, \(\omega_n\)
  converges in \(\Bisp_{n-1}\).  Since
  \(x_n \cdot \omega_n \cdot y_n\) converges in~\(Y\),
  \ref{en:extra_condition_convergence_1} implies that
  \(\omega_n \cdot y_n\) converges in~\(Y\).  Then the induction
  assumption about~\(\mu_{n-1}\) implies that~\(y_n\) converges.
\end{proof}

So an action on~\(Y\) induces \Star{}homomorphisms from
\(\Cont_0(\Bisp_n/\Gr)\) to \(\Contb(Y)\) for all \(n\in\N\).  We now
define a \(\Cst\)\nb-algebra that allows us to combine these maps for
different~\(n\).

\begin{definition}
  \label{def:Amn}
  For \(0\le m \le n\), let
  \(A_{[m,n]} \defeq \bigoplus_{k=m}^n \Cont_0(\Bisp_k/\Gr)\) with the
  pointwise \Star{}operation \((f_k)^* \defeq (f_k^*)\) and the
  multiplication where \((f_k) \cdot (h_l)\) is the family with
  \(j\)th entry
  \[
    \sum_{k=0}^j (\pi_j^k)^*(f_k) \cdot h_j + f_j \cdot (\pi_j^k)^*(h_k).
  \]
\end{definition}

\begin{lemma}
  \label{lem:Amn}
  There is a unique norm that makes \(A_{[m,n]}\) into a commutative
  \(\Cst\)\nb-algebra.  If \(m\le k\le n\), then
  \(A_{[m,k]} \subseteq A_{[m,n]}\) is a \(\Cst\)\nb-subalgebra and
  \(A_{[k,n]} \subseteq A_{[m,n]}\) is a closed ideal, with quotient
  isomorphic to \(A_{[m,k-1]}\).
\end{lemma}

\begin{proof}
  Routine computations show that~\(A_{[m,n]}\) is a commutative
  \Star{}algebra, that \(A_{[m,k]} \subseteq A_{[m,n]}\) is a
  \Star{}subalgebra, and that~\(A_{[k,n]}\) is a \Star{}ideal in it
  such that \(A_{[m,n]}/A_{[k,n]} \cong A_{[m,k-1]}\).  We prove by
  induction on \(n-m\) that \(A_{[m,n]}\) is a \(\Cst\)\nb-algebra in
  a unique \(\Cst\)\nb-norm.  If \(n-m=0\), then
  \(A_{[m,n]} = \Cont_0(\Bisp_n/\Gr)\) is indeed a
  \(\Cst\)\nb-algebra and
  the assertion follows.  If the assertion is known for some
  \(n-m\), then in the extension of \Star{}algebras
  \begin{equation}
    \label{eq:split_extension_Amn}
    \Cont_0(\Bisp_n/\Gr) \into A_{[m,n]} \prto A_{[m,n-1]}
  \end{equation}
  both the kernel and quotient are known to be \(\Cst\)\nb-algebras in
  a unique \(\Cst\)\nb-norm.  The inclusion
  \(A_{[m,n-1]} \to A_{[m,n]}\) induces a map from~\(A_{[m,n-1]} \) to
  the multiplier algebra of \(\Cont_0(\Bisp_n/\Gr)\).  Then the
  obvious map provides an injective \Star{}homomorphism
  \(A_{[m,n]} \to \Mult(\Cont_0(\Bisp_n/\Gr)) \times A_{[m,n-1]}\).
  This induces a \(\Cst\)\nb-norm on~\(A_{[m,n]}\).  Since it
  generates the product topology on our direct sum, \(A_{[m,n]}\) is
  complete in this norm, and so it carries only one \(\Cst\)\nb-norm.
\end{proof}

Let~\(\Omega_{[m,n]}\) be the spectrum of~\(A_{[m,n]}\).  To
understand its topology, we recall a general topological construction
of Whyburn:

\begin{lemma}
  \label{lem:unified_space}
  Let \(X\) and~\(Y\) be locally compact Hausdorff spaces and let
  \(f\colon X\to Y\) be a continuous map.  There is a unique locally
  compact Hausdorff topology on \(Z\defeq X\sqcup Y\) with the
  following properties:
  \begin{itemize}
  \item the inclusion \(X\to Z\) is a homeomorphism onto an open
    subset;
  \item  the inclusion \(Y\to Z\) is a homeomorphism onto a closed
    subset;
  \item the map \((f, \Id_Y)\colon X\sqcup Y \to Y\) is continuous and
    proper.
  \end{itemize}
  A subset \(V\subseteq Z\) is open if and only if \(V\cap X\) is open
  in~\(X\), \(V\cap Y\) is open in~\(Y\) and for every compact subset
  \(K\subseteq V\cap Y\), \(f^{-1}(K)\setminus (V\cap X)\subseteq X\)
  is compact.
\end{lemma}

\begin{proof}
  It is shown by Whyburn~\cite{Whyburn:Unified} that the topology in
  the statement has the three properties in the lemma.  Now give~\(Z\)
  any topology with these properties.  Being locally compact and
  Hausdorff, \(Z\) is determined by the commutative
  \(\Cst\)\nb-algebra \(\Cont_0(Z)\).  The latter contains
  \(\Cont_0(X)\) as an ideal because~\(X\) is homeomorphic to an open
  subset.  Restriction to~\(Y\) identifies \(\Cont_0(Z)/\Cont_0(X)\)
  with \(\Cont_0(Y)\) because~\(Y\) is homeomorphic to
  \(Z\setminus X\).  Since the map \((f,\Id_Y)\) is proper, it induces
  a \Star{}homomorphism \(\sigma\colon \Cont_0(Y) \to \Cont_0(Z)\).
  That map is a section for the quotient map.  Therefore, there is a
  \Star{}isomorphism
  \[
    \Cont_0(Z) \congto \setgiven{(\xi,\eta) \in \Contb(X) \oplus \Cont_0(Y)}{\xi-
      \eta\circ f \in \Cont_0(X)},\qquad
    \zeta\mapsto (\zeta|_X,\zeta|_Y).
  \]
  Since its target only depends on~\(f\), the topology on its
  spectrum~\(Z\) is unique.
\end{proof}

If~\(Y\) is a one-point space, then the map~\(f\) is constant
and the space~\(Z\) in Lemma~\ref{lem:unified_space} is the one-point
compactification of~\(X\).  In general, we think of~\(X\) as a bundle
of spaces over~\(Y\) with fibres~\(f^{-1}(y)\) for \(y\in Y\).  When
we apply the one-point compactification to each fibre, we add one
point to each fibre, getting \(X\sqcup Y\) as a set.  The compactness
of the fibres translates into the projection \(Z\to Y\) being proper.
Thus the lemma describes a unique topology that makes~\(Z\) into a
fibrewise one-point compactification.  This is why we call~\(Z\) the
\emph{fibrewise one-point compactification for~\(f\)}.
Whyburn~\cite{Whyburn:Unified} and Antunes~\cite{Antunes:Thesis} call
it the unified space instead.

\begin{lemma}
  \label{lem:unified_pullback}
  The fibrewise one-point compactification construction is compatible
  with pull-backs.  More precisely, take a diagram
  \[
    \begin{tikzcd}
      && A \ar[d, "g"] \\
      X \ar[r, "f"] & Y \ar[r, "\pi"] &Y'.
    \end{tikzcd}
  \]
  Let~\(Z\) be the fibrewise one-point compactification of~\(f\) and
  let \(\tilde{f}\colon Z\to Y\) be the canonical map.  Then the
  pullback of \(\pi\circ \tilde{f}\) and~\(g\) is homeomorphic to the
  fibrewise one-point compactification of the pull-back map
  \(\Id_A\times_{Y'} f\colon A\times_{Y'} X \to A\times_{Y'}Y\).
\end{lemma}

\begin{proof}
  Pull-Backs of proper maps are again proper (see
  \cite{Antunes-Ko-Meyer:Groupoid_correspondences}*{Lemma~5.1}).
  Therefore, the topology of the pull-back of \(\Id_A\times_{Y'} f\)
  is a locally compact topology with the three properties required in
  Lemma~\ref{lem:unified_space}.
\end{proof}

Now we show how the spectra~\(\Omega_{[m,n]}\) may be built through
fibrewise one-point compactifications.  The inclusion of the ideal
\(\Cont_0(\Bisp_n/\Gr)\) in~\(A_{[m,n]}\) identifies \(\Bisp_n/\Gr\)
with an open subset of~\(\Omega_{[m,n]}\).  The complement is
naturally homeomorphic to~\(\Omega_{[m,n-1]}\).  The formula for the
multiplication in~\(A_{[m,n]}\) shows that the inclusion
\Star{}homomorphism \(A_{[m,n-1]} \hookrightarrow A_{[m,n]}\) is
nondegenerate.  Therefore, it is induced by a continuous, proper map
\begin{equation}
  \label{eq:pi_on_Omegas}
  \pi_{[m,n]}^{[m,n-1]}\colon \Omega_{[m,n]} \prto \Omega_{[m,n-1]}.
\end{equation}
This is a retraction onto the subset \(\Omega_{[m,n-1]}\).  Comparing
with the proof of Lemma~\ref{lem:unified_space}, we see that
\(\Omega_{[m,n]}\) is the fibrewise one-point compactification for the
continuous map
\begin{equation}
  \label{eq:step_n_map_to_compactify}
  \Bisp_n/\Gr \hookrightarrow \Omega_{[m,n]}  \xrightarrow{\pi_{[m,n]}^{[m,n-1]}}
  \Omega_{[m,n-1]}.
\end{equation}
Since \(\Omega_{[m,m]} = \Bisp_m/\Gr\), we may build the
spaces~\(\Omega_{[m,n]}\) from the spaces \(\Bisp_k/\Gr\) by repeated
fibrewise one-point compactification.  As a result, there is a
(discontinuous) bijection between \(\Omega_{[m,n]}\) and
\(\bigsqcup_{k=m}^n \Bisp_k/\Gr\).  Here
\(x\in \Bisp_k/\Gr\) corresponds to the character on~\(A_{[m,n]}\)
that vanishes on \(\Cont_0(\Bisp_m/\Gr)\) for \(m>k\) and is
evaluation at~\(\pi_k^m(x)\) for \(m\le k\).

\begin{definition}
  Let \(A_{[m,\infty)}\) be the colimit \(\Cst\)\nb-algebra of the
  inductive system of \(\Cst\)\nb-algebras~\(A_{[m,n]}\) for the
  obvious inclusion maps \(A_{[m,n-1]}\to A_{[m,n]}\).
  Let~\(\Omega_{[m,\infty)}\) be the spectrum of~\(A_{[m,\infty)}\).
\end{definition}

The spectrum of the inductive limit of commutative \(\Cst\)\nb-algebras
\(\varinjlim A_{[m,n]}\) is the projective limit of the spectra
of~\(A_{[m,n]}\).  So \(\Omega_{[m,\infty)}\) is the projective limit
of the projective system of maps in~\eqref{eq:pi_on_Omegas}.
Lemma~\ref{lem:Amn} implies that~\(A_{[m,\infty)}\) for \(m\in\N\) is
isomorphic to an ideal in~\(A_{[0,\infty)}\) and that these ideals
form a decreasing chain with
\[
  A_{[m,\infty)}/A_{[m+1,\infty)} \cong \Cont_0(\Bisp_m/\Gr).
\]

\begin{lemma}
  \label{lem:map_Y_to_Omega}
  Let~\(Y\) be a topological space with a
  \((\Gr,\Bisp,\emptyset)\)-action.  The maps
  \(\pi_m^*\colon \Cont_0(\Bisp_m/\Gr) \to \Contb(Y)\) for \(m\in\N\)
  constructed in
  Lemma~\textup{\ref{lem:action_iterated_extra_condition}} induce a
  \Star{}\alb{}homomorphism \(\Cont_0(\Omega_{[0,\infty)}) \to \Contb(Y)\),
  which is nondegenerate in the sense that it comes from a continuous
  map \(\varrho\colon Y\to \Omega_{[0,\infty)}\).
\end{lemma}

\begin{proof}
  The \Star{}homomorphism
  \(\pi_m^* \colon \Cont_0(\Bisp_m/\Gr) \to \Contb(Y)\) was built by
  pulling back functions along the projection
  \(\pi_m\colon \Bisp_m\cdot Y \to \Bisp_m/\Gr\) and then
  extending by~\(0\).  Let \(0\le m \le n\).  Using that
  \(\pi_n^m \circ \pi_n = \pi_m\) on
  \(\Bisp_n\cdot Y \subseteq \Bisp_m\cdot Y\), it
  follows that the resulting map \(A_{[0,n]} \to \Contb(Y)\),
  \((f_k) \mapsto \sum_{k=0}^n \pi_k^*(f_k)\), is a
  \Star{}homomorphism.  These maps induce a \Star{}homomorphism
  \(\varrho^*\colon \Cont_0(\Omega_{[0,\infty)}) = A_{[0,\infty)} \to
  \Contb(Y)\).  Each \(y\in Y\) defines a nonzero character \(\ev_y\)
  on \(\Contb(Y)\).  The character \(\ev_y\circ\varrho^*\) on
  \(\Cont_0(\Omega_{[0,\infty)})\) is also nonzero because it is even
  nonzero on
  \(\Cont_0(\Bisp_0/\Gr) = \Cont_0(\Gr^0) \subseteq
  \Cont_0(\Omega_{[0,\infty)})\).  Then it must be the evaluation
  character at some point \(\varrho(y)\in \Omega_{[0,\infty)}\).  The
  map \(\varrho\colon Y \to \Omega_{[0,\infty)}\) defined in this way
  is continuous because its composite with all
  \(\Cont_0\)\nb-functions on~\(\Omega_{[0,\infty)}\) is continuous.
\end{proof}

Each space \(\Bisp_n/\Gr\) carries a \(\Gr\)\nb-action and the maps
\(\pi_n^m\colon \Bisp_n/\Gr \to \Bisp_m/\Gr\) are
\(\Gr\)\nb-equivariant.  Recall that
\(\Omega_{[m,n]} = \bigsqcup_{k=m}^n \Bisp_k/\Gr\) as a set.  This
carries an obvious \(\Gr\)\nb-action.  This is inherited by the
projective limit sets \(\Omega_{[m,\infty)}\) for \(m\in\N\).  As a
set,
\begin{equation}
  \label{eq:X_circ_Omega_finite}
  \Omega_{[m,n]}
  \cong \bigsqcup_{k=m}^n \Bisp_k/\Gr
  \cong \bigsqcup_{k=0}^{n-m} \Bisp_m\circ (\Bisp_k/\Gr)
  \cong \Bisp_m\circ \bigsqcup_{k=0}^{n-m} (\Bisp_k/\Gr)
  \cong \Bisp_m\circ \Omega_{[0,n-m]}.
\end{equation}
Passing to the projective limit sets, these bijections induce a
canonical bijection
\begin{equation}
  \label{eq:X_circ_Omega_infty}
  \Omega_{[m,\infty)} \cong  \Bisp_m \circ \Omega_{[0,\infty]}.
\end{equation}

\begin{lemma}
  \label{lem:G-action_finite}
  The canonical \(\Gr\)\nb-action on \(\Omega_{[m,n]}\) is continuous.
\end{lemma}

\begin{proof}
  The anchor map \(\Omega_{[m,n]} \to \Gr^0\) factors through the
  continuous maps
  \(\Omega_{[m,n]} \to \Omega_{[m,m]} = \Bisp_m/\Gr \to \Gr^0\), so
  that it is continuous.  It remains to prove the continuity of a map
  \(\Gr\times_{\s,\Gr^0,\rg} \Omega_{[m,n]} \to \Omega_{[m,n]}\).
  This is clear for \(n=m\), and we prove this by induction
  over \(n-m\).  We have seen above that \(\Omega_{[m,n]}\) is the
  fibrewise one-point compactification of a certain continuous,
  \(\Gr\)\nb-equivariant map \(\Bisp_n/\Gr \to \Omega_{[m,n-1]}\).  By
  Lemma~\ref{lem:unified_pullback},
  \(\Gr\times_{\s,\Gr^0,\rg} \Omega_{[m,n]}\) is homeomorphic to the
  fibrewise one-point compactification of the induced map
  \(\Gr\times_{\s,\Gr^0,\rg}\Bisp_n/\Gr \to
  \Gr\times_{\s,\Gr^0,\rg}\Omega_{[m,n-1]}\).  Since the fibrewise one-point
  compactification is natural for commuting squares of maps, this
  implies the desired continuity.
\end{proof}

\begin{lemma}
  \label{lem:circ_homeo_finite}
  The bijections
  \(\Omega_{[m,n]} \cong \Bisp_m\circ \Omega_{[0,n-m]}\) in
  \eqref{eq:X_circ_Omega_finite} are homeomorphisms.
\end{lemma}

\begin{proof}
  Since \(\Bisp_m = \Bisp\circ \Bisp_{m-1}\), an induction on~\(m\)
  shows that it suffices to prove the claim for \(m=0\) and \(m=1\),
  and the case \(m=0\) follows from Lemma~\ref{lem:G-action_finite}.
  So let \(m=1\).  We argue by
  induction on \(n-m\).  The case \(n-m=0\) reduces to the true
  statement \(\Bisp \circ \Bisp_m/\Gr \cong \Bisp_{m+1}/\Gr\).  Assume
  the statement for all shorter intervals than \([m,n]\) and in
  particular for \([m,n-1]\).  The split exact sequence of
  \(\Cst\)\nb-algebras in~\eqref{eq:split_extension_Amn} and the proof
  of Lemma~\ref{lem:unified_space} show that \(\Omega_{[m,n]}\) is the
  fibrewise one-point compactification of the map
  \(\Bisp_n/\Gr \to \Omega_{[m,n-1]}\).  By the induction assumption,
  the canonical maps are homeomorphisms
  \(\Bisp_n \cong \Bisp\circ \Bisp_{n-1}/\Gr\) and
  \(\Omega_{[m,n-1]} \cong \Bisp \circ \Omega_{[m-1,n-2]}\).  Now
  \(\Bisp\circ \Omega_{[m-1,n-1]}\) is another locally compact
  Hausdorff space with the properties in Lemma~\ref{lem:unified_space}
  for the canonical map \(\Bisp_n/\Gr \to \Omega_{[m,n-1]}\).  Since
  the topology in Lemma~\ref{lem:unified_space} is unique, it follows
  that the canonical map in~\eqref{eq:X_circ_Omega_finite} is an
  isomorphism, completing the induction step.
\end{proof}

\begin{proposition}
  \label{pro:circ_homeo_projlim}
  Let~\((Y_n)\) be a projective system of \(\Gr\)\nb-spaces.  Then the
  induced action on \(\varprojlim Y_n\) is continuous and the
  canonical map
  \(\Bisp\circ \varprojlim Y_n \to \varprojlim \Bisp \circ Y_n\) is a
  homeomorphism.
\end{proposition}

\begin{proof}
  The canonical map exists and is continuous by the universal property
  of the limit and the naturality of \(Y\mapsto \Bisp\circ Y\).  It is
  also easy to see that it is bijective.  It is straightforward to
  check that the induced \(\Gr\)\nb-action on the projective limit
  remains continuous.  If \(\Slice\in\Bis(\Bisp)\), then
  Lemma~\ref{lem:slice_acts} provides a homeomorphism from
  \(\rg^{-1}(\s(\Slice)) \subseteq Y\) to an open subset of
  \(\Bisp\circ Y\) for any
  \(\Gr\)\nb-space~\(Y\).  Applying this to the induced
  \(\Gr\)\nb-action on \(\varprojlim Y_n\) gives a homeomorphism from
  the preimage of \(\s(\Slice)\) in \(\varprojlim Y_n\) onto an open
  subset of \(\Bisp\circ \varprojlim Y_n\).  Applying this to
  each~\(Y_n\) for \(n\in\N\) and using the naturality of the
  projective limit, we get that this map is also a homeomorphism onto
  an open subset of \(\varprojlim \Bisp\circ Y_n\).  Therefore, the
  canonical map
  \(\Bisp\circ \varprojlim Y_n \to \varprojlim \Bisp \circ Y_n\) is a
  homeomorphism on certain open subsets of the two spaces.  These open
  subsets cover the whole sets because the slices
  \(\Slice\in\Bis(\Bisp)\) cover~\(\Bisp\).  Therefore, the map is a
  homeomorphism everywhere.
\end{proof}

\begin{lemma}
  The bijections in~\eqref{eq:X_circ_Omega_infty} are homeomorphisms.
\end{lemma}

\begin{proof}
  Lemma~\ref{lem:G-action_finite} allows us to apply
  Proposition~\ref{pro:circ_homeo_projlim}, and then the claim follows
  from Lemma~\ref{lem:circ_homeo_finite}.
\end{proof}

\begin{proposition}
  \label{pro:universal_action_Toeplitz}
  The \(\Gr\)\nb-action on
  \(\Omega_{[0,\infty)} \cong \varprojlim \Omega_{[0,n]}\) and the
  homeomorphism
  \[
    \Bisp\circ \Omega_{[0,\infty)} \congto
    \Omega_{[1,\infty)} \subseteq \Omega_{[0,\infty)}
  \]
  are an action of \((\Gr,\Bisp,\emptyset)\) on
  \(\Omega_{[0,\infty)}\).  Its anchor map
  \(\Omega_{[0,\infty)} \to \Gr^0\) is proper and surjective.  This
  action of \((\Gr,\Bisp,\emptyset)\) is universal.
\end{proposition}

\begin{proof}
  Recall that \(\Omega_{[1,\infty)}\) is an open subset of
  \(\Omega_{[0,\infty)}\), being the spectrum of the
  ideal~\(A_{[1,\infty)}\) in~\(A_{[0,\infty)}\).  The
  conditions \ref{en:correspondence_action_1}
  and~\ref{en:correspondence_action_2} are easily seen to be
  satisfied, and \ref{en:correspondence_action_3}
  is empty for \(\Reg=\emptyset\).  The anchor map is proper (and
  surjective) because it is the map on spectra induced by the inclusion
  \(\Cont_0(\Gr^0) = A_{[0,0]} \hookrightarrow A_{[0,\infty)}\).
  This implies \ref{en:correspondence_action_4}
  by Lemma~\ref{lem:action_extra_condition_proper_case}.

  To prove the universality, take another action of
  \((\Gr,\Bisp,\emptyset)\) on a locally compact space~\(Y\).
  Lemma~\ref{lem:map_Y_to_Omega} provides a continuous map
  \(\varrho\colon Y\to \Omega_{[0,\infty)}\).  The canonical
  \(\Gr\)\nb-action on~\(\Omega_{[0,\infty)}\) makes this map
  \(\Gr\)\nb-equivariant.  In addition, \(y\in Y\) belongs to
  \(\Bisp\cdot Y\) if and only if there is \(h\in\Cont_0(\Bisp/\Gr)\)
  such that \(\pi_1^*(h)(y)\neq0\).  This implies that
  \(\varrho^{-1}(\Bisp\cdot \Omega_{[0,\infty)}) =
  \varrho^{-1}(\Omega_{[1,\infty)}) = \Bisp\cdot Y\).  Our canonical
  constructions ensure that \(\varrho(x\cdot y) = x\cdot \varrho(y)\)
  for all \(x\in\Bisp\), \(y\in Y\) with \(\s(x) = \rg(y)\).
  So~\(\varrho\) is indeed \((\Gr,\Bisp)\)-equivariant.

  Now let \(\varphi\colon Y\to \Omega_{[0,\infty)}\) be any
  \((\Gr,\Bisp)\)-equivariant map.  We claim that
  \(\varphi = \varrho\).  We prove that the projections of
  \(\varphi\) and~\(\varrho\) to
  \(\Omega_{[0,n]}\) are the same for all \(n\in\N\).  This proves the
  claim.  The claim for \(n=0\) concerns the projections to
  \(\Omega_{[0,0]} = \Gr^0\).  These are the anchor maps, and
  \(\rg_{\Omega_{[0,\infty)}}\circ \varphi = \rg_Y
  =\rg_{\Omega_{[0,\infty)}}\circ\varrho\) implies the claim.  Assume
  the claim has been shown for \(n-1\).  We are going to prove it
  for~\(n\).  Let \(y\in \Bisp\cdot Y\) and write \(y = x\cdot y'\)
  for \(x\in\Bisp\), \(y'\in Y\).  Then
  \(\varphi(y) = x\cdot \varphi(y')\) and
  \(\varrho(y) = x\cdot \varrho(y')\).  Now if
  \(\omega\in \Omega_{[0,\infty)}\), then the image of
  \(x\cdot \omega\) in \(\Omega_{[0,n]}\) depends only on the image
  of~\(\omega\) in \(\Omega_{[0,n-1]}\).  So the induction assumption
  implies \(\varphi(y) = \varrho(y)\) for all \(y\in\Bisp\cdot Y\).
  Since
  \(\varphi^{-1}(\Bisp\cdot \Omega_{[0,\infty)}) = \Bisp\cdot Y\) and
  \(\varrho^{-1}(\Bisp\cdot \Omega_{[0,\infty)}) = \Bisp\cdot Y\),
  both maps must send any \(y\in Y\setminus \Bisp\cdot Y\) to a point in
  \(\Omega_{[0,\infty)} \setminus \Omega_{[1,\infty)} = \Gr^0\).
  Since both maps are compatible with the anchor map to~\(\Gr^0\),
  \(\varphi(y) = \rg(y) = \varrho(y)\) holds for all these~\(y\).
  This completes the induction step and shows that~\(\varrho\) is the
  unique equivariant map \(Y\to \Omega_{[0,\infty)}\).
\end{proof}

\begin{proposition}
  \label{pro:Omega_set}
  There is a bijection between~\(\Omega_{[0,\infty)}\) and
  \[
    \bigl(\varprojlim \Bisp_m/\Gr\bigr) \sqcup
    \bigsqcup_{m=0}^\infty \Bisp_m/\Gr,
  \]
  such that the projection from~\(\Omega_{[0,\infty)}\) to
  \(\Omega_{[0,n]} \cong \bigsqcup_{m=0}^n \Bisp_m/\Gr\) becomes the
  identity map on the components~\(\Bisp_m/\Gr\) for \(m\le n\),
  \(\pi_m^n\) on~\(\Bisp_m/\Gr\) for \(m>n\), and the canonical map
  to~\(\Bisp_n/\Gr\) on the projective limit.
\end{proposition}

\begin{proof}
  Identify~\(\Omega_{[0,n]}\) with \(\bigsqcup_{k=0}^n \Bisp_k/\Gr\)
  as a set.  The map
  \(\pi_{[0,m]}^{[0,n]}\colon \Omega_{[0,m]} \to \Omega_{[0,n]}\) for
  \(m>n\) is the identity map on the components \(\Bisp_k/\Gr\) for
  \(k\le n\) and restricts to~\(\pi_k^n\) for \(k>n\).  An element of
  \(\Omega_{[0,\infty)} = \varprojlim \Omega_{[0,n]}\) is a family of
  elements \(x_n \in \Omega_{[0,n]}\) with
  \(\pi_{[0,m]}^{[0,n]}(x_m) = x_n\) for all \(m\ge n\).  The preimage
  of \(\bigsqcup_{k=0}^{n-1} \Bisp_k/\Gr \subseteq \Omega_{[0,n]}\) is
  the same subset of \(\Omega_{[0,m]}\), and \(\pi_{[0,m]}^{[0,n]}\)
  restricts to the identity map on this subset.  Therefore, if there
  is \(n\in\N\) with \(x_n \in \Bisp_k/\Gr\) for some \(k<n\),
  then~\(x_m\) must be the same element of \(\Bisp_k/\Gr\) for all
  \(m\ge n\).  This gives the subsets~\(\Bisp_k/\Gr\) in
  \(\Omega_{[0,\infty)}\).  If this does not happen, then
  \(x_n \in \Bisp_n/\Gr\) for all \(n\in\N\).  Then
  \(x_n = \pi_m^n(x_m)\) for \(m \ge n\).  So these elements
  of~\(\Omega_{[0,\infty)}\) correspond to elements of the projective
  limit \(\varprojlim \Bisp_n/\Gr\).
\end{proof}

From now on, let \(\Reg\subseteq \Gr^0\) be an open, \(\Gr\)\nb-invariant
subset such that \(\rg_*\colon \Bisp/\Gr\to \Gr^0\) restricts to a
proper map \(\rg_*^{-1}(\Reg) \to \Reg\).  We are going to build a universal
action of \((\Gr,\Bisp,\Reg)\).  For \(n\in\N\), let
\[
  (\Bisp_n/\Gr)_\Reg \defeq
  \setgiven{[\omega] \in \Bisp_n/\Gr}{\s(\omega) \in \Reg};
\]
since~\(\Reg\) is invariant, \(\s(\omega)\in \Reg\) implies
\(\s(\omega\cdot g) \in \Reg\) for all \(g\in \Gr\) with
\(\s(\omega) = \rg(g)\).

\begin{lemma}
  \label{lem:Reg-complement}
  Describe the underlying set of \(\Omega_{[0,\infty)}\) as in
  Proposition~\textup{\ref{pro:Omega_set}}.  The subsets
  \((\Bisp_n/\Gr)_\Reg \subseteq \bigsqcup_{k\in\N} \Bisp_k/\Gr
  \subseteq \Omega_{[0,\infty)}\) are open for all \(n\in\N\).
\end{lemma}

\begin{proof}
  Let \(n\in\N\).  Since pull-backs of proper maps are again proper,
  the restriction of the projection
  \(\Bisp_n \times_{\s,\rg} \Bisp \to \Bisp_n \times_{\s,\rg} \Gr^0
  \cong \Bisp_n\) to the preimage of \(\Bisp_n \times_{\s,\rg} \Reg\)
  is again proper.  This implies that the induced map
  \(\Bisp_{n+1}/\Gr \to \Bisp_n/\Gr\) restricts to a proper map on the
  preimage of \((\Bisp_n/\Gr)_\Reg\) .  The description of the
  topology of the fibrewise one-point compactification in
  Lemma~\ref{lem:unified_space} shows that \((\Bisp_n/\Gr)_\Reg\) is
  open in the fibrewise one-point compactification
  \(\Omega_{[0,n+1]}\) for the map
  \(\Bisp_{n+1}/\Gr \to \Bisp_n/\Gr \subseteq\Omega_{[0,n]}\).  The
  image of \(\Bisp_k/\Gr \to \Omega_{[0,k-1]}\) is contained
  in~\(\Bisp_{k-1}/\Gr\) and thus disjoint from~\(\Bisp_n/\Gr\) for
  \(k>n+1\).  Therefore, the image of~\((\Bisp_n/\Gr)_\Reg\)
  in~\(\Omega_{[0,k]}\) for \(k\ge n+1\) is simply the preimage
  of~\((\Bisp_n/\Gr)_\Reg\) in~\(\Omega_{[0,n+1]}\) under the
  canonical projection
  \(\pi_{[0,k]}^{[0,n+1]}\colon \Omega_{[0,k]} \to \Omega_{[0,n+1]}\).
  This remains true in the projective limit~\(\Omega_{[0,\infty)}\).
  As the preimage of an open subset under a continuous map,
  \((\Bisp_n/\Gr)_\Reg\) is open in \(\Omega_{[0,\infty)}\).
\end{proof}

\begin{lemma}
  \label{lem:Omega_Reg}
  The complement
  \[
    \Omega(\Reg) \defeq \Omega_{[0,\infty)} \setminus
    \bigsqcup_{n=0}^\infty (\Bisp_n/\Gr)_\Reg
  \]
  is closed in~\(\Omega_{[0,\infty)}\), and the restriction of the
  anchor map \(\rg\colon \Omega(\Reg) \to \Gr^0\) is proper.  There is a
  unique action of \((\Gr,\Bisp,\Reg)\) on~\(\Omega(\Reg)\) for which the
  inclusion into~\(\Omega_{[0,\infty)}\) becomes equivariant.
\end{lemma}

\begin{proof}
  Since each subset \((\Bisp_n/\Gr)_\Reg\) is open
  in~\(\Omega_{[0,\infty)}\) by Lemma~\ref{lem:Reg-complement}, their
  union remains open and so~\(\Omega(\Reg)\) is closed.  Since
  \((\Bisp_n/\Gr)_\Reg\) is \(\Gr\)\nb-invariant in \(\Bisp_n/\Gr\)
  for all \(n\in\N\), \(\Omega(\Reg)\) is \(\Gr\)\nb-invariant
  in~\(\Omega_{[0,\infty)}\).  So the \(\Gr\)\nb-action restricts to
  it.  The image of \(\Bisp \circ (\Bisp_n/\Gr)_\Reg\) under the
  canonical homeomorphism
  \(\Bisp\circ \Bisp_n/\Gr \cong \Bisp_{n+1}/\Gr\) is equal to
  \((\Bisp_{n+1}/\Gr)_\Reg\).  Thus the homeomorphism
  \(\Bisp\circ \Omega_{[0,\infty)} \to \Omega_{[1,\infty)}\) maps
  \(\Bisp\circ \Omega(\Reg)\) onto the complement of
  \(\bigsqcup_{n=0}^\infty {} (\Bisp_{n+1}/\Gr)_\Reg\).  This implies
  that the map
  \(\Bisp\circ \Omega_{[0,\infty)} \to \Omega_{[0,\infty)}\) restricts
  to a map from \(\Bisp\circ \Omega(\Reg)\) onto
  \(\Omega(\Reg) \cap (\Bisp\cdot \Omega_{[0,\infty)})\).  This
  restriction remains a homeomorphism onto its image, and it inherits
  the technical property \ref{en:correspondence_action_4} because the
  anchor map \(\Omega(\Reg) \to \Gr^0\) is proper.  The preimage
  of~\(\Reg\) under the anchor map \(\Omega(\Reg) \to \Gr^0\) is
  contained in \(\Bisp\cdot \Omega_{[0,\infty)} \cap \Omega(\Reg)\)
  because we removed \(\Reg\) from~\(\Gr^0\).  We have seen that this
  is equal to
  \(\Bisp\cdot \Omega_{[0,\infty)} \cap \Omega(\Reg) = \Bisp\cdot
  \Omega(\Reg)\).  So our action satisfies
  \ref{en:correspondence_action_3}.
\end{proof}

\begin{theorem}
  \label{the:universal_action_Reg}%
  The \((\Gr,\Bisp,\Reg)\)-action on~\(\Omega(\Reg)\) is the universal one.
\end{theorem}

\begin{proof}
  Let~\(Y\) carry an action of \((\Gr,\Bisp,\Reg)\).  This is also an
  action of \((\Gr,\Bisp,\emptyset)\), and the conditions on
  equivariant maps do not involve~\(\Reg\).  So
  Proposition~\ref{pro:universal_action_Toeplitz} provides a unique
  equivariant map \(\varrho\colon Y\to \Omega_\emptyset\).  It remains
  to check that the image of
  \(\varrho\colon Y\to \Omega_{[0,\infty)}\) is always disjoint from
  \(\bigsqcup_{n=0}^\infty (\Bisp_n/\Gr)_\Reg \subseteq
  \Omega_{[0,\infty)}\).

  The map~\(\varrho\) sends all points in \(\Bisp\cdot Y\) to
  \(\Bisp\cdot \Omega_{[0,\infty)} = \Omega_{[1,\infty)}\).  So only
  points in \(Y\setminus \Bisp\cdot Y\) are mapped to~\(\Gr^0\), and
  this is done by the anchor map \(\rg\colon Y\to \Gr^0\).  By
  assumption, \(\rg^{-1}(\Reg) \subseteq \Bisp\cdot Y\).  So no point
  in~\(Y\) can be mapped to \(\Reg\subseteq \Gr^0\).  Now let
  \(\omega\in \Bisp_n\), \(y\in Y\) satisfy \(\s(\omega) = \rg(y)\).
  Then \(\omega \cdot y \in \Bisp_n\cdot Y\) is only mapped to
  \(\Bisp_n/\Gr \subseteq \Omega_{[0,\infty)}\) if it does not belong
  to \(\Bisp_{n+1}\cdot Y = \Bisp_n\cdot (\Bisp\cdot Y)\).  Hence
  \(y\notin \Bisp\cdot Y\), so that \(\s(\omega) = \rg(y) \notin \Reg\).
  Thus \(\varrho(\omega\cdot y) \notin (\Bisp_n/\Gr)_\Reg\).
\end{proof}

\begin{example}
  \label{exa:Omega_infty_graph}
  Let \(\Gr=V\) be just a discrete set with only identity arrows.
  Then a correspondence \(\Bisp\colon \Gr\leftarrow \Gr\) is a
  discrete set~\(E\) with two maps
  \(\rg,\s\colon E\rightrightarrows V\).  Therefore
  \[
    \Bisp_n/\Gr
    = \Bisp_n
    = \setgiven{(e_1,\dotsc,e_n)\in E^n}{\s(e_j) = \rg(e_{j+1})
      \text{ for } j=1,\dotsc,n-1}.
  \]
  So we may identify \(\Bisp_m/\Gr\) with the set of paths of
  length~\(m\) and \(\varprojlim \Bisp_m/\Gr\) with the set of
  infinite paths in the directed graph described by
  \(\rg,\s\colon E\rightrightarrows V\).  The
  space~\(\Omega_{[0,\infty)}\) comprises these finite and infinite
  paths and carries a locally compact Hausdorff topology in which the
  final vertex map from paths to~\(V\) is a proper continuous map.  In
  particular, if~\(V\) is finite, then~\(\Omega_{[0,\infty)}\) is
  compact.

  Let \(\Reg \subseteq V\) be the
  set of regular vertices, that is, those vertices \(v\in V\) for
  which \(\rg^{-1}(v)\) is finite and nonempty.
  Lemma~\ref{lem:Omega_Reg} defines
  \(\Omega(\Reg) \subset \Omega_{[0,\infty)}\) as the subset
  consisting of the infinite paths and those finite paths that start in
  an irregular vertex \(v\in V\setminus \Reg\).
  Paterson~\cite{Paterson:Graph_groupoids} has already described the
  graph \(\Cst\)\nb-algebra of a possibly irregular graph as a
  groupoid \(\Cst\)\nb-algebra.  We remark without proof that the
  space~\(\Omega(\Reg)\) is identical to the object space of
  Paterson's groupoid.
\end{example}

The universal action \(\Omega(\Reg)\) is a key ingredient in our
construction because it is the object space of the
groupoid model by Proposition~\ref{pro:universal_action}.  In this
context, it is useful to study an extra property of it.

Recall that a groupoid is \emph{ample} if and only if its object
space is totally disconnected, that is, every point has a compact
open neighbourhood.  The following proposition implies that the
groupoid model is ample if the groupoid~\(\Gr\) is:

\begin{proposition}
  Assume that~\(\Gr^0\) is totally disconnected.  Then so is
  \(\Omega(\Reg)\).
\end{proposition}

\begin{proof}
  Since the orbit space projection \(\Bisp_n \to \Bisp_n/\Gr\) and
  the source map \(\Bisp_n \to \Gr^0\) are local homeomorphisms,
  compact open neighbourhoods in~\(\Gr^0\) provide compact open
  neighbourhoods in~\(\Bisp_n\) and then in~\(\Bisp_n/\Gr\).  The
  explicit description of the topology on the fibrewise one-point
  compactification of \(f\colon X\to Y\) in
  Lemma~\ref{lem:unified_space} shows that it is totally
  disconnected once \(X\) and~\(Y\) are so.  Thus the
  spaces~\(\Omega_{[0,n]}\) for \(n\in\N\) are totally disconnected.
  Finally, Tychonov's Theorem implies that the preimage of a compact
  open subset in~\(\Omega_{[0,n]}\) in the projective
  limit~\(\Omega_{[0,\infty)}\) is again compact open.  It follows
  that~\(\Omega_{[0,\infty)}\) is totally disconnected.  This is
  inherited by the closed subset~\(\Omega(\Reg)\).
\end{proof}

\section{Construction of the groupoid model}
\label{sec:construct_model}

\begin{definition}
  Let \(\IS = \IS(\Gr,\Bisp)\) be the inverse semigroup that is
  generated by the set \(\Bis = \Bis(\Gr) \sqcup \Bis(\Bisp)\) and the
  relations in \ref{en:action_from_theta1}
  and~\ref{en:action_from_theta2}, that is,
  \(\Theta(\Slice \Slice[V]) = \Theta(\Slice)\Theta(\Slice[V])\) if
  \(\Slice,\Slice[V]\in \Bis\) and \(\Slice\in \Bis(\Bisp)\) or
  \(\Slice[V]\in \Bis(\Bisp)\), and
  \(\Theta(\Slice_1)^*\Theta(\Slice_2) =
  \Theta(\braket{\Slice_1}{\Slice_2})\) for all
  \(\Slice_1,\Slice_2\in\Bis(\Bisp)\); here we
  write~\(\Theta(\Slice)\) for~\(\Slice\) viewed as a generator
  of~\(\IS(\Gr,\Bisp)\).
\end{definition}

To construct \(\IS\), we first let \(\mathrm{Free}(\Bis)\) be the free
inverse semigroup on the set \(\Bis\) (see
\cite{Higgins:Techniques_semigroup_theory}*{Theorem~1.1.10}).  Then we
let~\(\IS\) be the quotient of \(\mathrm{Free}(\Bis)\) by the smallest
congruence relation that contains the relations in
\ref{en:action_from_theta1} and~\ref{en:action_from_theta2}.  This is
an inverse semigroup by
\cite{Higgins:Techniques_semigroup_theory}*{Corollary~1.1.8}.  By
definition, there is a natural bijection between inverse semigroup
homomorphisms \(\IS\to S\) for another inverse semigroup~\(S\) and
maps \(\vartheta\colon \Bis\to S\) that satisfy
\ref{en:action_from_theta1} and~\ref{en:action_from_theta2}.

\begin{remark}
  The element \(\Theta(\Gr^0)\) is a unit element because
  of~\ref{en:action_from_theta1}.  The relations defining~\(\IS\)
  imply that \(\Bis(\Gr)\) is an inverse subsemigroup of~\(\IS\).
\end{remark}

\begin{lemma}
  \label{lem:IS_structure}
  Any element in~\(\IS\) is of the form
  \(a_1\dotsm a_n \cdot b_1^*\dotsm b_m^*\) or \(a_1\dotsm a_n\) or
  \(b_1^*\dotsm b_m^*\) with \(a_i,b_j\in \Bis(\Bisp)\), or just
  \(a\in\Bis(\Gr)\).
\end{lemma}

\begin{proof}
  Any element of~\(\IS\) may be written as a product of generators
  \(a,a^*\) with \(a\in \Bis\).  Since \(\Bis(\Gr)\) is an inverse
  subsemigroup, we do not need the letters~\(a^*\) for
  \(a\in \Bis(\Gr)\).  The defining relations allow us to shorten any
  word in the generators that contains an expression~\(a^* b\) with
  \(a,b\in\Bis\).  Therefore, every element in~\(\IS\) may be
  rewritten as \(a_1\dotsm a_n \cdot b_1^*\dotsm b_m^*\) with
  \(a_i,b_j\in \Bis\).  If \(n\ge 1\), then we may further shorten the
  word if some~\(a_i\) belongs to~\(\Bis(\Gr)\), and similarly for
  \(m\ge 1\) and for \(b_j\in\Bis(\Gr)\).  This reduces every word in
  the generators to the form claimed in the lemma.
\end{proof}

\begin{definition}
  Let~\(S\) be an inverse semigroup and let \(Y\) and~\(Z\) be
  topological spaces equipped with actions
  \(\vartheta_Y\colon S\to I(Y)\) and \(\vartheta_Z\colon S\to I(Z)\)
  of~\(S\) by partial homeomorphisms.  A continuous map
  \(f\colon Y\to Z\) is \emph{\(S\)\nb-equivariant} if
  \(f\circ \vartheta_Y(t) = \vartheta_Z(t)\circ f\) as partial maps
  for all \(t\in S\).
\end{definition}

In particular, the equality of
\(f\circ \vartheta_Y(t) = \vartheta_Z(t)\circ f\) says that both maps
have the same domain, that is,
\(f^{-1}(\Dom \vartheta_Z(t)) = \Dom \vartheta_Y(t)\).

\begin{proposition}
  \label{pro:action_through_IS}
  An action of \((\Gr,\Bisp,\Reg)\) on a topological space~\(Y\) is
  equivalent to an inverse semigroup homomorphism
  \(\IS\to I(Y)\) for which there exists an
  \(\IS\)-equivariant continuous map \(Y \to \Omega(\Reg)\).
\end{proposition}

\begin{proof}
  Lemma~\ref{lem:action_from_theta} shows that the maps
  \(\vartheta(\Slice)\) defined by an action of \((\Gr,\Bisp,\Reg)\)
  satisfy the relations that are needed for~\(\vartheta\) to induce a
  homomorphism \(\IS\to I(Y)\).  In addition,
  Lemma~\ref{lem:theta_gives_equivariant_invariant} implies easily
  that a map is \((\Gr,\Bisp)\)-equivariant if and only if it is
  equivariant for the induced actions of~\(\IS\).  In particular, the
  continuous map \(Y\to \Omega(\Reg)\) is \(\IS\)-equivariant.
  Conversely, assume that \(\vartheta\colon\IS\to I(Y)\) is a
  homomorphism and that there is an \(\IS\)\nb-equivariant map
  \(\varrho\colon Y\to \Omega(\Reg)\).
  Lemma~\ref{lem:action_from_theta} applied to the
  \((\Gr,\Bisp,\Reg)\)-action on~\(\Omega(\Reg)\) shows that the
  action~\(\vartheta_{\Omega(\Reg)}\) on~\(\Omega(\Reg)\) satisfies
  all the conditions in Lemma~\ref{lem:action_from_theta}.  The
  equivariance of~\(\varrho\) requires, among others,
  that~\(\varrho^{-1}\) maps the domain of
  \(\vartheta_{\Omega(\Reg)}(t)\) for \(t\in\IS\) to the domain of
  \(\vartheta_Y(t)\).  In particular, for the empty slice and the unit
  slice in~\(\Gr\), this gives~\ref{en:action_from_theta4}
  for~\(\vartheta_Y\).  Since taking preimages commutes with unions
  and the closure of a preimage is contained in the preimage of the
  closure, the action~\(\vartheta_Y\) also inherits the properties in
  \ref{en:action_from_theta3} and~\ref{en:action_from_theta6}
  from~\(\vartheta_{\Omega(\Reg)}\).
\end{proof}

We form the transformation groupoid \(\Omega \rtimes S\) for an action
of an inverse semigroup~\(S\) on a space~\(\Omega\) as
in~\cite{Exel:Inverse_combinatorial}, where it is called a groupoid of
germs.  The following proposition describes the transformation
groupoid through a universal property:

\begin{proposition}
  \label{pro:isg_trafo_universal}
  Let~\(S\) be an inverse semigroup and let~\(X\) be a space with an
  \(S\)\nb-action~\(\vartheta_X\) by partial homeomorphisms.
  Let~\(Y\) be a space.  There is a natural bijection between
  actions of the transformation groupoid \(X\rtimes S\) on~\(Y\) and
  pairs \((\vartheta_Y,f)\) consisting of an action~\(\vartheta_Y\)
  of~\(S\) on~\(Y\) and an \(S\)\nb-equivariant map
  \(f\colon Y\to X\).
\end{proposition}

\begin{proof}
  In this proof, we identify an idempotent partial homeomorphism of a
  space with its domain, which is an open subset.  Assume
  that~\(X\rtimes S\) acts on~\(Y\).  Let \(f\colon Y\to X\) be the
  anchor map of the action.  Any \(t\in S\) gives a slice~\(\Theta_t\)
  of~\(X\rtimes S\).  We define a partial map~\(\vartheta_{Y,t}\)
  of~\(Y\) with domain \(f^{-1}(\vartheta_X(t^* t))\) by
  \(\vartheta_{Y,t}(y) \defeq \gamma\cdot y\) for the unique
  \(\gamma\in\Theta_t\) with \(\s(\gamma) = f(y)\) in~\(X\).  This is
  a partial homeomorphism because~\(\vartheta_{Y,t^*}\) is a partial
  inverse for it.  The partial homeomorphisms~\(\vartheta_{Y,t}\) for
  \(t\in S\) define an action~\(\vartheta_Y\) of~\(S\) on~\(Y\) such
  that~\(f\) is \(S\)\nb-equivariant.  Conversely, let
  \((\vartheta_Y,f)\) be given.  We are going to define an action
  of~\(X\rtimes S\) with anchor map~\(f\).  Let
  \(\gamma\in X\rtimes S\) and \(y\in Y\) satisfy
  \(\s(\gamma) = f(y)\).  There is \(t\in S\) with
  \(\gamma\in \Theta_t\).  Since \(\s(\Theta_t)\) is the open subset
  in~\(X\) corresponding to the idempotent element \(t^* t \in S\),
  \(f(y) = \s(\gamma) \in \vartheta_X(t^* t)\).  Then
  \(y\in \vartheta_Y(t^* t)\) because~\(f\) is \(S\)\nb-equivariant.
  So \(\gamma\cdot y\defeq \vartheta_{Y,t}(y)\) is defined.  If
  \(\gamma\in \Theta_t\cap \Theta_u\), then there is an idempotent
  element \(e\in S\) with \(t e = u e\) and
  \(\s(\gamma)\in \vartheta_X(e)\).  Then \(y\in \vartheta_Y(e)\) as
  well and hence
  \(\vartheta_{Y,t}(y) = \vartheta_{Y,t e}(y) = \vartheta_{Y,u e}(y) =
  \vartheta_{Y,u}(y)\).  Thus \(\gamma\cdot y\) does not depend on the
  choice of \(t\in S\) with \(y\in\Theta_t\).  The multiplication map
  \((X\rtimes S) \times_{\s,X,f} Y \to Y\) is continuous because this
  holds on each \(\Theta_t\times_{s,X,f} Y\).  Routine computations
  show that the multiplication satisfies
  \(\gamma_1\cdot (\gamma_2\cdot y) = (\gamma_1\cdot \gamma_2)\cdot
  y\) for
  \((\gamma_1,\gamma_2,y) \in (X\rtimes S)\times_{\s,X,\rg} (X\rtimes
  S) \times_{\s,X,f} Y\) and \(1_{f(y)}\cdot y = y\) for all
  \(y\in Y\).  Thus the pair \((\vartheta_Y,f)\) gives rise to an
  action of \(X\rtimes S\).  The two constructions above are inverse
  to each other, so that we get the desired bijection.  Both are
  natural; that is, a map \(\varphi\colon Y\to Y'\) is equivariant
  with respect to two actions of \(X\rtimes S\) if and only if it is
  \(S\)\nb-equivariant and satisfies \(f'\circ \varphi = f\).
\end{proof}

\begin{theorem}
  \label{the:groupoid_model}
  The transformation groupoid \(\Omega(\Reg) \rtimes \IS(\Gr,\Bisp)\) is
  a groupoid model for \((\Gr,\Bisp,\Reg)\).
\end{theorem}

\begin{proof}
  This follows by combining the universal property of the
  transformation groupoid in Proposition~\ref{pro:isg_trafo_universal}
  with the description of actions of \((\Gr,\Bisp,\Reg)\) on a
  space~\(Y\) in Proposition~\ref{pro:action_through_IS}.
\end{proof}

\begin{proposition}
  \label{pro:universal_restricted_to_Reg}
  The restriction of
  \(\Mod \defeq \Omega_{[0,\infty)} \rtimes \IS(\Gr,\Bisp)\) to the
  open subset \(\Reg\subseteq \Gr^0\subseteq \Omega_{[0,\infty)}\) is
  isomorphic to~\(\Gr_\Reg\), and the orbit \(\Mod \cdot \Reg\) is
  \(\bigsqcup_{n\in\N} (\Bisp_n/\Gr)_\Reg\).  The restriction
  of~\(\Mod\) to \(\bigsqcup_{n\in\N} (\Bisp_n/\Gr)_\Reg\) is Morita
  equivalent to~\(\Gr_\Reg\).
\end{proposition}

\begin{proof}
  If \(a\in\Bis(\Bisp)\), then \(\vartheta_a\) maps
  \(\Omega_{[0,\infty)}\) into
  \(\Bisp\cdot \Omega_{[0,\infty)} = \Omega_{[1,\infty)}\), which is
  disjoint from \(\Reg \subseteq \Gr^0\subseteq \Omega_{[0,\infty)}\).
  Therefore, any element of~\(\IS(\Gr,\Bisp)\) of the form~\(s a^*\) with
  \(s \in \IS(\Gr,\Bisp)\) vanishes on~\(\Reg\).  By
  Lemma~\ref{lem:IS_structure}, the only remaining elements of~\(\IS(\Gr,\Bisp)\)
  are of the form \(a_1 \dotsm a_n\) for
  \(a_1,\dotsc,a_n \in \Bis(\Bisp)\) or \(a\in \Bis(\Gr)\).  An
  element of the first type maps~\(\Gr^0\) into
  \(\Bisp_n/\Gr \subseteq \Omega_{[n,\infty)}\), so that its codomain
  is disjoint from~\(\Reg\).  Therefore, an arrow in
  \(\Omega_{[0,\infty)} \rtimes \IS(\Gr,\Bisp)\) with range and source
  in~\(\Reg\) must belong to~\(\Theta_a\) for some \(a\in \Bis(\Gr)\).
  The germ relation for these arrows is also the same as for
  \(\Reg \rtimes \Bis(\Gr) = \Gr_\Reg\).  Thus the restriction
  of~\(\Mod\) to~\(\Reg\) is isomorphic to~\(\Gr_\Reg\).  Let
  \(x_{n+1}\in\Reg\) belong to the domain of a slice of the form
  \(a_1\dotsm a_n\) with \(a_1,\dotsc,a_n \in \Bis(\Bisp)\).  Let
  \(x_1,\dotsc,x_n\in \Bisp\) be such that \(x_j \in a_j\) and
  \(\s(x_j) = \rg(x_{j+1})\) for \(j=1,\dotsc,n+1\).  An induction
  on~\(n\) shows that
  \(\vartheta_{a_1\dotsm a_n}(x_{n+1}) = [x_1,\dotsc,x_n] \in
  \Bisp_n/\Gr \subseteq \Omega_{[0,\infty)}\).  The latter belongs to
  \((\Bisp_n/\Gr)_\Reg\).  Conversely, any element
  \([x_1,\dotsc,x_n]\) of \((\Bisp_n/\Gr)_\Reg\) is of this form by
  choosing \(x_{n+1} \defeq \s(x_n) \in \Reg\) and slices
  \(a_j\in\Bis(\Gr)\) with \(x_j \in a_j\).  Thus
  \(\Mod \cdot \Reg = \bigsqcup_{n\in\N} (\Bisp_n/\Gr)_\Reg\).
  Since~\(\Reg\) is open and the multiplication map in an \'etale
  groupoid is open, \(\Mod \cdot \Reg \subseteq \Omega_{[0,\infty)}\)
  is open.  The subset
  \[
    \Mod_\Reg \defeq \setgiven{\gamma\in\Mod}{\s(\gamma)\in\Reg}
    \subseteq \Mod
  \]
  is also open, and the left and right multiplication actions
  of~\(\Mod\) on its arrow space restrict to actions of
  \(\Mod|_{\Mod \cdot \Reg}\) and \(\Mod|_{\Reg}\) on \(\Mod_\Reg\) on
  the left and right, respectively.  These actions remain free and
  proper, and the range and source maps induce homeomorphisms
  \(\Mod_\Reg/\Mod|_{\Reg} \cong \Mod\cdot\Reg\),
  \(\Mod|_{\Mod\cdot\Reg} \backslash \Mod_\Reg \cong \Reg\).
  Thus~\(\Mod_\Reg\) is a Morita equivalence between the groupoids
  \(\Mod|_{\Mod \cdot \Reg}\) and \(\Mod|_{\Reg}\).
\end{proof}

\section{A universal property for transformation groupoid
  \texorpdfstring{$\Cst$}{C*}-algebras}
\label{sec:universal_groupoid_Cstar}

Cuntz--Pimsner algebras are defined by a universal property that
specifies their \Star{}homomorphisms to arbitrary
\(\Cst\)\nb-algebras.  We would like a similar universal property for
the \(\Cst\)\nb-algebra of our groupoid model.  In this section, we
formulate the relevant universal property, which applies to all
transformation groupoids of inverse semigroup actions.  Similar
results exist in the literature, but with extra assumptions, such as
assuming the groupoid to be Hausdorff
(\cite{Buss-Holkar-Meyer:Universal}*{Theorem~7.6}) or second countable
(\cite{Exel:Inverse_combinatorial}*{Proposition~9.8}).

Let~\(Y\) be a locally compact space and let~\(S\) be an inverse
semigroup.  Let \(\vartheta\colon S\to I(Y)\) be an inverse semigroup
homomorphism.  Let \(Y\rtimes S\) be the resulting transformation
groupoid.  Recall that this is an \'etale groupoid with unit
space~\(Y\).  For each \(a\in S\), let
\(\Theta_a\subseteq Y\rtimes S\) be the set of all arrows
\(a\colon y\to \vartheta_a(y)\) for \(y\in D_{a^* a}\).  These subsets
of the arrow space are slices that cover~\(Y\rtimes S\), and
\(\Theta_a \cap \Theta_b\) is the union of all \(\Theta_c\) for
\(c\le a,b\).  Extend functions in \(\Contc(\Theta_a)\) by zero to the
arrow space \(Y\rtimes S\) and take linear combinations.  The
resulting space of functions is a \Star{}algebra
\(\ContS(Y\rtimes S)\) for the usual convolution and involution.  The
groupoid \(\Cst\)\nb-algebra \(\Cst(Y\rtimes S)\) is defined as the
completion of \(\ContS(Y\rtimes S)\) in the maximal
\(\Cst\)\nb-seminorm.

We are going to describe nondegenerate \Star{}homomorphisms
\(\Cst(Y\rtimes S) \to \Comp(\Hilm)\) for a Hilbert module~\(\Hilm\)
over some \(\Cst\)\nb-algebra~\(D\).  This differs slightly
from~\cite{Buss-Holkar-Meyer:Universal} in that we only allow
representations by compact operators.  This makes it easier to prove
that certain operators are adjointable.  Our definition of a covariant
representation differs from the one
in~\cite{Buss-Holkar-Meyer:Universal} because we want to use
generators and relations for an inverse semigroup and thus want the
action of the inverse semigroup to be an action by partial maps.  On a
Hilbert module, the right partial maps are the following:

\begin{definition}
  A \emph{partial unitary} on~\(\Hilm\) is a unitary operator between
  two Hilbert submodules of~\(\Hilm\).
\end{definition}

The set of partial unitaries is a unital inverse subsemigroup of the
inverse semigroup of partial bijections on~\(\Hilm\), that is, the
identity map is a partial unitary, the partial inverse of a partial
unitary is a partial unitary, and the composite of two partial
unitaries (as partial maps) is again a partial unitary.

Since \(\Hilm = D\) is possible, our theory contains nondegenerate
\Star{}homomorphisms to any \(\Cst\)\nb-algebra~\(D\) as a special
case.  Conversely, since \(\Comp(\Hilm)\) is also a
\(\Cst\)\nb-algebra, what we do is equivalent to describing
nondegenerate \Star{}homomorphisms to all \(\Cst\)\nb-algebras.  If
\(\Hilm=D\) for a \(\Cst\)\nb-algebra, a partial unitary is a map
between two closed right ideals in~\(D\) that becomes unitary when
these are viewed as Hilbert modules.  This idea becomes much more
natural in the setting of Hilbert modules.

Before we define covariant representations, we prove a useful lemma
about the domains of the partial unitaries that occur in our covariant
representations.

\begin{lemma}
  \label{lem:domains_Hilm_apprid}
  Let~\(\Hilm\) be a Hilbert module over some \(\Cst\)\nb-algebra and
  let \(\varphi\colon \Cont_0(Y) \to \Bound(\Hilm)\) be a
  representation.  Let \(W\subseteq Y\) be an open subset and
  let~\((f_n)\) be a bounded approximate unit in \(\Cont_0(W)\).  Then
  a vector
  \(\xi\in\Hilm\) is of the form \(\varphi(g)\eta\) for some
  \(g\in \Cont_0(W)\), \(\eta\in\Hilm\) if and only if the net
  \(\varphi(f_n)\xi\) converges towards~\(\xi\).  If
  \(W_1,W_2\subseteq Y\)  are open, then
  \[
    \Cont_0(W_1)\Hilm \cap \Cont_0(W_2)\Hilm
    = \Cont_0(W_1 \cap W_2)\Hilm.
  \]
\end{lemma}

\begin{proof}
  The approximate unit property implies immediately that
  \(\varphi(f_n) \varphi(g) \eta\) converges towards
  \(\varphi(g) \eta\) for \(g\), \(\eta\) as above.  For the converse,
  we use the Cohen--Hewitt Factorisation Lemma.  It shows that the
  subset of \(\varphi(g)\eta\) is a closed linear subspace.  Since it
  contains \(\varphi(f_n) \xi\), it also contains the limit of
  \(\varphi(f_n)\xi\) if that limit exists.

  For the second statement, the inclusion
  \(\Cont_0(W_1)\Hilm \cap \Cont_0(W_2)\Hilm \supseteq \Cont_0(W_1
  \cap W_2)\Hilm\) is obvious.  For the converse, let \((f_{j,n})\) be
  bounded approximate units in \(\Cont_0(W_j)\) for \(j=1,2\),
  respectively.  Then \((f_{1,n} f_{2,n})\) is an approximate
  unit in \(\Cont_0(W_1 \cap W_2)\Hilm\).  If
  \(\xi\in\Cont_0(W_1)\Hilm \cap \Cont_0(W_2)\Hilm\), then
  \(\varphi(f_{2,n})\xi\) converges towards~\(\xi\), and then
  \(\varphi(f_{1,n})\varphi(f_{2,n})\xi\) converges towards~\(\xi\) as
  well because~\(f_{1,n}\) is bounded.  This implies
  \(\xi \in\Cont_0(W_1 \cap W_2)\Hilm\) by the first statement.
\end{proof}

\begin{definition}[compare~\cite{Buss-Holkar-Meyer:Universal}]
  \label{def:cov-rep}
  Let~\(Y\) be a locally compact Hausdorff space with a unital action
  of a unital inverse semigroup~\(S\) by partial homeomorphisms
  \(\vartheta_a\colon D_{a^*a}\congto D_{aa^*}\) between open subsets
  \(D_e\subseteq Y\) for idempotent \(e\in S\).  A \emph{compact
    covariant representation} of this system on a Hilbert
  \(D\)\nb-module~\(\Hilm\) consists of a nondegenerate representation
  \(\varphi\colon \Cont_0(Y)\to \Comp(\Hilm)\) and a family of
  unitaries \(\Rep_a\colon \Hilm_{a^*a}\congto \Hilm_{aa^*}\) for
  \(a\in S\), where \(\Hilm_e\defeq \varphi(\Cont_0(D_e))\Hilm\), that
  satisfy the following conditions:
  \begin{enumerate}[label=\textup{(\ref*{def:cov-rep}.\arabic*)},
    leftmargin=*,labelindent=0em]
  \item \label{en:cov-rep1}%
    \(\Rep_1 = \Id_{\Hilm}\) and \(\Rep_a \Rep_b = \Rep_{a b}\) for all
    \(a,b \in S\), where \(\Rep_a \Rep_b\) denotes the composition of
    partial maps on~\(\Hilm\);
\item \label{en:cov-rep4}%
    \(\Rep_a^*\varphi(f)\Rep_a = \varphi(f\circ\vartheta_a)\) as elements of
    \(\Comp(\Hilm_{a^* a})\), for all \(f\in \Cont_0(D_{a a^*})\).
  \end{enumerate}
\end{definition}

\begin{theorem}
  \label{the:covariant_rep}
  There is a natural bijection between compact covariant
  representations and nondegenerate \Star{}homomorphisms
  \(\Cst(Y\rtimes S) \to \Comp(\Hilm)\).  Naturality means two things:
  \begin{enumerate}[label=\textup{(\ref*{the:covariant_rep}.\arabic*)},
    leftmargin=*,labelindent=0em]
  \item \label{en:covariant_rep_1}%
    A bounded, possibly nonadjointable operator
    \(\xi\colon \Hilm_1 \to \Hilm_2\) intertwines two nondegenerate
    \Star{}homomorphisms
    \(\psi_j\colon \Cst(Y\rtimes S) \to \Comp(\Hilm_j)\) for \(j=1,2\),
    if and only if it intertwines the corresponding covariant
    representations \((\varphi_j,\Rep_j)\), that is, it intertwines
    the representations \(\varphi_1\) and~\(\varphi_2\) and
    \(\xi \Rep_{1,a} = \Rep_{2,a} \xi\) holds as maps
    \(\Hilm_{1,a^* a} \to \Hilm_{2,a a^*}\) for all \(a\in A\).
  \item \label{en:covariant_rep_2}%
    Let~\(\HilmF\) be a Hilbert \(D'\)\nb-module for another
    \(\Cst\)\nb-algebra with a nondegenerate left action of~\(D\) by
    compact operators.  Then a nondegenerate \Star{}\alb{}homomorphism
    \(\Cst(Y\rtimes S) \to \Comp(\Hilm)\) induces a nondegenerate
    \Star{}homomorphism
    \(\Cst(Y\rtimes S) \to \Comp(\Hilm \otimes_D \HilmF)\), and a
    compact covariant representation \((\varphi,\Rep)\) induces a
    compact covariant representation
    \((\varphi\otimes_D \Id_{\HilmF}, \Rep_a \otimes_D \Id_{\HilmF})\)
    on \(\Hilm \otimes_D \HilmF\).  The bijection between
    representations is compatible with these two constructions.
  \end{enumerate}
\end{theorem}

The first naturality property differs slightly from the one in
\cite{Buss-Holkar-Meyer:Universal}*{Theorem~3.23} in that we allow
nonadjointable bounded operators instead of nonadjointable isometries.
Michelle G\"obel~\cite{Goebel:Thesis} needed this extra generality in
her recent work about
induction theorems for \(\Cst\)\nb-hulls, and it is no extra effort to
prove the stronger property.

\begin{proof}
  First let \(\Rep_a\) and~\(\varphi\) be a covariant representation.
  Let \(a\in S\) and \(f\in \Contc(\Theta_a)\).  The source
  map~\(\s|_{\Theta_a}\) is a homeomorphism onto
  \(\Theta_a^* \Theta_a = \Theta_{a^* a} = D_{a^* a}\).  So we may
  transfer~\(f\) to a function
  \(f^\s\in \Contc(\s(\Theta_a)) = \Contc(D_{a^* a})\) by composing
  with the inverse of the homeomorphism
  \(\s|_{\Theta_a}\colon \Theta_a \congto \s(\Theta_a)\).  Define
  \(f^\rg\in \Contc(D_{a a^*})\) similarly.  We claim that
  \(\varphi(f^\s)\) belongs to the image of \(\Comp(\Hilm_{a^* a})\)
  in \(\Comp(\Hilm)\), which is a hereditary \(\Cst\)\nb-subalgebra.
  To see this, we use that products of functions
  \(f_1\cdot f_2\cdot f_3\) with \(f_j\in \Contc(D_{a^* a})\) are
  dense in \(\Cont_0(D_{a^* a})\) and that
  \(\varphi(f_1) \ket{\xi} \bra{\eta} \varphi(f_3) =
  \ket{\varphi(f_1)\xi}\bra{\varphi(f_3^*)\eta} \in \Comp(\Hilm_{a^*
    a})\) for all \(\xi,\eta\in \Hilm\).  Since~\(\Rep_a\) is an
  isometry from~\(\Hilm_{a^* a}\) to~\(\Hilm_{a a^*}\),
  \[
    \psi_a(f) \defeq \Rep_a \varphi(f)
  \]
  is a compact operator from~\(\Hilm_{a^* a}\) to~\(\Hilm_{a a^*}\)
  for all \(a\in S\) and \(f\in \Contc(\Theta_a)\).  The space of such
  compact operators is contained in \(\Comp(\Hilm)\).

  We claim that the linear maps
  \(\psi_a\colon \Contc(\Theta_a) \to \Comp(\Hilm)\) defined above
  piece together to a representation of \(\ContS(Y\rtimes S)\).
  First, we check that it is compatible with the involution and
  convolution for functions that live on single slices.  This requires
  a slightly more general formula.  Any \(f\in\Contc(\Theta_a)\) is a
  pointwise product \(f= f_1 f_2\) of two functions in
  \(\Contc(\Theta_a)\), for instance, take \(f_1 = \sqrt{\abs{f}}\)
  and \(f_2(g) = f(g)/f_1(g)\) if \(f(g)=0\) and~\(0\) otherwise.
  Then
  \(\Rep_a^* \varphi(f_1^\rg) \Rep_a = \varphi(f_1^\rg\circ
  \vartheta_a) = \varphi(f_1^\s)\) as operators
  \(\Hilm_{a^* a} \to \Hilm_{a^* a}\).  Since
  \(\varphi(f_2^\s) \in \Comp(\Hilm_{a^* a}, \Hilm_{a^* a})\), we
  compute
  \[
    \psi_a(f_1 f_2)
    = \Rep_a \varphi(f_1^\s f_2^\s)
    = \Rep_a \varphi(f_1^\s) \varphi(f_2^\s)
    = \Rep_a \Rep_a^* \varphi(f_1^\rg) \Rep_a \varphi(f_2^\s)
    = \varphi(f_1^\rg) \Rep_a \varphi(f_2^\s),
  \]
  where the last step uses that \(\varphi(f_1^\rg) \Rep_a \varphi(f_2^\s)\) is
  a compact operator to \(\Hilm_{a a^*}\) and that \(\Rep_a \Rep_a^*\) is
  the identity map on that Hilbert submodule.

  Since \(a\mapsto \Rep_a\) is an inverse semigroup homomorphism,
  \(\Rep_{a^*} = \Rep_a^*\).  Let \(f^\dagger(g) \defeq \conj{f(g)}\) denote
  the pointwise complex conjugate.  This is the adjoint in
  \(\Cont_0(Y\rtimes S)\), but not in \(\ContS(Y\rtimes S)\).  We compute
  \[
    \psi_a(f_1 f_2)^*
    = \varphi(f_2^\s)^* \Rep_a^* \varphi(f_1^\s)^*
    = \varphi( (f_2^\s)^\dagger) \Rep_{a^*} \varphi( (f_1^\rg)^\dagger)
    = \psi_{a^*}((f_1 f_2)^*)
  \]
  because \((f_1 f_2)^* \in \Contc(\Theta_{a^*})\) is given by
  \((f_1 f_2)^*(g) = (f_1 f_2)^\dagger(g^{-1})\) and the inverse flips
  range and source maps.  Thus \(\psi_a(f)^* = \psi_{a^*}(f^*)\) for
  all \(f\in \Contc(\Theta_a)\), \(a\in S\).

  Next, let \(a,b\in S\) and \(f_a\in \Contc(\Theta_a)\),
  \(f_b \in \Contc(\Theta_b)\).  Write them as pointwise products
  \(f_a = f_{a,1} f_{a,2}\) and \(f_b = f_{b,1} f_{b,2}\).  Recall
  that \(f_a * f_b \in \Contc(\Theta_{a b})\) with
  \((f_a * f_b)(g\cdot h) = f_a(g) f_b(h)\) for all \(g\in\Theta_a\),
  \(h\in\Theta_b\) with \(\s(g) = \rg(h)\).  Here
  \(\rg(g\cdot h) = \rg(g)\) and
  \(\s(g\cdot h) = \s(h) = \vartheta_b^*(\s(g))\).  Then
  \begin{multline*}
    \psi_a(f_a) \psi_b(f_b)
    = \varphi(f_{a,1}^\rg) \Rep_a \varphi(f_{a,2}^\s) \varphi(f_{b,1}^\rg) \Rep_b
    \varphi(f_{b,2}^\s)
    \\= \varphi(f_{a,1}^\rg) \Rep_a \varphi(f_{a,2}^\s \cdot f_{b,1}^\rg) \Rep_b
    \varphi(f_{b,2}^\s)
    = \varphi(f_{a,1}^\rg) \Rep_a \Rep_b  \varphi\Bigl(\bigl( (f_{a,2}^\s \cdot
    f_{b,1}^\rg)\circ \vartheta_b\bigr) \cdot f_{b,2}^\s\Bigr)
    \\= \varphi(f_{a,1}^\rg) \Rep_{a b}  \varphi\Bigl(\bigl( (f_{a,2}^\s \cdot
    f_{b,1}^\rg)\circ \vartheta_b\bigr) \cdot f_{b,2}^\s\Bigr)
    = \psi_{a b}(f_a * f_b).
  \end{multline*}
  because the functions \(f_{k,j}\) provide a suitable factorisation
  of \(f_a*f_b\).  Thus the family of maps \(\psi_a\) is
  multiplicative.

  Next, we claim that \(\psi_a\) and \(\psi_b\) coincide on
  \(\Contc(\Theta_a \cap \Theta_b)\) for all \(a,b\in S\).  To prove
  this, we use that \(\Theta_a \cap \Theta_b = \bigcup \Theta_c\) for
  \(c\in S\) with \(c \le a\) and \(c\le b\).  Thus each such~\(c\) is
  of the form \(c = a\cdot e = b\cdot e\) with \(e = c^* c\)
  idempotent.  Thus \(\Rep_b \Rep_e = \Rep_c = \Rep_a \Rep_e\).  Any
  idempotent partial unitary on~\(\Hilm\) is the identity map on a
  Hilbert submodule.  So the partial unitaries \(\Rep_a\)
  and~\(\Rep_b\) restrict to the same map on the image \(\Hilm_e\) of
  the idempotent~\(e\).  Then they also restrict to the same map on
  the sum of \(\Hilm_{c^* c}\) over all \(c \le a,b\).  This sum is
  dense in \(\Contc(\s(\Theta_a \cup \Theta_b))\Hilm\).  Therefore, if
  \(f\in \Contc(\Theta_a \cap \Theta_b)\), then \(\Rep_a\)
  and~\(\Rep_b\) restrict to the same map on \(\varphi(f^\s)\xi\) for
  all \(\xi\in\Hilm\).  It follows that
  \(\psi_a(f) = \Rep_a \varphi(f^\s)\) and
  \(\psi_b(f) = \Rep_b \varphi(f^\s)\) are equal in \(\Comp(\Hilm)\),
  as desired.

  Putting the maps \(\psi_a\colon \Contc(\Theta_a) \to
  \Comp(\Hilm)\) together gives a map
  \(\bigoplus_{a\in S} \Contc(\Theta_a) \to \Comp(\Hilm)\).  Since
  \(\psi_a\) and \(\psi_b\) coincide on
  \(\Contc(\Theta_a \cap \Theta_b)\) for all \(a,b\in S\),
  \cite{Buss-Meyer:Actions_groupoids}*{Proposition~B.2} implies that
  this map descends to a well defined map
  \(\psi\colon \ContS(Y\rtimes S) \to \Comp(\Hilm)\).  The
  computations above show that this is a \Star{}homomorphism.  Its
  restriction to the unit slice \(\Theta_1\) is~\(\varphi\)
  because~\(\Rep_1 = \Id_{\Hilm}\).  Therefore, \(\psi\) is
  nondegenerate.  It extends to a nondegenerate representation on the
  \(\Cst\)\nb-completion \(\Cst(Y\rtimes S)\) of
  \(\ContS(Y\rtimes S)\).

  Conversely, let \(\psi\colon \Cst(Y\rtimes S)\to \Comp(\Hilm)\) be
  a nondegenerate \Star{}homomorphism.  Restricting to the unit slice
  gives a nondegenerate \Star{}homomorphism
  \(\varphi\colon \Cont_0(Y) \to \Comp(\Hilm)\).  Fix \(a\in S\).  Let
  \(f_1,f_2\in\Contc(\Theta_a)\), \(\xi_1,\xi_2\in\Hilm\).  We compute
  \[
    \braket{\psi(f_1)\xi_1}{\psi(f_2) \xi_2}
    =  \braket{\xi_1}{\varphi(f_1^**f_2) \xi_2}
    =  \braket{\xi_1}{\varphi((f_1^\s)^\dagger \cdot f_2^\s) \xi_2}
    =  \braket{\varphi(f_1^\s)\xi_1}{\varphi(f_2^\s) \xi_2}.
  \]
  Thus the map \(\varphi(f^\s) \xi \mapsto \psi(f)\xi\) defines an
  isometric map from \(\varphi(\Contc(D_{a^* a}))\Hilm\) to
  \(\psi(\Contc(\Theta_a))\Hilm\).  Extending this to the closed
  linear spans gives a partial unitary~\(\Rep_a\) on~\(\Hilm\), whose
  domain is \(\Hilm_{a^* a} \defeq \varphi(\Cont_0(D_{a^* a}))\Hilm\)
  and whose codomain is the closure of
  \(\psi(\Contc(\Theta_a))\Hilm\).
  We claim that the latter is equal to~\(\Hilm_{a a^*}\).  To prove
  this, we use the chain of inclusions
  \[
    \psi(\Contc(\Theta_a))\Hilm
    \supseteq \psi(\Contc(\Theta_a)) \psi(\Contc(\Theta_a))^* \Hilm
    \supseteq \psi(\Contc(\Theta_a)) \psi(\Contc(\Theta_a))^*
    \psi(\Contc(\Theta_a))\Hilm.
  \]
  Since \(\Theta_a \Theta_a^* = D_{a a^*}\), the closed linear span of
  \(\psi(f_1) \psi(f_2)^* \xi\) with \(f_1,f_2\in \Contc(\Theta_a)\),
  \(\xi\in\Hilm\) is
  \(\Hilm_{a a^*} \defeq \varphi(\Cont_0(D_{a a^*}))\Hilm\).  Since
  any element of \(\Contc(\Theta_a)\) is a product
  \(f_1* f_2^* * f_3\) for \(f_j \in \Contc(\Theta_a)\), the sets
  \(\psi(\Contc(\Theta_a))\Hilm\) and
  \(\psi(\Contc(\Theta_a)) \psi(\Contc(\Theta_a))^*
  \psi(\Contc(\Theta_a))\Hilm\) are equal.  Since~\(\Hilm_{a a^*}\) is
  sandwiched between them, this proves the claim.  We have now
  attached partial unitaries
  \(\Rep_a\colon \Hilm_{a^* a} \to \Hilm_{a a^*}\) to~\(\psi\).  By
  construction, these satisfy \(\psi(f) = \Rep_a \varphi(f^\s)\) for
  all \(f\in\Contc(\Theta_a)\).

  If \(f_1,f_2\in \Contc(\Theta_a)\), then
  \(\varphi(f_1^\rg) \psi(f_2) = \psi(f_1 f_2)\), where
  \(f_1 f_2 \in \Contc(\Theta_a)\) means the pointwise product.  This
  gives the more symmetric formula
  \[
    \psi(f_1 f_2) = \varphi(f_1^\rg) \Rep_a \varphi(f_2^\s)
  \]
  for all \(f_1,f_2\in\Contc(\Theta_a)\).  So
  \(\psi(f) \in \Comp(\Hilm_{a^* a}, \Hilm_{a a^*}) \subseteq
  \Comp(\Hilm)\) for all \(f\in \Contc(\Theta_a)\).

  Since \(\psi(f_1 f_2) = \Rep_a \varphi(f_1^\s) \varphi(f_2^\s)\) as
  well, we get
  \(\Rep_a \varphi(f_1^\s) \varphi(f_2^\s) = \varphi(f_1^\rg) \Rep_a
  \varphi(f_2^\s)\).  This implies
  \(\Rep_a^* \Rep_a \varphi(f_1^\s) \varphi(f_2^\s) = \Rep_a^*
  \varphi(f_1^\rg) \Rep_a \varphi(f_2^\s)\).  Since
  \(\varphi(f_1^\s)\) takes values in \(\Hilm_{a^* a}\) where
  \(\Rep_a^* \Rep_a\) is the identity, we may leave out
  \(\Rep_a^* \Rep_a\).  Now we interpret \(\varphi(f_1^\s)\) and
  \(\Rep_a^* \varphi(f_1^\rg) \Rep_a\) as compact operators
  on~\(\Hilm_{a^* a}\).  Since~\(\Hilm_{a^* a}\) is the span of
  \(\varphi(f_2^\s)\xi\) for \(f_2\in\Contc(\Theta_a)\),
  \(\xi\in\Hilm\), the equality we have proven
  implies~\ref{en:cov-rep4}.

  Next we prove~\ref{en:cov-rep1} for
  \(a,b\in S\).  The operator \(\Rep_{a b}\) has the domain
  \(\Hilm_{b^* a^* a b}\).  This is contained in the
  domain~\(\Hilm_{b^* b}\) of~\(\Rep_b\), so that~\(\Rep_b\) is
  defined there.  The image of~\(\Rep_b\) is \(\Hilm_{b b^*}\).
  Lemma~\ref{lem:domains_Hilm_apprid} implies
  \[
    \Hilm_{b b^*} \cap \Hilm_{a^* a}
    = \varphi(\Cont_0(D_{b b^*}))\Hilm \cap  \varphi(\Cont_0(D_{a^* a}))\Hilm
    = \varphi(\Cont_0(D_{b b^*} \cap D_{a^* a}))\Hilm.
  \]
  The
  property~\ref{en:cov-rep4} that we
  have already proven shows that the \(\Rep_b\)\nb-preimage of
  \(\varphi(\Cont_0(D_{b b^*} \cap D_{a^* a}))\Hilm\) is
  \(\Cont_0(D_{b^* a^* a b}) \Hilm\) because the partial
  homeomorphism~\(\vartheta_b\) associated to~\(b\) restricts to a
  homeomorphism
  \[
    D_{b^* a^* a b}
    = D_{b^* b} \cap D_{b^* a^* a b}
    \congto D_{b b^*} \cap D_{a^* a}.
  \]
  Thus \(\Rep_a \Rep_b\) has the same domain as~\(\Rep_{a b}\).

  Let \(f_1\in \Contc(\Theta_a)\), \(f_2\in \Contc(\Theta_b)\) and
  \(f_3\in \Contc(D_{b^* a^* a b})\).  Then
  \(f_1*f_2 \in \Contc(\Theta_{a b})\) and
  \begin{multline*}
    \Rep_{a b} \varphi((f_1* f_2)^\s) \varphi(f_3)
    = \psi(f_1 * f_2) \varphi(f_3)
    \\= \psi(f_1) \psi(f_2) \varphi(f_3)
    = \Rep_a \varphi(f_1^\s) \Rep_b \varphi(f_2^\s) \varphi(f_3).
  \end{multline*}
  Here we use that the various functions belong to the domain where
  \(\psi_x(f) = \Rep_x \varphi(f^\s)\) holds for \(x\in \{a,b,a b\}\).
  The right hand side further simplifies to
  \(\Rep_a \Rep_b \varphi(f_1^\s\circ \vartheta_b^*) \varphi(f_2^\s)
  \varphi(f_3)\) because the domain of~\(f_3\) is small enough.  Now
  we may let \(f_1^\s\) and~\(f_2^\s\) run through approximate units
  in \(\Cont_0(D_{a^* a})\) and \(\Cont_0(D_{b b^*})\), respectively.
  Then \((f_1* f_2)^\s\) and
  \((f_1^\s\circ \vartheta_b^*)\cdot f_2^\s\) run through
  approximate units in \(\Cont_0(D_{b^* a^* a b})\).  Therefore, we get
  \(\Rep_a \Rep_b \varphi(f_3) =\Rep_{a b} \varphi(f_3)\) for all
  \(f_3\in \Contc(D_{b^* a^* a b})\).  Since both \(\Rep_a \Rep_b\)
  and~\(\Rep_{a b}\) are defined only on
  \(\varphi(\Contc(D_{b^* a^* a b}))\Hilm\), this implies
  \(\Rep_a \Rep_b=\Rep_{a b}\).  Thus any nondegenerate representation
  \(\psi\colon \Contc(Y\rtimes S) \to \Comp(\Hilm)\) gives a
  covariant pair \((\varphi,\Rep)\).

  When we start with a nondegenerate representation~\(\psi\), make a
  covariant pair \((\varphi,\Rep)\), and turn this into a
  nondegenerate representation again, we get back~\(\psi\) because
  \(\psi(f) = \Rep_a \varphi(f^\s)\) for all
  \(f\in \Contc(\Theta_a)\), \(a\in S\), and these functions generate
  \(\ContS(Y\rtimes S)\).  The converse is also true because the
  representation~\(\psi\) built from a covariant pair
  \((\varphi,\Rep)\) restricts to~\(\varphi\) on the unit slice and
  the property \(\psi(f) = \Rep_a \varphi(f^\s)\) for all
  \(f\in \Contc(\Theta_a)\), \(a\in S\) determines~\(\Rep_a\) as a
  partial unitary \(\Hilm_{a^* a} \to \Hilm_{a a^*}\).  So the two
  constructions are inverse to each other.

  Routine computations show that our constructions in both directions
  have the two naturality properties stated in the theorem.
\end{proof}

\section{Comparison to the relative Cuntz--Pimsner algebra}
\label{sec:compare_CP}

In this section, we are going to identify the groupoid
\(\Cst\)\nb-algebra of the groupoid model of \((\Gr,\Bisp,\Reg)\) with
the Cuntz--Pimsner algebra of \(\Cst(\Bisp)\) relative to
\(\Cst(\Gr_\Reg)\).  Since the groupoid model is unique up to isomorphism
by Proposition~\ref{pro:groupoid_model_unique}, we may work with the
specific groupoid model built in Theorem~\ref{the:groupoid_model}.
This is the transformation groupoid
\[
  \Mod \defeq \Omega(\Reg) \rtimes \IS(\Gr,\Bisp)
\]
for the inverse semigroup \(\IS(\Gr,\Bisp)\) acting on the
space~\(\Omega(\Reg)\) of Lemma~\ref{lem:Omega_Reg}.  Let
\(\vartheta_a\colon D_{a^* a} \to D_{a a^*}\) for
\(a\in \IS(\Gr,\Bisp)\) denote the partial homeomorphisms by which
\(\IS(\Gr,\Bisp)\) acts on \(\Omega(\Reg)\).

\begin{proposition}
  \label{pro:rep_model_Hilm}
  Let~\(\Hilm\) be a Hilbert module over a \(\Cst\)\nb-algebra~\(D\).
  There is a bijection between nondegenerate
  \Star{}homomorphisms \(\psi\colon \Cst(\Mod) \to \Comp(\Hilm)\) and
  pairs \((\varphi,\Rep)\) where
  \(\varphi\colon \Cont_0(\Omega(\Reg)) \to \Comp(\Hilm)\) is a
  nondegenerate \Star{}homomorphism and \(\Rep_a\) for
  \(a\in \Bis= \Bis(\Gr) \sqcup \Bis(\Bisp)\) are partial unitaries,
  for which the following hold:
  \begin{enumerate}[label=\textup{(\ref*{pro:rep_model_Hilm}.\arabic*)},
    leftmargin=*,labelindent=0em]
  \item \label{en:rep_model_Hilm_1}%
    \(\Rep_a \Rep_b = \Rep_{a b}\) if \(a,b\in \Bis\) and
    \(a\in \Bis(\Gr)\) or \(b\in \Bis(\Gr)\);
  \item \label{en:rep_model_Hilm_3}%
    the domain of~\(\Rep_a\) for \(a\in\Bis\)
    is~\(\Hilm_{\rg^{-1}(\s(a))}\), where \(\s(a) \subseteq \Gr^0\)
    and we define \(\Hilm_X \defeq \varphi(\Cont_0(X))\Hilm\) for an
    open subset \(X\subseteq \Omega(\Reg)\);
  \item \label{en:rep_model_Hilm_4}%
    the codomain of~\(\Rep_a\) is~\(\Hilm_{\pi^{-1}(\Qu(a))}\) for
    \(a\in \Bis(\Bisp)\), where
    \(\Qu\colon \Bisp\to \Bisp/\Gr \subseteq \Omega_{[0,1]}\) is the
    orbit space projection and
    \(\pi\colon \Omega(\Reg) \subseteq \Omega_{[0,\infty)}\to
    \Omega_{[0,1]}\) is the canonical projection;
  \item \label{en:rep_model_Hilm_5}%
    \(\Rep_a^* \varphi(f_1) \Rep_a \varphi(f_2) =
    \varphi\bigl((f_1\circ \vartheta_a)f_2\bigr)\) if
    \(a\in \Bis\), \(f_1\in \Cont_0(D_{a a^*})\),
    \(f_2\in \Cont_0(D_{a^* a})\);
  \item \label{en:rep_model_Hilm_2}%
    \(\Rep_a^* \Rep_b = \Rep_{\braket{a}{b}}\) for all \(a,b\in \Bis(\Bisp)\).
  \end{enumerate}
  This bijection has the two naturality properties in
  Theorem~\textup{\ref{the:covariant_rep}}.
  Condition~\ref{en:rep_model_Hilm_2} is redundant, that is, it
  follows from the other conditions.
\end{proposition}

\begin{proof}
  We first start with a nondegenerate \Star{}homomorphism
  \(\Cst(\Mod) \to \Comp(\Hilm)\) and construct a pair
  \((\varphi,\Rep)\) as in the proposition.  Recall that
  \(\Mod = \Omega(\Reg) \rtimes \IS(\Gr,\Bisp)\).
  Theorem~\ref{the:covariant_rep} provides a natural bijection between
  nondegenerate \Star{}homomorphisms \(\Cst(\Mod) \to \Comp(\Hilm)\)
  and covariant representations \((\varphi,\Rep)\) of the action of
  \(\IS(\Gr,\Bisp)\) on \(\Cont_0(\Omega(\Reg))\).  The inverse
  semigroup \(\IS(\Gr,\Bisp)\) is the universal inverse semigroup
  generated by the set~\(\Bis\) subject to the relations in
  \ref{en:rep_model_Hilm_1} and~\ref{en:rep_model_Hilm_2}.  In
  particular, a representation~\(\Rep\) of~\(\IS(\Gr,\Bisp)\) provides
  partial unitaries~\(\Rep_a\) for all \(a\in\Bis\) that satisfy the
  relations in \ref{en:rep_model_Hilm_1}
  and~\ref{en:rep_model_Hilm_2}.  In addition, the domain and codomain
  of~\(\Rep_a\) in a covariant representation are \(\Hilm_{a^* a}\)
  and~\(\Hilm_{a a^*}\), respectively.  These submodules involve the
  domains and codomains of the partial homeomorphisms induced by the
  action on~\(\Omega(\Reg)\), and Lemma~\ref{lem:slice_acts} allows us to
  compute these.  Putting things together, it follows that the domain
  of~\(\Rep_a\) for \(a\in\Bis\) is
  \(\Hilm_{a^* a} = \Hilm_{\rg^{-1}(\s(a))}\) as required
  in~\ref{en:rep_model_Hilm_3} and the codomain of~\(\Rep_a\) for
  \(a\in\Bis(\Bisp)\) is \(\Hilm_{a a^*} = \Hilm_{\pi^{-1}(\Qu(a))}\)
  as required in~\ref{en:rep_model_Hilm_4}.  Thus a nondegenerate
  \Star{}homomorphism \(\Cst(\Mod) \to \Comp(\Hilm)\) gives rise to a
  pair \((\varphi,\Rep)\) satisfying the conditions
  \ref{en:rep_model_Hilm_1}--\ref{en:rep_model_Hilm_2}.

  Conversely, assume that such a pair \((\varphi,\Rep)\) is given.
  The conditions \ref{en:rep_model_Hilm_1}
  and~\ref{en:rep_model_Hilm_2} ensure that~\(\Rep\) extends to a
  homomorphism~\(\bar\Rep\) from~\(\IS(\Gr,\Bisp)\) to the inverse
  semigroup of partial unitaries on~\(\Hilm\).  We claim that
  \((\varphi,\bar\Rep)\) is a covariant representation of the action
  of~\(\IS(\Gr,\Bisp)\) on~\(\Omega(\Reg)\).  Thus \((\varphi,\Rep)\)
  induces a nondegenerate \Star{}homomorphism
  \(\Cst(\Mod) \to \Comp(\Hilm)\).  The property
  \(\Rep_a \Rep_b = \Rep_{a b}\) for \(a,b\in\IS(\Gr,\Bisp)\) in
  \ref{en:cov-rep1} is already built in.  The unit element of
  \(\IS(\Gr,\Bisp)\) is the unit slice in~\(\Gr\).  This is idempotent
  and~\ref{en:rep_model_Hilm_3} says that its domain is all
  of~\(\Hilm\).  So \(\Rep_1 = \Id_{\Hilm}\), as required
  in~\ref{en:rep_model_Hilm_3}.

  We call \(a\in \IS(\Gr,\Bisp)\) \emph{good} if the domain and
  codomain of~\(\Rep_a\) are \(\Hilm_{a^* a}\) and~\(\Hilm_{a a^*}\),
  respectively, and
  \(\Rep_a^* \varphi(f) \Rep_a = \varphi(f\circ \vartheta_a)\) in
  \(\Comp(\Hilm_{a^* a})\) for all \(f\in \Cont_0(D_{a a^*})\).  The
  pair \((\varphi,\bar\Rep)\) is a covariant representation if and
  only if all elements of~\(\IS(\Gr,\Bisp)\) are good.  We are going
  to prove this by showing that the generators \(a\in\Bis\) are good
  and that being good is hereditary for products and adjoints
  in~\(\IS(\Gr,\Bisp)\).  That is, the good elements form an inverse
  subsemigroup of~\(\IS(\Gr,\Bisp)\) that contains all the generators,
  and this forces it to be all of~\(\IS(\Gr,\Bisp)\).

  We have already discussed that the domain of~\(\Rep_a\) for
  \(a\in\Bis\) is~\(\Hilm_{a^* a}\) if and only if
  \ref{en:rep_model_Hilm_3} is satisfied.  Since \(a^*\in\Bis\) for
  \(a\in\Bis(\Gr)\), this also gives the correct codomain for
  \(a\in \Bis(\Gr)\).  If \(a\in\Bis(\Bisp)\),
  then~\ref{en:rep_model_Hilm_4} ensures that the codomain
  of~\(\Rep_a\) is~\(\Hilm_{a a^*}\).  The covariance
  condition~\ref{en:cov-rep4} identifies
  \(\Rep_a^*\varphi(f_1)\Rep_a\) and \(\varphi(f_1\circ\vartheta_a)\)
  in \(\Comp(\Hilm_{a^* a})\) for all \(f_1\in \Cont_0(D_{a a^*})\).
  This is equivalent to
  \(\Rep_a^*\varphi(f_1)\Rep_a \varphi(f_2) =
  \varphi(f_1\circ\vartheta_a) \varphi(f_2)\) as elements of
  \(\Comp(\Hilm)\) for all \(f_1\in \Cont_0(D_{a a^*})\) and
  \(f_2\in \Cont_0(D_{a^* a})\) by the definition
  of~\(\Hilm_{a^* a}\); this is what is required
  in~\ref{en:rep_model_Hilm_5} for \(a\in\Bis\).  This finishes the
  proof that all elements of~\(\Bis\) are good.

  It is easy to see that~\(a^*\) is good if~\(a\) is.  To show that
  all elements of \(\IS(\Gr,\Bisp)\) are good, it remains to prove
  that \(a b\) is good if \(a\) and~\(b\) are good.  So let \(a,b\) be
  good.  The intersection of the image of~\(\Rep_b\) with the domain
  of~\(\Rep_a\) is
  \[
    \Hilm_{b b^*}\cap \Hilm_{a^* a}
    = \Cont_0(D_{b b^*})\Hilm \cap \Cont_0(D_{a^* a})\Hilm
    = \Cont_0(D_{b b^*} \cap D_{a^* a})\Hilm
  \]
  by Lemma~\ref{lem:domains_Hilm_apprid}.  If
  \(f\in\Cont_0(D_{b b^*} \cap D_{a^* a})\), then
  \(\Rep_b^* \varphi(f) \Rep_b = \varphi(f\circ \vartheta_b)\) as
  operators on~\(\Hilm_{b^* b}\) because~\(b\) is good.  So the
  domain of \(\Rep_{a b} = \Rep_a \Rep_b\) is
  \(\Cont_0(\vartheta_b^{-1}(D_{b b^*} \cap D_{a^* a}))\Hilm =
  \Cont_0(D_{(a b)^* a b}) \Hilm = \Hilm_{(a b)^* ab}\) as desired.
  The codomain is
  \[
    \Rep_a\bigl( \Cont_0(D_{b b^*} \cap D_{a^* a})\Hilm\bigr)
    = \Rep_a\varphi\bigl(\Cont_0(D_{b b^*})\bigr) \Hilm_{a^* a}
    = \varphi\bigl(\Cont_0(\vartheta_a(D_{b b^*}))\bigr)\Hilm_{a a^*}
    = \Hilm_{a b b^* a^*}
  \]
  as needed.  Finally, the following equalities hold as operators
  on~\(\Hilm_{(a b)^* ab}\):
  \[
    \Rep_{a b}^* \varphi(f) \Rep_{a b}
    = \Rep_b^* \Rep_a^*  \varphi(f) \Rep_a \Rep_b
    = \Rep_b^* \varphi(f\circ \vartheta_a) \Rep_b
    = \varphi(f\circ\vartheta_a \circ \vartheta_b)
    = \varphi(f\circ\vartheta_{a b})
  \]
  for all \(f\in \Cont_0(D_{a b (a b)^*})\).  So~\(a b\) is good as
  desired.

  Finally, we show that condition~\ref{en:rep_model_Hilm_2} is
  redundant.  Let \(a,b\in\Bis(\Bisp)\).  First, we claim
  that~\ref{en:rep_model_Hilm_2} follows if
  \begin{equation}
    \label{eq:rep_model_Hilm_2_variant}
    \Rep_a \Rep_a^* \Rep_b = \Rep_a \Rep_{\braket{a}{b}}.
  \end{equation}
  The codomain of \(\Rep_{\braket{a}{b}}\) is the same as the domain
  of \(\Rep_{\braket{a}{b}}^* = \Rep_{\braket{b}{a}}\), which
  is~\(\Hilm_{\rg^{-1}(\s(\braket{b}{a}))}\)
  by~\ref{en:rep_model_Hilm_3}.  This is contained
  in~\(\Hilm_{\rg^{-1}(\s(a))}\), which is the domain of \(\Rep_a\)
  by~\ref{en:rep_model_Hilm_3}.  This is where \(\Rep_a^*\Rep_a\) is
  the identity map.  So
  \(\Rep_a^* \Rep_a \Rep_{\braket{a}{b}} = \Rep_{\braket{a}{b}}\).  In
  addition, \(\Rep_a^* \Rep_a \Rep_a^* = \Rep_a^*\).  Therefore,
  \eqref{eq:rep_model_Hilm_2_variant} implies
  \[
    \Rep_a^* \Rep_b
    = \Rep_a^* \Rep_a \Rep_a^* \Rep_b
    = \Rep_a^* \Rep_a \Rep_{\braket{a}{b}}
    = \Rep_{\braket{a}{b}},
  \]
  which is~\ref{en:rep_model_Hilm_2}.  So it suffices to prove
  \eqref{eq:rep_model_Hilm_2_variant}.  Now
  \(\Rep_a \Rep_{\braket{a}{b}} = \Rep_{a \braket{a}{b}}\)
  by~\ref{en:rep_model_Hilm_1}.  Since \(x \braket{x}{y} = y\) for all
  \(x,y\in \Bisp\) with \(\Qu(x) = \Qu(y)\), the slice
  \(a\braket{a}{b}\) is contained in~\(b\).  Namely, it consists of
  all \(y\in b\) with \(\Qu(y) \in \Qu(a)\).  Let
  \(W = \s(a\braket{a}{b})\).  This is an open subset of~\(\Gr^0\) and
  thus an idempotent element of \(\Bis(\Gr)\).  The computation above
  implies \(a\braket{a}{b} = b W\).  So
  \(\Rep_a \Rep_{\braket{a}{b}} = \Rep_b\Rep_W\).
  By~\ref{en:rep_model_Hilm_1}, \(\Rep_W\) is the identity map on its
  domain, which is \(\Hilm_{\rg^{-1}(W)}\)
  by~\ref{en:rep_model_Hilm_3}.  So \(\Rep_b\Rep_W\) is the
  restriction of~\(\Rep_b\) to~\(\Hilm_{\rg^{-1}(W)}\).  Next
  \(\Rep_a\Rep_a^*\) is the identity map on the codomain
  of~\(\Rep_a\), which is \(\Hilm_{\pi^{-1}(\Qu(a))}\)
  by~\ref{en:rep_model_Hilm_4}.  Let \((f_n)\) be an approximate unit
  in
  \(\Cont_0(\Qu(a)) \subseteq \Cont_0(\Bisp/\Gr) \subseteq
  \Cont_0(\Omega_{[0,1]})\).  Then \((f_n\circ \pi)\) is an
  approximate unit in \(\Cont_0(\pi^{-1}(\Qu(a)))\).
  Lemma~\ref{lem:domains_Hilm_apprid} implies that
  \(\Rep_a \Rep_a^* \Rep_b\) is defined at \(\xi\in\Hilm\) if and only
  if \(\xi\in\Hilm_{\s(b)}\) and
  \(\lim \varphi(f_n\circ \pi) \Rep_b(\xi) = \Rep_b(\xi)\).  If
  \(\xi\in\Hilm_{\s(b)}\), then \(\xi = \varphi(f_b) \xi'\) for some
  \(f_b\in \Cont_0(\s(b))\), \(\xi\in\Hilm\) by the Cohen--Hewitt
  Factorisation Theorem.  So~\ref{en:rep_model_Hilm_5} implies
  \[
    \varphi(f_n\circ \pi) \Rep_b(\xi)
    = \Rep_b \varphi(f_n\circ \pi\circ \vartheta_b) \varphi(f_b) \xi'
    = \Rep_b \varphi(f_n\circ\pi\circ \vartheta_b) \xi.
  \]
  If \(\omega \in \Omega(\Reg)\), then
  \(\pi(\vartheta_b(\omega)) \in \Omega_{[0,1]}\) is defined if and
  only if \(\rg(\omega) \in \Gr^0\) is in \(\s(b)\), and then it is
  \([x]\in \Bisp/\Gr \subseteq \Omega_{[0,1]}\) for the unique
  \(x\in b\) with \(\s(x) = \rg(\omega)\).  It follows that
  \(f_n\circ\pi \circ \vartheta_b(\omega)\) converges to~\(1\) if
  \(\rg(\omega) \in W\) and vanishes otherwise, with
  \(W\subseteq \Gr^0\) as above.  So we may write
  \(f_n\circ \pi \circ \vartheta_b(\omega) = f'_n\circ \rg(\omega)\)
  with an approximate unit~\(f'_n\) for the ideal \(\Cont_0(W)\)
  in~\(\Cont_0(\Gr^0)\).  Using Lemma~\ref{lem:domains_Hilm_apprid}
  once again, we conclude that the domain of
  \(\Rep_a \Rep_a^* \Rep_b\) is also equal to~\(\Hilm_{\rg^{-1}(W)}\).
  Since \(\Rep_a\Rep_a^*\) is idempotent, \(\Rep_a \Rep_a^* \Rep_b\)
  is the restriction of~\(\Rep_b\) to this domain.  This is equal to
  \(\Rep_a \Rep_{\braket{a}{b}} = \Rep_{a \braket{a}{b}} = \Rep_b \Rep_W\).
  So~\eqref{eq:rep_model_Hilm_2_variant} and
  thus~\ref{en:rep_model_Hilm_2} follow from the other conditions in
  the proposition.
\end{proof}

\begin{remark}
  \label{rem:relation_IS_redundant}
  The last sentence in Proposition~\ref{pro:rep_model_Hilm} implies
  that the transformation groupoid
  \(\Omega(\Reg)\rtimes \IS(\Gr,\Bisp)\) does not change if we drop
  the relation
  \(\Theta(\Slice_1)^*\Theta(\Slice_2) =
  \Theta(\braket{\Slice_1}{\Slice_2})\) for
  \(\Slice_1,\Slice_2\in\Bis(\Bisp)\) in the definition of
  \(\IS(\Gr,\Bisp)\).
\end{remark}

The next proposition describes nondegenerate representations of the
Cuntz--Pimsner algebra
\(\mathcal{O} \defeq \mathcal{O}_{\Cst(\Bisp),\Cst(\Gr_\Reg)}\) in
\(\Comp(\Hilm)\) in a somewhat similar fashion.  It uses the partial
map~\(\varrho_a\) on~\(\Gr^0\) induced by a slice \(a\in\Bis(\Bisp)\).
This is defined on the open subset \(\s(a)\subseteq\Gr^0\) and maps
\(\s(x)\) for \(x\in a\) to \(\rg(x) \in \rg(a) \subseteq \Gr^0\).
Since~\(\rg\) is not a local homeomorphism, this map
\(\varrho_a\colon \s(a) \to \rg(a)\) is not a homeomorphism.  Instead,
it is the composite of the homeomorphism
\(\vartheta_a\colon \s(a) \to \Qu(a) \subseteq \Bisp/\Gr\) with
\(\rg_*\colon \Bisp/\Gr \to \Gr^0\).  If \(f\in \Cont_0(\Gr^0)\), then
\(f\circ \varrho_a \in \Contb(\s(a))\), and it need not belong to
\(\Cont_0(\s(a))\) unless the correspondence~\(\Bisp\) is proper.  If
\(f_1\in \Cont_0(\Gr^0)\), \(f_2\in \Cont_0(\s(a))\), then
\((f_1\circ \varrho_a) \cdot f_2 \in \Cont_0(\s(a)) \subseteq
\Cont_0(\Gr^0)\), although \(f_1\circ \varrho_a\) is only defined
on~\(\s(a)\).

\begin{proposition}
  \label{pro:rep_CP_Hilm}
  Let~\(\Hilm\) be a Hilbert module over a \(\Cst\)\nb-algebra~\(D\).
  There is a natural bijection between nondegenerate
  \Star{}homomorphisms
  \(\psi\colon \mathcal{O}_{\Cst(\Bisp),\Cst(\Gr_\Reg)} \to
  \Comp(\Hilm)\) and pairs \((\varphi_0,\Rep)\) where
  \(\varphi_0\colon \Cont_0(\Gr^0) \to \Comp(\Hilm)\) is a
  nondegenerate \Star{}homomorphism and \(\Rep_a\) for
  \(a\in \Bis= \Bis(\Gr) \sqcup \Bis(\Bisp)\) are partial unitaries, such
  that the following hold:
  \begin{enumerate}[label=\textup{(\ref*{pro:rep_CP_Hilm}.\arabic*)},
    leftmargin=*,labelindent=0em]
  \item \label{en:rep_CP_Hilm_1}%
    \(\Rep_a \Rep_b = \Rep_{a b}\) if \(a,b\in \Bis\) and
    \(a\in \Bis(\Gr)\) or \(b\in \Bis(\Gr)\);
  \item \label{en:rep_CP_Hilm_3}%
    the domain of~\(\Rep_a\) for \(a\in\Bis\) is~\(\Hilm_{\s(a)}\),
    where \(\s(a) \subseteq \Gr^0\) and we define
    \(\Hilm_W \defeq \varphi_0(\Cont_0(W))\Hilm\) for an open subset
    \(W\subseteq \Gr^0\);
  \item \label{en:rep_CP_Hilm_4}%
    \(\Hilm_\Reg\) is contained in the closed linear
    span of the images of~\(\Rep_a\) for \(a\in\Bis(\Bisp)\);
  \item \label{en:rep_CP_Hilm_5}%
    \(\Rep_a^* \varphi_0(f_1) \Rep_a \varphi_0(f_2) = \varphi_0\bigl(
    (f_1\circ \varrho_a) \cdot f_2\bigr)\) if \(a\in \Bis\),
    \(f_1\in \Cont_0(\Gr^0)\), \(f_2\in \Cont_0(\s(a))\);
  \item \label{en:rep_CP_Hilm_2}%
    let \(a,b\in \Bis(\Bisp)\), \(f_j \in \Cont_0(j)\),
    \(\xi_j \in \Hilm\) for \(j=a,b\); then
    \[
      \braket{\Rep_a \varphi_0(f_a^\s) \xi_a}{\Rep_b \varphi_0(f_b^\s)
        \xi_b}
      = \braket{\xi_a}{\Rep_{\braket{a}{b}}
        \varphi_0\bigl((f_a^\s\circ \vartheta_{\braket{a}{b}})^\dagger
        f_b^\s\bigr) \xi_b};
    \]
    here~\(\dagger\) denotes pointwise complex conjugation and
    \(f_j^\s \in \Contc(\s(j))\) is \(f_j \circ (\s|_j)^{-1}\) for the
    homeomorphism \((\s|_j)^{-1} \colon \s(j) \congto j\) for \(j=a,b\).
  \end{enumerate}
  This bijection has the two naturality properties in
  Theorem~\textup{\ref{the:covariant_rep}}.
\end{proposition}

\begin{proof}
  The universal property says that a \Star{}homomorphism
  \(\mathcal{O} \to \Comp(\Hilm)\) is equivalent to a Toeplitz
  representation \(\Cst(\Bisp)\to\Comp(\Hilm)\) that is Cuntz--Pimsner
  covariant on the ideal \(\Cst(\Gr_\Reg)\).  Here a Toeplitz
  representation is a pair of maps
  \(\psi\colon \Cst(\Gr) \to \Comp(\Hilm)\) and
  \(L\colon \Cst(\Bisp) \to \Comp(\Hilm)\), where~\(\psi\) is a
  \Star{}homomorphism and~\(L\) is a linear map that satisfies
  \(\psi(f) L(\xi) = L(f*\xi)\), \(L(\xi) \psi(f) = L(\xi*f)\), and
  \(L(\xi_1)^* L(\xi_2) = \psi(\braket{\xi_1}{\xi_2})\) for
  \(f\in\Cst(\Gr)\), \(\xi,\xi_1,\xi_2\in \Cst(\Bisp)\).  By
  \cite{Meyer-Sehnem:Bicategorical_Pimsner}*{Proposition~2.15}, this
  Toeplitz representation is Cuntz--Pimsner covariant on the ideal
  \(\Cst(\Gr_\Reg) \subseteq \Cst(\Gr)\) if and only if
  \(\psi(\Cst(\Gr_\Reg)) \Hilm \subseteq L(\Cst(\Bisp))\Hilm\).  Since
  \(\Cont_0(\Reg)\) is a nondegenerate \(\Cst\)\nb-subalgebra of
  \(\Cst(\Gr_\Reg)\), this is equivalent to
  \(\psi(\Cont_0(\Reg)) \Hilm \subseteq L(\Cst(\Bisp))\Hilm\).  In
  addition, the \Star{}homomorphism \(\mathcal{O}\to\Comp(\Hilm)\) is
  nondegenerate if and only if~\(\psi\) is nondegenerate.

  We may write~\(\Gr\) as the inverse semigroup transformation
  groupoid \(\Gr^0 \rtimes \Bis(\Gr)\).  By
  Theorem~\ref{the:covariant_rep}, a nondegenerate \Star{}homomorphism
  \(\psi\colon \Cst(\Gr) \to \Comp(\Hilm)\) is equivalent to a
  covariant representation \((\varphi_0,\Rep)\) of
  \(\Gr^0 \rtimes \Bis(\Gr)\) as in Definition~\ref{def:cov-rep}.
  This consists of a nondegenerate \Star{}homomorphism
  \(\varphi_0\colon \Cont_0(\Gr^0) \to \Comp(\Hilm)\) and partial
  unitaries \(\Rep_a\colon \Hilm_{\s(a)} \to \Hilm_{\rg(a)}\) for
  \(a\in \Bis(\Gr)\).  The conditions in Definition~\ref{def:cov-rep}
  are equivalent to the conditions in \ref{en:rep_CP_Hilm_1},
  \ref{en:rep_CP_Hilm_3} and~\ref{en:rep_CP_Hilm_5} for
  \(a\in \Bis(\Gr)\); compare the proof of
  Theorem~\ref{the:covariant_rep} for why we may add the
  factor~\(\varphi_0(f_2)\) in~\ref{en:rep_CP_Hilm_5}.

  Next we relate the map \(L\colon \Cst(\Bisp) \to \Comp(\Hilm)\) in a
  Toeplitz representation to a family of partial unitaries~\(\Rep_a\)
  for \(a\in\Bis(\Bisp)\).  By the construction
  in~\cite{Antunes-Ko-Meyer:Groupoid_correspondences}, \(\Cst(\Bisp)\)
  is defined as the Hilbert module completion of the space of
  functions \(\ContS(\Bisp)\), which is the closed linear span of
  functions in \(\Contc(a)\) for slices \(a\subseteq \Bisp\), extended
  by zero outside~\(a\).  So~\(L\) is determined uniquely by its
  restrictions to \(\Contc(a) \subseteq \Cst(\Bisp)\) for all
  \(a\in \Bis(\Bisp)\).  Given~\(\varphi_0\) and~\(\Rep_a\) as in the
  statement, we define a map
  \[
    L\colon \bigoplus_{a\in\Bis(\Bisp)} \Contc(a) \to
    \Comp(\Hilm),\qquad
    L\bigl((f_a)_{a\in\Bis(\Bisp)}\bigr) \defeq \sum \Rep_a
    \varphi_0(f_a^\s).
  \]
  Condition~\ref{en:rep_CP_Hilm_2} is equivalent to
  \(\braket{L(f_a) \xi_a}{L(f_b)\xi_b} =
  \braket{\xi_a}{\psi(\braket{f_a}{f_b})\xi_b}\) for
  \(f_a\in \Contc(a)\), \(f_b\in \Contc(b)\) for \(a,b\in\Bis(\Bisp)\)
  and \(\xi_a,\xi_b\in \Hilm\) because
  \(\varphi_0(f_a^\s)\xi_a \in \Hilm_{\s(a)}\),
  \(\varphi_0(f_b^\s)\xi_b \in \Hilm_{\s(b)}\), and
  \(\psi(\braket{f_a}{f_b}) = \Rep_{\braket{a}{b}} \varphi_0((f_a^**
  f_b)^\s) = \Rep_{\braket{a}{b}} \varphi_0\bigl((f_a^\s\circ
  \vartheta_{\braket{a}{b}})^\dagger f_b^\s\bigr)\).
  So~\ref{en:rep_CP_Hilm_2} is equivalent to
  \(L(f_a)^* L(f_b) = \psi(\braket{f_a}{f_b})\).  By linearity, we get
  \(L(f_1)^* L(f_2) = \psi(\braket{f_1}{f_2})\) for all
  \(f_1,f_2\in \bigoplus \Contc(a)\) with the scalar product that
  defines \(\Cst(\Bisp)\).  So~\(L\) factors through an isometric map
  \(\Cst(\Bisp) \to \Comp(\Hilm)\), which we also denote by~\(L\).

  Let \(a\in\Bis(\Gr)\), \(b\in \Bis(\Bisp)\), \(f_a\in\Contc(a)\),
  \(f_b\in\Contc(b)\).  The product \(f_a*f_b\) in \(\ContS(\Bisp)\)
  is defined in
  \cite{Antunes-Ko-Meyer:Groupoid_correspondences}*{Equation~(7.3)}.
  It is supported on the slice \(a b\subseteq \Bisp\) and has the
  value \(f_a(g_a) f_b(x_b)\) at \(g_a x_b\) if \(g_a\in a\),
  \(x_b\in b\) are such that \(\s(g_a) = \rg(x_b)\).  Thus
  \((f_a * f_b)^\s = (f_a^\s \circ \varrho_b) f_b^\s\).  Then
  \begin{multline*}
    \psi(f_a) L(f_b)
    = \Rep_a \varphi_0(f_a^\s) \Rep_b \varphi_0(f_b^\s)
    = \Rep_a \Rep_b \varphi_0((f_a^\s \circ \varrho_b)f_b^\s)
    \\= \Rep_{a b} \varphi_0\bigl((f_a*f_b)^\s\bigr)
    = L(f_a*f_b)
  \end{multline*}
  by~\ref{en:rep_CP_Hilm_5} and~\ref{en:rep_CP_Hilm_1}.  This implies
  that~\(L\) is a left \(\Cst(\Gr)\)-module map.  A similar
  computation shows that it is a right \(\Cst(\Gr)\)-module map.  So
  \((\psi,L)\) is a Toeplitz representation.  We have already argued
  above that \((\psi,L)\) is Cuntz--Pimsner covariant on
  \(\Cst(\Gr_\Reg)\) if and only if
  \(\Hilm_\Reg \subseteq L(\Cst(\Bisp))\Hilm\).  The latter is
  equivalent to~\ref{en:rep_CP_Hilm_4}.  Summing up, we get a
  \Star{}homomorphism \(\mathcal{O}\to \Comp(\Hilm)\) from the data
  and conditions in the proposition.

  Conversely, let \((\psi,L)\) be a Toeplitz representation.  Let
  \(a\in\Bis(\Bisp)\) and
  \(f_1,f_2\in \Contc(a) \subseteq \Cst(\Bisp)\).  If \(x\in\Bisp\),
  \(g\in\Gr\) are such that \(x\in a\) and \(x g \in a\), then~\(g\)
  is a unit because the orbit space projection is injective on~\(a\)
  and~\(\Gr\) acts freely on~\(\Bisp\).  In addition, then \(g=\s(x)\)
  because it is composable with~\(x\).  Therefore, the Hilbert module
  scalar product on \(\ContS(\Bisp)\) in
  \cite{Antunes-Ko-Meyer:Groupoid_correspondences}*{(7.2)} specialises
  to \(\braket{f_1}{f_2}(\s(x)) = \conj{f_1(x)} f_2(x)\) for
  \(x\in a\) and \(\braket{f_1}{f_2}(g)=0\) for
  \(g\in \Gr\setminus \s(a)\).  More briefly,
  \(\braket{f_1}{f_2} = (f_1^\s)^\dagger f_2^\s\).  So
  \(L(f_1)^* L(f_2) = \psi(\braket{f_1}{f_2})\) implies that there is a
  unique isometric map \(\Rep_a\colon \Hilm_a \to \Hilm\) with
  \(\Rep_a(\varphi_0(f^\s)\xi) = L(f)\xi\) for all \(f\in\Contc(a)\),
  \(\xi\in\Hilm\).  By construction, \(\Rep_a\) has the domain
  required in~\ref{en:rep_CP_Hilm_3}.

  Let \(a\in\Bis(\Gr)\), \(b\in\Bis(\Bisp)\) and
  \(f_j\in \Contc(j)\) for \(j=a,b\).  Then \(f_a*f_b\) is
  supported on the slice~\(a b\) and
  \((f_a * f_b)^\s = (f_a^\s \circ \varrho_b) f_b^\s\).  So
  \begin{multline}
    \label{eq:Rep_a_b_functions}
    \Rep_a \varphi_0(f_a^\s) \Rep_b \varphi_0(f_b^\s)
    = \psi(f_a) L(f_b)
    = L(f_a*f_b)
    \\= \Rep_{a b} \varphi_0 \bigl( (f_a*f_b)^\s\bigr)
    = \Rep_{a b} \varphi_0 (f_a^\s \circ \varrho_b) \varphi_0 (f_b^\s).
  \end{multline}
  If~\(a\) is the unit slice~\(\Gr^0\), then \(\Rep_{\Gr^0}\) is the
  identity map on all of~\(\Hilm\) by the construction above.
  Then~\eqref{eq:Rep_a_b_functions} specialises
  to~\ref{en:rep_CP_Hilm_5} with~\(b\) instead of~\(a\).
  Letting~\(a\) be general again, \ref{en:rep_CP_Hilm_5} identifies
  the left hand side in~\eqref{eq:Rep_a_b_functions} with
  \(\Rep_a \Rep_b \varphi_0 (f_a^\s \circ \varrho_b) \varphi_0
  (f_b^\s) = \Rep_a \Rep_b \varphi_0 \bigl( (f_a*f_b)^\s\bigr)\).
  Any function \(f\in \Contc(\s(a b))\) may be written as
  \((f_a*f_b)^\s\) for \(f_a,f_b\) as above.
  So~\eqref{eq:Rep_a_b_functions} says that
  \(\Rep_{a b} \varphi_0(f) = \Rep_a\Rep_b \varphi_0(f)\) for all
  \(f\in \Contc(\s(a b))\).  Therefore, \(\Rep_a \Rep_b\) is defined
  on~\(\Hilm_{\s(a b)}\), and equal to~\(\Rep_{a b}\) there.

  Next we show that the domain of~\(\Rep_a \Rep_b\) is contained in
  the domain of~\(\Rep_{a b}\).  By
  Lemma~\ref{lem:domains_Hilm_apprid}, \(\Rep_j\) is defined at
  \(\xi\in\Hilm\) if and only if \(\lim \varphi_0(f_j^\s)\xi = \xi\)
  in norm for an approximate unit~\((f_j)\) in \(\Contc(j)\) for
  \(j\in \{a,b, a b\}\).  We may choose the functions
  in~\eqref{eq:Rep_a_b_functions} such that \(f_a^\s\) and~\(f_b^\s\)
  converge to~\(1\) uniformly on compact subsets of
  \(\s(j)\subseteq \Gr^0\) for \(j=a,b\), respectively.  Let
  \(\xi\in\Hilm\) belong to the domain of~\(\Rep_a \Rep_b\).  Then
  \(\lim \varphi_0(f_b^\s)\xi = \xi\) and
  \(\lim \varphi_0(f_a^\s) \Rep_b(\xi) = \Rep_b(\xi)\) together imply
  that
  \(\lim \varphi_0(f_a^\s) \Rep_b \varphi_0(f_b^\s) \xi =
  \Rep_b(\xi)\).  Using~\ref{en:rep_CP_Hilm_5}, this implies
  \(\lim \Rep_b \varphi_0((f_a * f_b)^\s) \xi = \Rep_b(\xi)\).
  Since~\(\Rep_b\) is a partial unitary, this implies
  \(\lim \varphi_0((f_a * f_b)^\s) \xi = \xi\), and this says
  that~\(\xi\) belongs to the domain of~\(\Rep_{a b}\).  This finishes
  the proof that the partial unitaries~\(\Rep_a\)
  satisfy~\ref{en:rep_CP_Hilm_1} for \(a\in \Bis(\Gr)\),
  \(b\in \Bis(\Bisp)\).  Similar computations
  establish~\ref{en:rep_CP_Hilm_1} for \(a,b\in \Bis(\Gr)\) and for
  \(a\in \Bis(\Bisp)\), \(b\in \Bis(\Gr)\).  That is,
  \ref{en:rep_CP_Hilm_1} holds.

  Let \(a,b\in\Bis(\Bisp)\) and \(f_j\in \Contc(j)\) for \(j=a,b\).
  Then \(L(f_a)^* L(f_b) = \psi(\braket{f_a}{f_b})\) for
  \(f_a\in \Contc(a)\) for \(j=a,b\).  Since
  \(L(f_j) = \Rep_j \varphi_0(f_j^\s)\) for \(j=a,b\) and
  \(\psi(\braket{f_a}{f_b}) = \Rep_{\braket{a}{b}}
  \varphi_0\bigl((f_a^\s\circ \vartheta_{\braket{a}{b}})^\dagger
  f_b^\s\bigr)\), this implies~\ref{en:rep_CP_Hilm_2}.

  The two constructions outlined above are inverse to each other and
  so give a bijection between nondegenerate representations
  \(\mathcal{O}\to \Comp(\Hilm)\) and the families
  \((\varphi_0,\Rep_a)\) with the properties in the statement of the
  proposition.
\end{proof}

Comparing the descriptions of representations of \(\Cst(\Mod)\)
and~\(\mathcal{O}\), the main difference is that in
Proposition~\ref{pro:rep_model_Hilm} we require a representation of
\(\Cont_0(\Omega(R))\), whereas we only require a representation of
\(\Cont_0(\Gr^0)\) in Proposition~\ref{pro:rep_CP_Hilm}.

\begin{lemma}
  \label{lem:compare_reps}
  A nondegenerate \Star{}homomorphism \(\Cst(\Mod)\to\Comp(\Hilm)\)
  induces a nondegenerate \Star{}homomorphism
  \(\mathcal{O}\to\Comp(\Hilm)\).
\end{lemma}

\begin{proof}
  By composing with the proper map
  \(\rg\colon \Omega(\Reg) \to \Gr^0\), a \Star{}homomorphism on
  \(\Cont_0(\Omega(\Reg))\) induces one on \(\Cont_0(\Gr^0)\).  So the
  data in Proposition~\ref{pro:rep_model_Hilm} that is equivalent to a
  nondegenerate \Star{}homomorphism \(\Cst(\Mod)\to\Comp(\Hilm)\)
  produces the data in Proposition~\ref{pro:rep_CP_Hilm} that is
  equivalent to a nondegenerate \Star{}homomorphism
  \(\mathcal{O}\to\Comp(\Hilm)\).  The conditions
  \ref{en:rep_model_Hilm_1} and \ref{en:rep_CP_Hilm_1} are identical,
  and \ref{en:rep_model_Hilm_3} is equivalent to
  \ref{en:rep_CP_Hilm_3}.  In \(\Omega(\Reg)\),
  \(\rg^{-1}(\Reg) \subseteq \Bisp\cdot \Omega(\Reg)\).  This is
  covered by the subsets \(\pi^{-1}(\Qu(a))\) for \(a\in\Bis(\Bisp)\).
  So \ref{en:rep_model_Hilm_4} implies \ref{en:rep_CP_Hilm_4}.
  Computations as in the proof of Proposition~\ref{pro:rep_model_Hilm}
  show that
  \((f\circ \rg \circ \vartheta_a)(\omega) = f\circ \varrho_a \circ
  \rg(\omega)\) for \(f\in \Cont_0(\Gr^0)\), \(a\in\Bis(\Bisp)\),
  \(\omega \in \Omega(\Reg)\) with \(\rg(\omega) \in \s(a)\).
  Therefore, \ref{en:rep_model_Hilm_5} for functions
  on~\(\Omega(\Reg)\) of the form \(f\circ \rg\) implies
  \ref{en:rep_CP_Hilm_5}.  In the situation of \ref{en:rep_CP_Hilm_2},
  let~\((h_n)\) be a positive approximate unit in
  \(\Cont_0(\Qu(a))\).  Then
  \(\Rep_a \varphi_0(f_a^\s) = \lim \varphi(h_n \circ \pi) \Rep_a
  \varphi_0(f_a^\s)\) by Lemma~\ref{lem:domains_Hilm_apprid} and
  \ref{en:rep_model_Hilm_4}.  A computation in the proof of
  Proposition~\ref{pro:rep_model_Hilm} shows that
  \(h_n\circ \pi \circ \vartheta_b\) is of the form
  \(h_n'\circ \rg\) for an approximate unit \((h_n')\) in
  \(\Cont_0(\s(\braket{a}{b}))\).  Now
  \begin{multline*}
    \braket{\Rep_a \varphi_0(f_a^\s) \xi_a}{\Rep_b \varphi_0(f_b^\s)
      \xi_b}
    = \lim  {}\braket{\varphi(h_n\circ \pi)\Rep_a \varphi_0(f_a^\s)
      \xi_a}  {\Rep_b \varphi_0(f_b^\s) \xi_b}
    \\= \lim  {}\braket{\Rep_a \varphi_0(f_a^\s)
      \xi_a}  {\Rep_b \varphi_0(h'_n f_b^\s) \xi_b}.
  \end{multline*}
  Therefore, it suffices to prove the equation in
  \ref{en:rep_CP_Hilm_2} when~\(f_b^\s\) is supported in
  \(\s(\braket{a}{b}) \subseteq \s(b)\).  We assume this.  Then
  \(\Rep_b(f_b^\s)\) belongs to \(\Hilm_{\pi^{-1}(\Qu(a) \cap \Qu(b))}\),
  which is contained in the codomain of~\(\Rep_a^*\).  So
  \begin{multline*}
      \braket{\Rep_a \varphi_0(f_a^\s) \xi_a}
      {\Rep_b \varphi_0(f_b^\s) \xi_b}
      = \braket{\varphi_0(f_a^\s) \xi_a}
      {\Rep_a^* \Rep_b \varphi_0(f_b^\s) \xi_b}
      \\= \braket{\varphi_0(f_a^\s) \xi_a}
      {\Rep_{\braket{a}{b}} \varphi_0(f_b^\s) \xi_b}
      = \braket{\xi_a} {\varphi_0((f_a^\s)^\dagger)
        \Rep_{\braket{a}{b}} \varphi_0(f_b^\s) \xi_b}.
  \end{multline*}
  This together with \ref{en:rep_CP_Hilm_5} implies
  \ref{en:rep_CP_Hilm_2}.  So all the conditions in
  Proposition~\ref{pro:rep_CP_Hilm} hold.
\end{proof}

Applied to the identity representation on \(\Hilm= \Cst(\Mod)\), the
lemma implies that there is a nondegenerate \Star{}homomorphism
\(\mathcal{O}\to \Cst(\Mod)\).  In fact, it is not hard to describe
the Toeplitz representation \(\Cst(\Gr) \to \Cst(\Mod)\) and
\(\Cst(\Bisp) \to \Cst(\Mod)\) directly.  The more difficult point
is that this map is an isomorphism or, equivalently, both
\(\Cst\)\nb-algebras have the same nondegenerate compact
representations.  This means that, in the presence of the
operators~\(\Rep_a\) as in Proposition~\ref{pro:rep_CP_Hilm}, we may
always extend the homomorphism
\(\varphi_0\colon \Cont_0(\Gr^0) \to \Comp(\Hilm)\) to all of
\(A_{[0,\infty)} \cong \Cont_0(\Omega_{[0,\infty)})\).

\begin{lemma}
  \label{lem:left_action_space_n}
  The pointwise multiplication action of \(\Cont_0(\Bisp_n/\Gr)\) on
  \(\ContS(\Bisp)\) defined by \((M_f h)(x) \defeq f([x]) h(x)\) for
  \(f\in \Cont_0(\Bisp_n/\Gr)\), \(h\in \ContS(\Bisp)\),
  \(x\in \Bisp_n\), extends to a compact operator on
  \(\Cst(\Bisp_n)\), and this defines a nondegenerate
  \Star{}homomorphism
  \(\Cont_0(\Bisp_n/\Gr) \to \Comp(\Cst(\Bisp_n))\).
\end{lemma}

\begin{proof}
  The orbit space projection \(\Bisp_n \to \Bisp_n/\Gr\) may be
  viewed as an action of the space \(\Bisp_n/\Gr\) on~\(\Bisp_n\).
  It manifestly commutes with the right \(\Gr\)\nb-action and so it
  turns~\(\Bisp_n\) into a groupoid correspondence
  \(\Bisp_n/\Gr \leftarrow \Gr\).  Since the anchor map of the left
  action induces the identity homeomorphism
  \(\Bisp_n/\Gr \congto \Bisp_n/\Gr\), this groupoid correspondence
  is also proper.  So the left action of
  \(\Cst(\Bisp_n/\Gr) \cong \Cont_0(\Bisp_n/\Gr)\) on
  \(\Cst(\Bisp)\) is by compact operators by
  \cite{Antunes-Ko-Meyer:Groupoid_correspondences}*{Theorem~7.14}.
  This left action comes from the left
  \(\ContS(\Bisp_n/\Gr)\)-module structure on \(\ContS(\Bisp_n)\)
  that is defined in
  \cite{Antunes-Ko-Meyer:Groupoid_correspondences}*{(7.1)}.  In the
  case at hand, this simplifies to pointwise multiplication.
\end{proof}

\begin{proposition}
  \label{pro:main_for_emptyset}
  Let \(\Reg = \emptyset\), so that~\(\mathcal{O}\) is the Toeplitz
  \(\Cst\)\nb-algebra of~\(\Cst(\Bisp)\).  There is a canonical
  \Star{}homomorphism
  \(\Cont_0(\Omega_{[0,\infty)}) = A_{[0,\infty)} \to \mathcal{O}\),
  which together with the partial unitaries~\(\Rep_a\) for \(a\in \Bis\)
  associated to the identity representation of~\(\mathcal{O}\) on
  itself produces a nondegenerate \Star{}homomorphism
  \(\Cst(\Mod) \to \mathcal{O}\).
\end{proposition}

\begin{proof}
  Let \(D=\Cst(\Gr)\) and let \(\Hilm = \Cst(\Bisp)\).  Let
  \(\Hilm^{\otimes n}\) be the \(n\)\nb-fold tensor product
  of~\(\Hilm\) with itself over~\(D\).  The multiplicativity of
  \(\Bisp\mapsto \Cst(\Bisp)\) in
  \cite{Antunes-Ko-Meyer:Groupoid_correspondences}*{Proposition~7.12}
  implies that
  \[
    \Cst(\Bisp_n) \otimes_{\Cst(\Gr)} \Cst(\Bisp)
    \cong\Cst(\Bisp_{n+1})
  \]
  for all \(n\in\N\).  So there are natural isomorphisms
  \(\Hilm^{\otimes n} \cong \Cst(\Bisp_n)\).  Recall
  from~\cite{Pimsner:Generalizing_Cuntz-Krieger} that the Toeplitz
  \(\Cst\)\nb-algebra of~\(\Hilm\) contains a copy of
  \(\Comp(\Hilm^{\otimes n})\) for all \(n\in\N\) and that the
  product of \(T_n\in \Comp(\Hilm^{\otimes n})\) and
  \(T_m\in \Comp(\Hilm^{\otimes m})\) for \(m\le n\) is the product
  of~\(T_n\) with the adjointable operator
  \(T_m \otimes \Id_{\Hilm^{\otimes n-m}}\).

  Lemma~\ref{lem:left_action_space_n} provides nondegenerate
  \Star{}homomorphisms
  \[
    \Cont_0(\Bisp_n/\Gr)
    \to \Comp(\Bisp_n)
    \cong \Comp(\Hilm^{\otimes n})
    \to \mathcal{O}.
  \]
  If \(m\le n\) and \(f_m \in \Cont_0(\Bisp_m/\Gr)\), then the
  operator on~\(\Hilm^{\otimes n}\) induced by pointwise
  multiplication by~\(f_m\) is pointwise multiplication by the
  function \(f_m \circ \pi_n^m\) on~\(\Bisp_n/\Gr\).  Therefore, the
  product of \(f_n \in \Cont_0(\Bisp_n/\Gr)\) and
  \(f_m \in \Cont_0(\Bisp_m/\Gr)\) in~\(\mathcal{O}\) is the product
  of \(f_n\) and \(f_m \circ \pi_n^m\) as functions
  on~\(\Bisp_n/\Gr\).  The same product is used to define the
  commutative \(\Cst\)\nb-algebras~\(A_{[0,n]}\) in
  Definition~\ref{def:Amn}.  Therefore, we get \Star{}homomorphisms
  \(A_{[0,n]} \to \mathcal{O}\).  These are compatible with the
  canonical maps \(A_{[0,n]} \to A_{[0,m]}\) for \(n\le m\).  So they
  combine to a \Star{}homomorphism
  \(\varphi\colon \Cont_0(\Omega_{[0,\infty)}) = A_{[0,\infty)} \to
  \mathcal{O}\).

  Now use Proposition~\ref{pro:rep_CP_Hilm} to associate to the
  identity map on~\(\mathcal{O}\) a nondegenerate \Star{}homomorphism
  \(\varphi_0\colon \Cont_0(\Gr^0) \to \mathcal{O}\) and a family of
  partial unitaries~\(\Rep_a\) on~\(\mathcal{O}\), viewed as a Hilbert
  module over itself.  The Toeplitz representation \((\psi,L)\) that
  corresponds to the identity map on~\(\mathcal{O}\) is such that
  \(\psi\) and~\(L\) are the canonical inclusions of \(\Cst(\Gr) = D\)
  and \(\Cst(\Bisp) = \Hilm\) into the Toeplitz
  \(\Cst\)\nb-algebra~\(\mathcal{O}\) of~\(\Hilm\), respectively.  The
  \Star{}homomorphism~\(\varphi_0\) is the restriction of the
  canonical map \(\Cst(\Gr) \to \mathcal{O}\) to \(\Cont_0(\Gr^0)\).
  This agrees with the restriction of~\(\varphi\) to
  \(\Cont_0(\Gr^0) \subseteq \Cont_0(\Omega_{[0,\infty)})\).  The
  partial unitary~\(\Rep_a\) is defined by
  \(\Rep_a(\psi(f^\s)\xi) = L(f) \xi\) for all
  \(f\in \Contc(a) \subseteq \Cst(\Bisp)\), \(\xi\in \mathcal{O}\).

  We claim that \(\varphi\) and the partial unitaries~\(\Rep_a\)
  satisfy the conditions in Proposition~\ref{pro:rep_model_Hilm} for
  \(\Reg=\emptyset\) with \(\Omega(\emptyset) = \Omega_{[0,\infty)}\).
  The conditions \ref{en:rep_CP_Hilm_1} and \ref{en:rep_CP_Hilm_3}
  imply \ref{en:rep_model_Hilm_1} and \ref{en:rep_model_Hilm_3}
  because~\(\varphi_0\) is the restriction of~\(\varphi\) to
  \(\Cont_0(\Gr^0) \subseteq \Cont_0(\Omega_{[0,\infty)})\).

  We now check \ref{en:rep_model_Hilm_4}.  The codomain
  of~\(\Rep_a\) for \(a\in \Bis(\Bisp)\) is the closed right ideal
  in~\(\mathcal{O}\) generated by
  \(\Contc(a) \subseteq \Cst(\Bisp) \subseteq \mathcal{O}\).  Let
  \(f_1,f_2 \in \Contc(a)\) and \(f_3 \in \ContS(\Bisp)\).  Then
  \begin{multline*}
    \ket{f_1}\bra{f_2} f_3(x)
    = \bigl(f_1 * \braket{f_2}{f_3}\bigr)(x)
    \\= \sum_{g\in \Gr^{\s(x)}} f_1(x g) \braket{f_2}{f_3}(g^{-1})
    = \sum_{g\in \Gr^{\s(x)}} \sum_{\setgiven{y\in \Bisp}{\s(y) = \s(g)}}
    f_1(x g) \conj{f_2(y)} f_3(y g^{-1}).
  \end{multline*}
  If \(g,y\) occur in a nonzero summand, then \(\s(y) = \s(x g)\) and
  \(x g,y\in a\).  This forces \(y = x g\) because~\(a\) is a slice.  So
  \[
    \ket{f_1}\bra{f_2} f_3(x)
    =  \sum_{g\in \Gr^{\s(x)}} f_1(x g) \conj{f_2(x g)} f_3(x).
  \]
  In other words, \(\ket{f_1} \bra{f_2}\) is the operator of pointwise
  multiplication by the function \(\Bisp/\Gr \to \C\),
  \([x] \mapsto \sum_{g\in \Gr^{\s(x)}} f_1(x g) \conj{f_2(x g)}\).
  If \(f_3 \in \Contc(a)\) as well, then
  \(\ket{f_1}\bra{f_2} f_3 \in \Contc(a)\) is the pointwise product
  \(f_1 f_2^\dagger f_3\).  Thus any function in~\(\Contc(a)\) may be
  written as \(\ket{f_1}\bra{f_2} f_3\) for
  \(f_1,f_2,f_3\in \Contc(a)\).  This implies that the closed right
  ideal in~\(\mathcal{O}\) generated by \(\Contc(a)\) is equal to the
  closed right ideal generated by
  \(\ket{f_1} \bra{f_2} \in \Comp(\Cst(\Bisp)) \subseteq \mathcal{O}\)
  for \(f_1,f_2\in \Contc(a)\).  The proof of
  Proposition~\ref{pro:main_for_emptyset} shows that this is the
  closed right ideal in~\(\mathcal{O}\) generated by
  \(\Contc(\Qu(a)) \subseteq \Cont_0(\Bisp/\Gr) \subseteq
  \Cont_0(\Omega_{[0,\infty)})\).  Then \ref{en:rep_model_Hilm_4}
  follows.

  Next we check~\ref{en:rep_model_Hilm_5} for \(a\in \Bis\).  Since
  the sum of \(\Cont_0(\Bisp_n/\Gr)\) is dense in~\(A_{[0, \infty)}\),
  we may assume without loss of generality that~\(f_1\) belongs to the
  image of \(\Cont_0(\Bisp_n/\Gr)\) for some \(n\in\N\); this image is
  \(\pi_n^*(\Cont_0(\Bisp_n/\Gr))\) for the canonical map
  \(\pi_n\colon \Omega_{[0,\infty)} \to \Omega_{[0,n]}\), where we
  tacitly embed \(\Bisp_n/\Gr\) into~\(\Omega_{[0,n]}\).  Then
  \(\varphi(f_1) \in \Comp(\Hilm^{\otimes n}) \subseteq \mathcal{O}\).
  The domain~\(D_{a^* a}\) is
  \(\rg^{-1}(\s(a)) \subseteq \Omega_{[0,\infty)}\).  Therefore,
  \(\Cont_0(D_{a^* a}) = \rg^*(\Cont_0(\s(a))) \cdot
  \Cont_0(\Omega_{[0,\infty)})\) and we may assume without
  loss of generality that \(f_2 = \rg^*(f_3)\cdot f_4\) for some
  \(f_3 \in \Cont_0(\s(a))\), \(f_4\in \Cont_0(\Omega_{[0,\infty)})\).
  We are going to prove that
  \[
    \Rep_a^* \varphi(f_1) \Rep_a \varphi_0(f_3^\s)
    = \varphi(f_1\circ \vartheta_a) \varphi_0(f_3^\s)
  \]
  for \(f_1\in \pi_n^*\Cont_0(\Bisp_n/\Gr)\), \(f_3\in \Cont_0(a)\),
  so that \(f_3^\s\in \Cont_0(\s(a))\).  This
  implies~\ref{en:rep_model_Hilm_5}.  We distinguish the cases
  \(a\in \Bis(\Gr)\) and \(a\in \Bis(\Bisp)\).

  First let \(a\in\Bis(\Gr)\).  Then
  \(\Rep_a \varphi_0(f_3^\s) = \psi(f_3) \in \Contc(a) \subseteq
  \Cst(\Gr)\).  When we multiply with
  \(\varphi(f_1) \in \Cont_0(\Bisp_n/\Gr) \subseteq
  \Comp(\Hilm^{\otimes n})\) in the Toeplitz \(\Cst\)\nb-algebra, we
  compose \(\varphi(f_1)\) with the operator
  \(\psi(f_3) \otimes \Id_{\Hilm^{\otimes n-1}} \in
  \Comp(\Hilm^{\otimes n})\).  The latter maps a function
  \(h\in \ContS(\Bisp_n) \subseteq \Cst(\Bisp_n) \cong \Hilm^{\otimes
    n}\) to the function
  \[
    [x_1,\dotsc,x_n] \mapsto
    \begin{cases}
      f_3(g) h[g^{-1}\cdot x_1,x_2,\dotsc,x_n] &\text{if }\rg(x_1) \in \rg(a);\\
      0 &\text{otherwise;}
    \end{cases}
  \]
  here \(x_1,\dotsc,x_n\in\Bisp\) satisfy \(\s(x_j) = \rg(x_{j+1})\)
  for \(j=1,\dotsc,n-1\), \([x_1,\dotsc,x_n]\) denotes the equivalence
  class of \((x_1,\dotsc,x_n)\) in~\(\Bisp_n\), and \(g\in a\) is the
  unique element with \(\rg(a) = \rg(x_1)\).  Then~\(\varphi(f_1)\)
  multiplies pointwise with the function
  \(f_1\in\Cont_0(\Bisp_n/\Gr)\), viewed as a \(\Gr\)\nb-invariant
  function on~\(\Bisp_n\).  Finally, \(\Rep_a^*\) gives the function
  \[
    [x_1,\dotsc,x_n] \mapsto
    \begin{cases}
      f_1[g\cdot x_1,x_2,\dotsc,x_n] f_3(g) h[g^{-1} g x_1,x_2,\dotsc,x_n]
      &\text{if }\rg(x_1) \in \s(a);\\
      0 &\text{otherwise;}
    \end{cases}
  \]
  here \(g\in a\) is the unique element with \(\s(g) = \rg(x_1)\), and
  then it is also the unique element of~\(a\) with
  \(\rg(g) = \rg(g \cdot x_1)\).  The function
  \(f_1\circ \vartheta_a\) is defined on the set of
  \([x_1,\dotsc,x_n] \in \Bisp_n\) with
  \(\rg[x_1,\dotsc,x_n] = \rg(x_1) \in \s(a)\), and there it has the
  value \(f_1[g\cdot x_1,\dotsc,x_n]\) for the unique \(g\in a\) with
  \(\s(g) = \rg(x_1)\).  So the composite operator
  \(\Rep_a^* \varphi(f_1) \Rep_a \varphi_0(f_3^\s)\) first multiplies
  pointwise with \(f_3^\s(\rg[x_1,\dotsc,x_n])\) and then with
  \(f_1\circ \vartheta_a\).  That is, it is equal to
  \(\varphi(f_1\circ \vartheta_a) \varphi_0(f_3^\s) \) as desired.

  Next let \(a\in\Bis(\Bisp)\).  The computation is similar.  Now
  \(\Rep_a \varphi_0(f_3^\s) = L(f_3)\in \Cst(\Bisp)\) maps
  \(\Hilm^{\otimes n-1} \to \Hilm^{\otimes n}\) in the Toeplitz
  \(\Cst\)\nb-algebra.  So the composite with~\(\varphi(f_1)\) is now
  the operator in \(\Comp(\Hilm^{\otimes n-1},\Hilm^{\otimes n})\)
  that maps
  \(h\in \ContS(\Bisp_{n-1}) \subseteq \Cst(\Bisp_{n-1}) \cong
  \Hilm^{\otimes n-1}\) to the function
  \[
    \Bisp_n \to \C,\qquad
    [x_1,\dotsc,x_n] \mapsto
    \begin{cases}
      f_3(x_1) h[x_2,x_3,\dotsc,x_n] &\text{if } \Qu(x_1) \in \Qu(a);\\
      0 &\text{otherwise;}
    \end{cases}
  \]
  here the representative \([x_1,x_2,\dotsc,x_n]\) is chosen with
  \(x_1\in a\), which is possible if and only if
  \(\Qu(x_1) \in \Qu(a)\).  Multiplying pointwise with~\(f_1\) and
  then applying~\(\Rep_a^*\) gives the operator in
  \(\Comp(\Hilm^{\otimes n-1})\) that maps
  \(h\in \ContS(\Bisp_{n-1}) \subseteq \Cst(\Bisp_{n-1}) \cong
  \Hilm^{\otimes n-1}\) to the function on~\(\Bisp_{n-1}\) given by
  \[
    [x_2,\dotsc,x_n] \mapsto
    \begin{cases}
      f_1[x_1,\dotsc,x_n] f_3(x_1) h[x_2,x_3,\dotsc,x_n] &\text{if } \rg(x_2) \in \s(a);\\
      0 &\text{otherwise;}
    \end{cases}
  \]
  here \(x_1\in a\) is the unique element with \(\s(x_1) = \rg(x_2)\).
  The function \(f_1\circ \vartheta_a\) is defined at
  \([x_2,\dotsc,x_n]\) if and only if \(\rg(x_2) \in \s(a)\) and then
  takes the value \(f_1[x_1,\dotsc,x_n]\) for~\(x_1\) as above.  So
  the composite operator
  \(\Rep_a^* \varphi(f_1) \Rep_a \varphi_0(f_3^\s)\) is equal to the
  image of \(\varphi(f_1\circ \vartheta_a) \varphi_0(f_3^\s)\) in
  \(\Comp(\Hilm^{\otimes n-1})\) as desired.  This finishes the proof
  of~\ref{en:rep_model_Hilm_5}.  Since the
  condition~\ref{en:rep_model_Hilm_2} is redundant, the
  representation~\(\varphi\) and the operators~\(\Rep_a\) have all the
  conditions required in Proposition~\ref{pro:rep_model_Hilm}.
\end{proof}

Finally, we may state and prove our main theorem:

\begin{theorem}
  \label{the:main}
  Let~\(\Gr\) be an \'etale locally compact groupoid and let
  \(\Bisp\colon \Gr\leftarrow \Gr\) be an \'etale locally compact
  groupoid correspondence.  Let \(\Reg\subseteq \Gr^0\) be an open
  subset such that~\(\Bisp\) is proper on~\(\Reg\).  Let
  \(\Mod \defeq \Omega(\Reg) \rtimes \IS(\Gr,\Bisp)\) be the groupoid
  model of \((\Gr,\Bisp,\Reg)\).  Then \(\Cst(\Mod)\) is isomorphic to
  the Cuntz--Pimsner algebra of the \(\Cst\)\nb-correspondence
  \(\Cst(\Bisp)\colon \Cst(\Gr) \leftarrow \Cst(\Gr)\) relative to the
  ideal \(\Cst(\Gr_\Reg)\), which acts on~\(\Cst(\Bisp)\) by compact
  operators.
\end{theorem}

\begin{proof}
  We abbreviate \(\IS = \IS(\Gr,\Bisp)\) as above.  Let
  \(\mathcal{T}\) denote the Toeplitz \(\Cst\)\nb-algebra
  of~\(\Cst(\Bisp)\), which is the Cuntz--Pimsner algebra relative to
  the zero ideal, which corresponds to \(\Reg=\emptyset\).  Then
  Lemma~\ref{lem:compare_reps} and
  Proposition~\ref{pro:main_for_emptyset} produce canonical
  nondegenerate \Star{}homomorphisms
  \(\mathcal{T} \to \Cst(\Omega_{[0,\infty)}\rtimes \IS)\) and
  \(\Cst(\Omega_{[0,\infty)} \rtimes \IS) \to \mathcal{T}\).  On the
  level of the universal properties in Propositions
  \ref{pro:rep_model_Hilm} and~\ref{pro:rep_CP_Hilm}, these
  nondegenerate \Star{}homomorphisms map the representation of
  \(\Cst(\Omega_{[0,\infty)} \rtimes \IS)\) corresponding to
  \((\varphi,\Rep_a)\) to the representation of~\(\mathcal{T}\)
  corresponding to \((\varphi_0,\Rep_a)\) with the same partial
  unitaries~\(\Rep_a\) and \(\varphi_0 = \varphi\circ \rg^*\) for
  \(\rg\colon \Omega_{[0,\infty)} \to \Gr^0\).  The proof of
  Proposition~\ref{pro:main_for_emptyset} shows that~\(\varphi_0\) may
  be extended in a unique way to~\(\varphi\) so as to satisfy the
  conditions in Proposition~\ref{pro:rep_model_Hilm}.  In particular,
  we get a bijection on the level of the data in Propositions
  \ref{pro:rep_model_Hilm} and~\ref{pro:rep_CP_Hilm}.  So the two maps
  above are inverse to each other.

  Now let~\(\Reg\) be arbitrary.  Lemma~\ref{lem:compare_reps} shows
  that the map
  \(\mathcal{T} \to \Cst(\Omega_{[0,\infty)}\rtimes \IS)\) descends to
  a nondegenerate \Star{}homomorphism
  \(\mathcal{O}\to \Cst(\Omega(\Reg)\rtimes \IS)\).  It remains to
  prove that the nondegenerate \Star{}homomorphism
  \(\Cst(\Omega_{[0,\infty)} \rtimes \IS) \to \mathcal{T}\) descends
  to \(\Cst(\Omega(\Reg)\rtimes \IS) \to \mathcal{O}\).  The
  embeddings of \(\Cont_0(\Gr^0)\) and \(\Cont_0(\Bisp/\Gr)\)
  into~\(\mathcal{O}\) combine to a canonical map
  \(\Cont_0(\Omega_{[0,1]})\to \mathcal{O}\).  The Cuntz--Pimsner
  covariance condition implies that if \(f\in \Cont_0(\Reg)\), then
  \(f\in\Cont_0(\Gr^0)\) and \(f\circ \rg_*\in \Cont_0(\Bisp/\Gr)\)
  have the same image in~\(\mathcal{O}\).  This says that the
  \Star{}homomorphism \(\Cont_0(\Omega_{[0,1]}) \to \mathcal{O}\)
  vanishes on the ideal \(\Cont_0(\Reg)\).  The covariance
  condition~\ref{en:rep_model_Hilm_5} implies that the kernel of the
  \Star{}homomorphism
  \(\Cont_0(\Omega_{[0,\infty)}) \to \mathcal{T} \to \mathcal{O}\) is
  invariant under the action of~\(\IS\).  By
  Proposition~\ref{pro:universal_restricted_to_Reg}, the
  \(\IS\)\nb-orbit of \(\Reg\subseteq \Gr^0\) is equal to the
  complement of \(\Omega(\Reg)\) in~\(\Omega_{[0,\infty)}\).  This
  gives the desired factorisation of the \Star{}homomorphism
  \(\Cont_0(\Omega_{[0,\infty)}) \to \mathcal{T} \to \mathcal{O}\)
  through \(\Cont_0(\Omega(\Reg))\).
\end{proof}

\begin{remark}
  Since the construction of full groupoid \(\Cst\)\nb-algebras is
  exact, the closed invariant subset
  \(\Omega(\Reg)\subseteq \Omega_{[0,\infty)}\) gives rise to an
  extension of groupoid \(\Cst\)\nb-algebras
  \[
    \Cont_0\Bigl( \bigsqcup_{n\in\N} (\Bisp_n/\Gr)_\Reg\Bigr) \rtimes \IS(\Bisp,\Gr)
    \into \Cont_0( \Omega_{[0,\infty)}) \rtimes \IS(\Bisp,\Gr)
    \prto  \Cont_0( \Omega_\Reg) \rtimes \IS(\Bisp,\Gr).
  \]
  So the kernel in this extension is the kernel of the canonical
  quotient map from the Toeplitz \(\Cst\)\nb-algebra to the relative
  Cuntz--Pimsner algebra.  It is well known that this kernel is
  isomorphic to the \(\Cst\)\nb-algebra of compact operators on
  \(\HilmF\cdot \Cst(\Gr_\Reg)\), where~\(\HilmF\) denotes the Fock
  module associated to \(\Cst(\Bisp)\).  This is Morita--Rieffel
  equivalent to the ideal \(\Cst(\Gr_\Reg)\).  The Morita--Rieffel
  equivalence of this ideal and \(\Cst(\Gr_\Reg)\) also follows from
  Theorem~\ref{the:Morita_equivalence} and
  Proposition~\ref{pro:universal_restricted_to_Reg}.
\end{remark}

\section{Some more examples}
\label{sec:examples}

We have discussed throughout the text how graph \(\Cst\)\nb-algebras
fit into our theory.  Namely, this is the special case where the
groupoid~\(\Gr\) is just a discrete set~\(V\) and~\(\Reg\) is the set
of regular vertices in the graph; a groupoid correspondence on~\(V\)
is the same as a directed graph.  It is well known that the
Cuntz--Pimsner algebra of the associated \(\Cst\)\nb-correspondence
relative to the ideal \(\Cont_0(\Reg)\) is the graph
\(\Cst\)\nb-algebra.  Finally, we consider the case of groupoid
correspondences of discrete groupoids in complete generality.  These
groupoid correspondences are roughly the same as the self-similar
groupoid actions
of~\cite{Laca-Raeburn-Ramagge-Whittaker:Equilibrium_self-similar_groupoid},
except that we do not require the graph in question to be finite.

So let~\(\Gr\) be a groupoid with the discrete topology.
Then~\(\Bisp\) also carries the discrete topology because
\(\s\colon \Bisp\to \Gr^0\) is a local homeomorphism.  Being a
groupoid correspondence means that the set~\(\Bisp\) carries commuting
actions of~\(\Gr\) on the left and right, where the right action
is free.  Then we may choose a fundamental domain \(F\subseteq \Bisp\)
for the right \(\Gr\)\nb-action, such that the following map is a
bijection:
\[
  F\times_{\s,\Gr^0,\rg} \Gr \congto \Bisp,
  \qquad (x,g)\mapsto x\cdot g.
\]
Since the left action commutes with the right action, it must be of
the form
\begin{equation}
  \label{eq:action_X_self-similar}
  h\cdot (x,g) = (h\circ x,h|_x\cdot g)  
\end{equation}
for \(h\in\Gr\), \(x\in F\), \(g\in\Gr\) with \(\s(h) = \rg(x)\) and
\(\s(x) = \rg(g)\); here~\(\circ\) is a \(\Gr\)\nb-action on~\(F\)
and \(h|_x\in\Gr\) is defined for \(h\in\Gr\), \(x\in F\) with
\(\s(h) = \rg(x)\), and it satisfies \(\s(h|_x) = \s(x) = \rg(g)\) and
\((h g)|_x = h|_{g\circ x} \cdot g|_x\) for composable \(h,g,x\).  We
refer to
\cite{Antunes-Ko-Meyer:Groupoid_correspondences}*{Proposition~4.3} for
more details in the special case where \(\Gr = V\rtimes \Gamma\) is
the transformation groupoid of a group action on a discrete set, or
to~\cite{Laca-Raeburn-Ramagge-Whittaker:Equilibrium_self-similar_groupoid}.

The maps \(\rg,\s\colon \Bisp \rightrightarrows \Gr^0\) restrict to
the fundamental domain and allow us to view it as a directed graph.
This point of view is taken
in~\cite{Laca-Raeburn-Ramagge-Whittaker:Equilibrium_self-similar_groupoid}.
The groupoid~\(\Gr\) acts both on the set~\(\Gr^0\) of vertices of the
graph and on the edges by the action~\(\circ\).  However,
\eqref{eq:action_X_self-similar} requires the relation
\[
  \s(h\circ x)
  = \rg(h|_x\cdot g)
  = h|_x \cdot \rg(g)
  =  h|_x \cdot \s(x)
\]
instead of \(\s(h\circ x) = h\cdot \s(x)\).  The relation is also used
in~\cite{Laca-Raeburn-Ramagge-Whittaker:Equilibrium_self-similar_groupoid},
but it is rather unnatural if one thinks of an action on a graph.  In
fact, the original article~\cite{Exel-Pardo:Self-similar} assumes both
\(\s(h\circ x) = h\cdot \s(x)\) and
\(h|_x \cdot \s(x) = h\cdot \s(x)\).

Let \(\Reg \subseteq \Gr^0\).  Then~\(\Bisp\) is proper on~\(\Reg\) if
and only if \(\rg^{-1}(v) \cap F \subseteq F\) is finite for all
\(v\in \Reg\); equivalently, \(\rg^{-1}(v)/\Gr\) is finite in
\(\Bisp/\Gr\).  We call~\(\Bisp\) \emph{regular on~\(\Reg\)} if, in
addition, \(\rg^{-1}(v)\) is nonempty for all \(v\in\Reg\).  It is
reasonable to restrict attention to this case, although our theory
also works without this assumption.

Since \(\Gr\) and~\(\Bisp\) are discrete, the sets of
\(\delta\)\nb-functions~\(\delta_x\) for \(x\in\Gr\) or \(x\in\Bisp\)
are bases of \(\ContS(\Gr)\) and \(\ContS(\Bisp)\), respectively, and
the \Star{}algebra structure on \(\ContS(\Gr)\) and the
\(\ContS(\Gr)\)\nb-bimodule structure and inner product
on~\(\ContS(\Bisp)\) are uniquely determined by their values on the
\(\delta\)\nb-functions.  These are given by
\(\delta_g * \delta_h \defeq \delta_{g h}\) if \(\s(g) = \rg(h)\)
and~\(0\) otherwise, \(\delta_g^* = \delta_{g^{-1}}\) for \(g\in\Gr\),
and \(\braket{\delta_x}{\delta_y} = \delta_h\) if \(\s(x) = \rg(h)\)
and \(x\cdot h = y\) and~\(0\) if there is no \(h\in\Gr\) with
\(\s(x) = \rg(h)\) and \(x\cdot h = y\); the relation for
\(\delta_g * \delta_h\) also applies if one of \(\{g,h\}\) belongs
to~\(\Bisp\) and then describes the \(\ContS(\Gr)\)\nb-bimodule
structure on~\(\ContS(\Bisp)\).  This implies the following:

\begin{theorem}
  \label{the:gen_rel_for_self-similar}
  Let \(\Gr\) be a discrete groupoid and let
  \(\Bisp\colon \Gr\leftarrow \Gr\) be a groupoid correspondence that
  is proper on a subset \(\Reg \subseteq \Gr^0\).  Let~\(D\) be a
  \(\Cst\)\nb-algebra.  A Toeplitz representation \(\Cst(\Bisp)\to D\)
  is equivalent to elements \(\Rep_x\in D\) for
  \(x\in \Gr\sqcup \Bisp\) that satisfy the following relations:
  \begin{itemize}
  \item let \(x,y\in\Gr \sqcup \Bisp\) and \(x\in\Gr\) or \(y\in\Gr\);
    then \(\Rep_x * \Rep_y \defeq \Rep_{x y}\) if \(\s(x) = \rg(y)\)
    and \(\Rep_x * \Rep_y \defeq 0\) if \(\s(x) \neq \rg(y)\);
  \item \(\Rep_g^* = \Rep_{g^{-1}}\) for \(g\in\Gr\);
  \item if \(x,y\in\Bisp\) and there is \(h\in\Gr\) with
    \(\s(x) = \rg(h)\) and \(x\cdot h = y\), then
    \(\Rep_x^*\Rep_y = \Rep_h\), and \(\Rep_x^*\Rep_y = 0\) if no
    such~\(h\) exists.
  \end{itemize}
  This representation is Cuntz--Pimsner covariant on
  \(\Cst(\Gr_\Reg)\) if and only if
  \begin{equation}
    \label{eq:CP-covariance_self-similar}
    \Rep_v = \sum_{x\in \rg_{\Bisp}^{-1}(v) \cap F} \Rep_x \Rep_x^*  
  \end{equation}
  holds for all \(v\in \Reg\).  As a consequence, the Cuntz--Pimsner
  algebra of \(\Cst(\Bisp)\) is the universal \(\Cst\)\nb-algebra with
  the generators~\(\Rep_x\) for \(x\in \Gr^0\sqcup \Bisp\) and the
  relations above.
\end{theorem}

If~\(v\) does not belong to the range of \(\rg\colon \Bisp\to Y\),
then the relation~\eqref{eq:CP-covariance_self-similar} says that
\(\Rep_v=0\).  This is why it is reasonable to restrict attention to
the case when~\(\Bisp\) is regular on~\(\Reg\).

\begin{remark}
  Recall that \(F\subseteq \Bisp\) is a fundamental domain.  The
  products \(\Rep_x \Rep_g = \Rep_{x g}\) for all \(x\in F\),
  \(g\in\Gr\) with \(\s(x) = \rg(g)\) exhaust the generators for
  elements of~\(\Bisp\).  So the elements~\(\Rep_x\) for
  \(x\in \Gr^0\sqcup F\) suffice to generate the Toeplitz and
  Cuntz-Pimsner algebras of \(\Cst(\Bisp)\).  When we use only these
  generators, there is no longer any relation for \(\Rep_x \Rep_g\)
  for \(x\in F\) and \(g\in \Gr\), and the product relation for
  elements of \(\Gr\) and~\(\Bisp\) becomes
  \(\Rep_g \Rep_x = \Rep_{g\circ x} \Rep_{g|_x}\) for \(g\in \Gr\),
  \(x\in F\) with \(\s(g) = \rg(x)\), with the notation
  in~\eqref{eq:action_X_self-similar}.  If \(x,y\in F\), then
  \(\Rep_x^*\Rep_y\) is only nonzero if \(x=y\) because \(\Qu|_F\) is
  injective.  So \(\Rep_x^*\Rep_y = \delta_{x,y} \Rep_{\s(x)}\) for
  \(x,y \in F\), and this relation implies and therefore completely
  replaces the more general relation
  \(\Rep_x^*\Rep_y = \Rep_{\braket{x}{y}}\) for \(x,y \in \Bisp\).
\end{remark}

Now let~\(\Reg\) be the set of regular vertices for the graph
\(\rg,\s\colon F\rightrightarrows \Gr^0\) in the usual sense.  Then
the relations for the generators~\(\Rep_x\) for \(x\in \Gr^0\sqcup F\)
in the Cuntz--Pimsner algebra~\(\mathcal{O}\) of~\(\Cst(\Bisp)\)
relative to \(\Cst(\Gr_\Reg)\) are exactly the defining relations of a
Cuntz--Krieger family for the graph
\(\rg,\s\colon F\rightrightarrows \Gr^0\).  Thus these generators
generate a \Star{}homomorphism from the graph \(\Cst\)\nb-algebra of
this graph to~\(\mathcal{O}\).  The generators~\(\Rep_x\)
for \(x\in \Gr\supseteq \Gr^0\) satisfy the same relations as in the
groupoid \(\Cst\)\nb-algebra \(\Cst(\Gr)\).  Finally,
there are some extra commutation relations for \(\Rep_x\Rep_y\) for
\(x\in\Gr\) and \(y\in F\).

If \(\Reg \subseteq \Gr^0\) is a different subset, then the statements
above remain true if we replace the graph \(\Cst\)\nb-algebra by the
relative graph \(\Cst\)\nb-algebra, where the relation
\(p_v = \sum_{x\in \rg^{-1}(v) \cap F} \Rep_x\Rep_x^*\) is imposed for
\(v\in\Reg\).

Our main theorem, Theorem~\ref{the:main}, says that~\(\mathcal{O}\) is
the groupoid \(\Cst\)\nb-algebra of the groupoid
\(\Mod = \Omega(\Reg)\rtimes \IS(\Gr,\Bisp)\).  We have also
described~\(\Mod\) through a universal property, which specifies how
it acts on locally compact Hausdorff spaces.  Such an action on a
space~\(Y\) consists of an action of~\(\Gr\) and a map
\(\Bisp\times_{\s,\Gr^0,\rg} Y \to Y\) with certain properties.
Since~\(\Gr^0\) is discrete, the anchor map \(\rg\colon Y\to \Gr^0\)
is equivalent to a disjoint union decomposition
\(Y = \bigsqcup_{x\in\Gr^0} Y_x\), namely,
\(Y_x \defeq \rg_Y^{-1}(x) \subseteq Y\).  Then \(g\in\Gr\) acts
on~\(Y\) by a homeomorphism
\(\vartheta_g\colon Y_{\s(g)} \to Y_{\rg(g)}\), \(y\mapsto g\cdot y\).
If \(x\in \Bisp\), then left multiplication by~\(x\) must be a
homeomorphism~\(\vartheta_x\) from~\(Y_{\s(x)}\) onto a clopen subset
of~\(Y_{\rg(x)}\).  In fact, each singleton in \(\Gr\) and~\(\Bisp\)
is a slice, and \(\vartheta_x = \vartheta(\{x\})\) is the partial
homeomorphism associated to a singleton slice.  The product of two
singleton slices \(\{x\}\), \(\{y\}\) is either the singleton
\(\{x \cdot y\}\) if \(\s(x) = \rg(y)\) or empty otherwise, and the
empty slice acts by the empty map on~\(Y\).  Similarly,
\(\braket{\{x\}}{\{y\}} = \{\braket{x}{y}\}\) if \(x,y\in \Bisp\)
satisfy \(\Qu(x) = \Qu(y)\), so that \(\braket{x}{y}\) is defined and
\(\braket{\{x\}}{\{y\}}\) is empty otherwise.  With these
clarifications, our partial homeomorphisms
\(\vartheta_x = \vartheta(\{x\})\) for \(x\in \Gr\sqcup \Bisp\) now
satisfy the relations \ref{en:action_from_theta1}
and~\ref{en:action_from_theta2}.  In addition, it is clear that the
graph of \(\vartheta(\Slice)\) for a slice \(\Slice\) is simply the
disjoint union of the graphs of~\(\vartheta_x\) for all
\(x\in\Slice\).  Thus the partial homeomorphisms~\(\vartheta_x\) for
\(x\in\Gr\sqcup \Bisp\) determine the action of all slices on~\(Y\).
In fact, we may replace~\(\IS\) in our description of~\(\Mod\) by the
inverse semigroup with zero element that is generated by the singleton
slice generators \(\Theta(x)\) for \(x\in \Gr\sqcup \Bisp\) with the
analogues of the relations \ref{en:action_from_theta1}
and~\ref{en:action_from_theta2}, where the right hand side is
understood to be the zero element if the relevant slice is empty.

For \(n\in\N\), let
\[
  F_n \defeq \setgiven{ (x_1,\dotsc,x_n) \in F^n}{\s(x_j) =
    \rg(x_{j+1}) \text{ for }j=1,\dotsc,n-1}.
\]
This is in bijection with a fundamental domain for the composite
groupoid correspondence \(\Bisp_n = \Bisp^{\circ n}\).  That is,
\(F_n \cong \Bisp_n/\Gr\).  So
\(\Omega_{[0,n]} = \bigsqcup_{k=0}^n F_k\) as a set.  If \(F\)
and~\(\Gr^0\) are finite, then it follows that all the sets~\(F_k\)
are finite and so the topology on~\(\Omega_{[0,n]}\) is still
discrete.  Then \(\Omega_{[0,\infty)}\) carries the projective limit
topology of these finite compact spaces.  In general,
\(\Omega_{[0,n]}\) carries the locally compact Hausdorff topology
induced by the commutative \(\Cst\)\nb-algebra~\(A_{[0,n]}\) in
Definition~\ref{def:Amn}.  We also get this topology by iterating the
fibrewise one-point compactification construction several times.  We
refrain from giving more details.  The space~\(\Omega(\Reg)\) is also
the object space of the groupoid that underlies the graph
\(\Cst\)\nb-algebra of \(\rg,\s\colon F\rightrightarrows \Gr^0\)
inside~\(\mathcal{O}\).

\begin{remark}
  The definition of a self-similar groupoid action
  in~\cite{Laca-Raeburn-Ramagge-Whittaker:Equilibrium_self-similar_groupoid}
  imposes some extra assumptions.  Namely, the vertex and edge sets of
  the graph, \(\Gr^0\) and~\(F\), are assumed to be finite, and the
  induced action of~\(\Gr\) on the path space of the graph~\(F\) is
  assumed to be faithful.  In this case, the \(\Cst\)\nb-algebra of
  the self-similar groupoid action that is studied
  in~\cite{Laca-Raeburn-Ramagge-Whittaker:Equilibrium_self-similar_groupoid}
  is the absolute Cuntz--Pimsner algebra of \(\Cst(\Bisp)\), that is,
  for \(\Reg = \Gr^0\) (compare
  \cite{Laca-Raeburn-Ramagge-Whittaker:Equilibrium_self-similar_groupoid}*{Proposition~4.7}).
\end{remark}

Summing up, our construction for a discrete groupoid~\(\Gr\) gives the
\(\Cst\)\nb-algebras of self-similar groupoid actions
of~\cite{Laca-Raeburn-Ramagge-Whittaker:Equilibrium_self-similar_groupoid},
except that we drop some assumptions such as the graph being finite or
the action on the path space being faithful.  Thus we propose that our
construction for a general \'etale locally compact groupoid is a good
definition of a self-similar action of an \'etale locally compact
groupoid.  Of course, topological graphs are a special case of our
theory.  To get Katsura's topological graph \(\Cst\)\nb-algebras
of~\cite{Katsura:class_I}, we must let~\(\Reg\) be the largest subset
of the vertex set where the range map is surjective and proper.  For a
more general groupoid correspondence~\(\Bisp\), it is not clear
whether Katsura's ideal for the corresponding
\(\Cst\)\nb-correspondence \(\Cst(\Bisp)\) is of the form
\(\Cst(\Gr_{\Bisp})\) for an open invariant subset
\(\Bisp\subseteq \Gr^0\).  We seem to need this to be the case in
order for the resulting Cuntz--Pimsner algebra to be a groupoid
\(\Cst\)\nb-algebra in a natural way.

Finally, we specialise to the case of a groupoid correspondence
\(\Bisp\colon \Gr \leftarrow \Gr\) for a discrete group~\(\Gr\).  By
\cite{Antunes-Ko-Meyer:Groupoid_correspondences}*{Example~4.2}, this
is the same as a discrete set~\(\Bisp\) with commuting actions
of~\(\Gr\) on the left and right, such that the right action is free
and proper.  The only two options for \(\Reg \subseteq \Gr^0\) are the
empty set and the one-point set~\(\Gr^0\) itself.  If
\(\Reg = \Gr^0\), then we need~\(\Bisp\) to be proper, which amounts
to \(\Bisp/\Gr\) being finite.  If we add the condition that the
\(\Gr\)\nb-action on the set \(\varprojlim \Bisp_n/\Gr\) should be
free, then this is the same as a self-similar group.  If
\(\Reg=\emptyset\), then any groupoid correspondence~\(\Bisp\) is
allowed, that is, we may also allow \(\Bisp/\Gr\) to be infinite.
This goes beyond the scope of the theory of self-similar groups.

\end{document}